\documentclass[10pt]{article}

\usepackage{amssymb,amsmath,amsthm}
\usepackage{amsfonts}
\usepackage{graphicx}
\usepackage[usenames]{color}
\usepackage{caption}
\usepackage{bm,mathrsfs}
\usepackage{url}
\usepackage{subfigure}
\usepackage{algorithm,algorithmic}
\usepackage{dsfont,verbatim}
\usepackage{epsfig}
\usepackage[margin=1.2in]{geometry}
\usepackage{enumerate}

\renewcommand{\footnotemark}{}

\usepackage{hyperref}
\hypersetup{
	colorlinks=true,   
	linkcolor=blue,    
	citecolor=blue,    
	filecolor=magenta, 
	urlcolor=cyan      
}

\usepackage{indentfirst}

\def\div{\mbox{div }}

\def\w{\widetilde}

\def\u{\mathbf{u}}

\def\v{\textbf{v}}
\def\w{\textbf{w}}
\def\f{\mathbf{f}}

\def\V{\textbf{V}}
\def\H{\textbf{H}}
\def\L{\textbf{L}}
\def\W{\textbf{W}}

\def\e{\textbf{e}}

\def\P{\textbf{P}}

\def\I{\textbf{I}}
\def\i{\textbf{i}}

\graphicspath{{./pics/}}
\allowdisplaybreaks
\newtheorem{theorem}{Theorem}[section]
\newtheorem{lemma}{Lemma}[section]

\newtheorem{remark}{Remark}[section]

\newtheorem{Assumption}{Assumption}[section]
\newtheorem{Algorithm}{Algorithm}[section]

\numberwithin{equation}{section}

\title{\bf 
	Unconditionally Optimal error Estimate of a  linearized Second-order Fully Discrete Finite Element Method for the bioconvection flows with concentration dependent viscosity
	}

\author{
	Chenyang Li\textsuperscript{1,*}
	\thanks{\textsuperscript{1,*}School of Mathematical Sciences, East China Normal University, Shanghai 200241, China \texttt{52275500026@stu.ecnu.edu.cn}(Corresponding author)}
	\and 
	Yuze Lu\textsuperscript{1}
	\thanks{\textsuperscript{1}School of Mathematical Sciences, East China Normal University, Shanghai 200241, China \texttt{51265500066@stu.ecnu.edu.cn}}
		\and  
		Haibiao Zheng\textsuperscript{2,*}
	\thanks{\textsuperscript{2,*}School of Mathematical Sciences, Ministry of Education Key Laboratory of Mathematics and Engineering Applications, Shanghai Key Laboratory of PMMP,  East China Normal University, Shanghai 200241, China. \texttt{hbzheng@math.ecnu.edu.cn}(Corresponding author) }
}
\date{}
\begin{document}

\maketitle

\begin{abstract}
In this paper, the coupled and decoupled BDF2 finite element discrete schemes are obtained for the time-dependent bioconvection flows problem with concentration dependent viscosity, which consisting of the Navier-Stokes equation coupled with a linear convection-diﬀusion equation modeling the concentration of microorganisms in a culture fluid. The unconditionally optimal error estimate for the velocity and concentration in $L^2$-norm are proved by using finite element approximations in space and finite diﬀerences in time. Finally, the numerical results for different viscosity are showed to support the theoretical analysis. 
\\[5pt]
\end{abstract}

\pagestyle{myheadings}
\thispagestyle{plain}

\section{Introduction}
The bioconvection model is coupled by the Navier-Stokes
 type equations describe the flow of the incompressible viscous culture fluid and the advection-diffusion equations describe the transport of micro-organisms:
 \begin{align}
 \frac{\partial \u}{\partial t} - \div ( \nu(c) D(\u)) + \u \cdot \nabla \u + \nabla p = - g (1+\gamma c) \textbf{i}_2 + \textbf{f}, \quad  x \in \Omega, \, t>0,\label{biobdf-32}\\
 \nabla \cdot \u = 0, \quad  x \in \Omega, \, t>0, \\
 \frac{\partial c}{\partial t} -  \theta \Delta c + \u \cdot \nabla c + U \frac{\partial c}{\partial x_2}=0, \quad  x \in \Omega, \, t>0.\label{biobdf-33}
 \end{align}
 
The unknowns are the concentration $c$, the velocity $\u$ and the pressure $p$, which is assumed to have zero mean for uniqueness. Here $\Omega$ is a bounded domain with smooth boundary $\partial \Omega$, $x_2$ is the second component of independent variable $x$, $\theta >0$ is the diffusivity, the kinematic viscosity $\nu(c)$ depends on the concentration of the micro-organisms \cite{batchelor1972,brady1993}, $D(\u) = \frac{1}{2} ( \nabla \u + \nabla \u ^{T})$ is the stress tensor. $\textbf{f}$ is the external force and $g$ is the acceleration of the gravity, $U>0$ is the mean velocity of upwind swimming of the micro-organisms.

The parameter $\gamma>0$ represents the relative difference of the density $\rho_0$ of the micro-organism form, the density $\rho_m$ of the culture flow is determined by $\gamma = \frac{\rho_0}{\rho_m}-1$ and the term  $-g(1+\gamma c )\textbf{i}_2$ captures the effects of gravity on the organisms at the upper surface, the term $U \frac{\partial c}{\partial x_2}$ representing the effects of the average upward swimming. 

In an ideal Newtonian fluid, the viscosity is assumed to be a constant, but is not suitable for the real-life suspensions,  where the viscosity depends on the concentration of the micro-organisms and the exponent expressions of $\nu(\cdot)$ as a function of concentration $c$ are demonstrated in  \cite{batchelor1972,brady1993,mooney1951,einstein1906,krieger1959}. We consider homogenous dirichlet boundary condition $\u=0$ and the non-flux Robin condition for the concentration
 $$\theta \frac{\partial c}{\partial n} - c U n_2=0,$$
 where $\textbf{n}=(n_1,n_2,n_3)$ is the unit outward normal vector on the boundary. In addition, we assume that the total mass of the micro-organisms in the container remains constant, that is
  \begin{align}\label{biobdf-1}
  \frac{1}{|\Omega|} \int_{\Omega} c(x) dx = \alpha,
  \end{align}
where  $|\Omega|$ is expressed as the measure of $\Omega$ and the constant $\alpha$ represents the average concentration in the container.

Finally the bioconvection model is described  on a bounded time interval $I=(0,T]$ and is prescribed with initial conditions on $\u, p,c$ as follow. 
 \begin{align}\label{biobdf-45}
 	\left\{\begin{aligned}
	\frac{\partial \u}{\partial t} - \div(\nu(c) D (\u)) + \u \cdot \nabla \u + \nabla p = - g (1+\gamma c) \textbf{i}_2 + \textbf{f}, \quad  &x \in \Omega, \, t>0,\\
	\nabla \cdot \u = 0, \quad  &x \in \Omega, \, t>0, \\
	\frac{\partial c}{\partial t} -\theta \Delta c + \u \cdot \nabla c + U \frac{\partial c}{\partial x_2}=0, \quad  &x \in \Omega, \, t>0,\\
	\frac{1}{|\Omega|} \int_{\Omega} c(x) dx = \alpha, \quad &x \in \Omega, \, t>0.\\
	\u=0, \,\, \theta \frac{\partial c}{\partial n} - c U n_2=0, \quad &x \in \partial \Omega, \, t>0.
\end{aligned}\right.
\end{align}

There have been several studies examining the well-posedness of solutions to both time-dependent and steady bio-convection flow problems. The existence of a solution of system (\ref{biobdf-45}) with constant viscosity was proved by the semi-group method in \cite{kan1992}, the corresponding numerical study can be founded in \cite{harashima1988}. The authors proved  the existence and uniqueness of a periodic solution of bioconvection with concentration dependent density in  \cite{climent2013}. For more general boundary conditions \cite{cao2020}, the authors proved the existence and uniqueness of a weak solution for the system (\ref{biobdf-45}) under quite loose (and realistic) assumptions on the viscosity, notably the boundedness condition $0 < \nu_{*} \leq \nu(c) \leq \nu^*$.

There are some contributions on the numerical analysis and simulations for the bioconvection flows problem. Numerical approximations based on the finite element method are constructed by using the stable finite element pair for the velocity and pressure, and error estimates of the velocity and concenteation in $H^1$- norm are obtained in \cite{cao2014}. 
The authors studied some convergence rates of the spectral Galerkin approximations for unsteady bioconvection flow in \cite{aguiar2017}.
The fully discrete finite element scheme for the time-dependent problem was proposed in \cite{cao2020}, the error estimate for the proposed  backward Euler finite element approximations scheme are sub-optimal in $L^2$-norm.
The authors rewrite the steady bioconvection flow problem in term of a first-order system based on the introduction of the shear-stress, the vorticity, and the pseudo-stress tensors in the fluid equations along with an auxiliary vector in the concentration equation,  the existence and uniqueness results are proved by the Lax-Milgram theorem and Banach fixed-point theorems, an augmented mixed finite element method are studied in \cite{colmenares2021}.

To the best of our knowledge, the optimal error analysis for the unsteady model, particularly for high-order schemes, has not yet been thoroughly investigated. The objective of this paper is to introduce the BDF2 fully discrete finite element scheme for solving the time-dependent bioconvection flows problem with concentration dependent viscosity. The fully coupled and decoupled BDF2 finite element method are proposed, and the unconditionally optimal convergent rate for the velocity and concentration in $L^2$-norm and $H^1$-norm are established in the sense that no any restrictions are imposed on the relationship
between the time-step size and the spatial size.

The rest of this paper is organized as follows. In the next section, we introduce the notations, some function spaces and numerical scheme. The decoupled and coupled BDF2 finite element methods and their stability results are presented in Section 3. The optimal error estimates for the proposed methods are derived in Section4. In final section, the numerical results for different viscosity are showed to support the theoretical analysis.
The symbol $C$ is used throughout to denote a generic positive constant, whose value may differ at different occurrences, but which does not depend on the discretization parameters 
$\tau$ (time step size) and 
$h$(mesh size).
\section{Preliminaries}
For the mathematical setting of this model, we introduce
some function spaces and their associated norms.
For $k\in\mathbb N^+, 1\leq p\leq +\infty$, let $W^{k,p}(\Omega)$ denote the standard Sobolev space. The norm in $W^{k,p}(\Omega)$ is denoted by $\|\cdot\|_{W^{k,p}}$.
We define $W_0^{k,p}(\Omega)$ to be the subspace of $W^{k,p}(\Omega)$ of functions with zero trace on $\partial\Omega$.
When $p=2$, we simply use $H^k(\Omega)$ to denote $W^{k,2}(\Omega)$. 
The boldface Sobolev spaces $\H^k(\Omega), \W^{k,p}(\Omega)$ and $\L^p(\Omega)$ are used to denote the vector Sobolev spaces
$H^k(\Omega)^2, W^{k,p}(\Omega)^2$ and $L^p(\Omega)^2$ respectively.
Let $X$ be a Banach space, for some $T>0$, denote Bochner spaces $L^p(0,T;X),1\leq p<+\infty$ as  the
spaces of measurable functions from the interval $[0,T]$ into $X$
such that
$$\int_0^T\| \u(t)\|_X^p d  t<+\infty,$$
if $p=+\infty$, the functions in $L^\infty(0,T;X)$ are required to satisfy
$$\sup\limits_{t\in [0,T]}\| \u(t)\|_X<+\infty.$$

For simplicity, we denote the inner products of both
$L^2(\Omega)$ and $\textbf{L}^2(\Omega)$
by $(\cdot,\cdot)$, and use $\langle \cdot,\cdot \rangle$ to denote the dual product of $H^{-1}(\Omega) \times H^1_0(\Omega)$. namely,
\begin{align*}
	\begin{split}
		&(u,v)=\int_\varOmega u(x)v(x) d x \quad \forall  \, u,v\in L^2(\Omega),\\
		&({\bf u},{\bf v})=\int_\Omega {\bf u}(x)\cdot {\bf v}(x) dx \quad \forall \, {\bf u},{\bf v}\in {\bf L}^2(\Omega) .\\
	\end{split}
\end{align*}

Introducing the following function spaces and notions:
\begin{align}
\textbf{V}&=H^1_0 = \{  \v \in H^1(\Omega) ; \,\, \v |_{\partial \Omega} =0 \},\\
  \textbf{V}_0 &= \{ \v \in \textbf{V}; \,  \nabla \cdot \v=0 \,\, \text{in} \,  \Omega\},\\
 M &= L^2_0 (\Omega) = \{ q \in L^2(\Omega); \,\, \int_{\Omega} q dx =0    \},\\
\tilde{H}&=H^1(\Omega)  \cap L^2_0(\Omega). 
\end{align}

The norm in $\V$ is given by
$$\| \v\|_{\V} = (\int_{\Omega} | \nabla \v |^2 dx )^{\frac{1}{2}} \quad \forall \, \v \in \V.$$

According to the Agmon inequality and the interpolation inequality \cite{adams1978}, one has
\begin{align}
	&\|\u\|_{L^{\infty}} \leq C \| \u \|^\frac{1}{2}_{L^2} \| \u \|^{\frac{1}{2}} _{H^2}, \quad \| \u \|_{L^4 } \leq C \| \u \|^{\frac{1}{2}}_{L^2}| \u \|^{\frac{1}{2}}_{H^1}, \quad \forall \u \in \H^2(\Omega)\cap \V,   \label{biobdf-15}\\
		&\|c\|_{L^{\infty}} \leq C \| c \|^\frac{1}{2}_{L^2} \| c \|^{\frac{1}{2}} _{H^2}, \quad \| c \|_{L^4 } \leq C \| c \|^{\frac{1}{2}}_{L^2}| c \|^{\frac{1}{2}}_{H^1}, \quad \forall c \in H^2(\Omega)\cap H^1_0 (\Omega).\label{biobdf-20}
 \end{align}

The following   Sobolev imbedding inequalities in 2D will be used in the following \cite{adams1978} 
\begin{align}\label{biobdf-52}
	\begin{split}
	&W^{2,4}(\Omega)  \hookrightarrow W^{1,\infty}(\Omega),\\ 
	&H^2(\Omega) \hookrightarrow W^{1,q}(\Omega), \quad 2 \leq q < \infty,\\
	&H^2(\Omega) \hookrightarrow L^{\infty}(\Omega).
		\end{split}
\end{align}

Noticing that (\ref{biobdf-1}) is equivalent to require that $c-\alpha \in L^2_0(\Omega)$, we can adopt the same method in \cite{cao2020,antwolevel2023,antwolevel2024}, introducing an auxiliary concentration $c_\alpha=c-\alpha$ and $\textbf{f}_\alpha=\textbf{f}-g\gamma \alpha \textbf{i}_2$, then we can rewrite the original form (\ref{biobdf-45}) to the following system (still denote $c=c_\alpha, \textbf{f} = \textbf{f}_\alpha$ to simplify notation).
\begin{align}\label{biobdf-3}
	\left\{\begin{aligned}
		\frac{\partial \u}{\partial t} - \div ( \nu(c+\alpha) \nabla \u) + \u \cdot \nabla \u + \nabla p = - g (1+\gamma c) \textbf{i}_2 + \textbf{f}, \quad  &x \in \Omega, \, t>0,\\
		\nabla \cdot \u = 0, \quad  &x \in \Omega, \, t>0, \\
		\frac{\partial c}{\partial t} - \theta  \Delta c + \u \cdot \nabla c + U \frac{\partial (c+\alpha)}{\partial x_2}=0, \quad  &x \in \Omega, \, t>0,\\
		\frac{1}{|\Omega|} \int_{\Omega} c(x) dx = 0, \quad &x \in \Omega, \, t>0.\\
		\u=0, \,\, \theta \frac{\partial c}{\partial n} - U(c+\alpha)n_2=0, \quad &x \in \partial \Omega, \, t>0.
	\end{aligned}\right.
\end{align}

The transport terms have the diﬃculty that the corresponding discrete forms do not preserve the alternance property as in the continuous case, and we need to introduce some
skew-symmetric trilinear forms to make the stability analysis and error estimate  by
\begin{align}
	\begin{split}
	B(\textbf{u},\textbf{v},\textbf{w}) &= \int_\Omega (\u \cdot \nabla \textbf{v}) \cdot\textbf{w} dx + \frac{1}{2} \int_{\Omega} (\nabla \cdot \u) \v \cdot \w dx \\
	&=   \frac{1}{2}\int_\Omega (\u \cdot \nabla \textbf{v}) \cdot\textbf{w} dx -  \frac{1}{2} \int_{\Omega} (\u \cdot \nabla) \w \cdot\v dx  \quad \forall ~\u,\v,\w\in\textbf{V},\\
	b(\u,c,r) &=  \int_\Omega (\u \cdot \nabla c)r dx+ \frac{1}{2} \int_{\Omega} (\nabla \cdot \u) crdx, \\
	&=\frac{1}{2}\int_\Omega (\u \cdot \nabla c) r dx -  \frac{1}{2} \int_{\Omega} (\u \cdot \nabla r) c dx \quad \forall ~\u \in \V, \,  \forall c,r \in \tilde{H}.
		\end{split}
\end{align}
which has the following properties \cite{heyinnian2005,heyinnian2022}
\begin{align}
		& B(\textbf{u},\textbf{v},\textbf{v})=0, \quad b(\u,r,r)=0,\label{biobdf-2}\\
		&B(\textbf{u},\textbf{v},\textbf{w}) = - b(\textbf{u},\textbf{w},\textbf{v}),\quad b(\textbf{u},c,r) = - b(\textbf{u},r,c),\label{biobdf-17}\\
		&B(\textbf{u},\textbf{v},\textbf{w}) \leq C \| \nabla \textbf{u} \|_{L^2} \| \nabla \textbf{v} \| _{L^2} \| \nabla \textbf{w} \|_{L^2},\quad b(\textbf{u},c,r) \leq C \| \nabla \textbf{u} \|_{L^2} \| \nabla c \| _{L^2} \| \nabla r \|_{L^2},\\
		&B(\textbf{u},\textbf{v},\textbf{w})  \leq C \| \textbf{u} \| ^{\frac{1}{2}}_{L^2} \| \nabla \textbf{u}\| ^{\frac{1}{2}}_{L^2} \| \nabla \textbf{v} \| _{L^2} \| \nabla \textbf{w} \|_{L^2}, \quad  b(\textbf{u},c,r)  \leq C \| \textbf{u} \| ^{\frac{1}{2}}_{L^2} \| \nabla \textbf{u}\| ^{\frac{1}{2}}_{L^2} \| \nabla c \| _{L^2} \| \nabla r \|_{L^2}.
\end{align}

If $\nabla \cdot \u=0$, there holds $B(\textbf{u},\textbf{v},\textbf{w}) = (\textbf{u} \cdot \nabla \textbf{v},\textbf{w})= \int_\Omega (\u \cdot \nabla \textbf{v}) \cdot\textbf{w} dx $ and $b(\u,c,r) = (\textbf{u} \cdot \nabla c,r) = \int_\Omega (\u \cdot \nabla c)r dx$.

Based on the above notation, for given $\f \in \L^2(\Omega)$, the weak formulation of the bio-convection model (\ref{biobdf-3}) read as follows. We find $(\u,p,c) \in \textbf{V} \times M\times \tilde{H}$ such that  
\begin{align}\label{biobdf-4}
		\left\{\begin{aligned}
	&(\frac{\partial \u}{\partial t} , \v) + (\nu(c+\alpha) \nabla \u,\nabla \v) + B(\u,\u,\v) - ( \div \v,p) + (\div \u,q) = (\f,\v)- g((1+\gamma c)\i_2,\v),\\
	&(\frac{\partial c}{\partial t},r) +\theta (\nabla c, \nabla r) +b(\u,c,r) -U(c,\frac{\partial r}{\partial x_2}) = U\alpha (1, \frac{\partial r}{\partial x_2}),\\
	&\u(0)=\u_0, \quad c(0)=c_0.
		\end{aligned}\right.
\end{align}
for any $(\v,q,r)\in \V \times M\times \tilde{H}$. 

It is equivalent to solve the following problem \cite{girault2012, temam2024}. We find $(\u,c) \in \textbf{V} \times \tilde{H}$ such that  
\begin{align}\label{biobdf-6}
	\left\{\begin{aligned}
		&(\frac{\partial \u}{\partial t} , \v) + (\nu(c+\alpha) \nabla \u,\nabla \v) + B(\u,\u,\v) = (\f,\v)- g((1+\gamma c)\i_2,\v),\\
		&(\frac{\partial c}{\partial t},r) + \theta ( \nabla c, \nabla r) +b(\u,c,r) -U(c,\frac{\partial r}{\partial x_2}) = U\alpha (1, \frac{\partial r}{\partial x_2})\\
		&\u(0)=\u_0, \quad c(0)=c_0.
	\end{aligned}\right.
\end{align}
for any $(\v,r)\in \V \times \tilde{H}$. 
The existence and uniqueness of the solution had been proved in \cite{cao2020}.

To prove the unconditional stability and error estimate of the following temporal and spatial discrete schemes, we recall the following discrete Gronwall inequality established in \cite{evance2022,heywood1990}.
\begin{lemma} (Discirete Gronwall's inequality ) Let $a_k , b_k$ and $y_k$ be the nonnegative numbers such that \label{biobdf-11}
	\begin{align}\label{growninequality-discrete}
		a_n+ \tau \sum \limits^n \limits_{k=0}  b_k\leq \tau \sum \limits^n \limits_{k=0} \gamma _k a_k + B \quad \text{for} \,\, n \geq 1,
	\end{align}
	Suppose $\tau \gamma _k \leq 1$ and set $\sigma_k = (1-\tau \gamma_k) ^{-1}$. Then there holds
	\begin{align}\label{grown2}
		a_n + \tau \sum \limits^n \limits_{k=0} b_k \leq exp(\tau \sum \limits^n \limits_{k=0} \gamma_k \sigma_k) B \quad  \text{for} \,\, n\geq 1.
	\end{align}
	\begin{remark}
		If the sum on the right-hand side of (\ref{growninequality-discrete}) extends only up to $n-1$, then the estimate (\ref{grown2}) still holds for all $k \geq 1$ with $\sigma_k=1$.
	\end{remark}
	
\end{lemma}

Next, we make the following Assumptions.
\begin{Assumption} \label{biobdf-9}
The kinematic viscosity $\nu(c)$ is smooth function and satisfies
\begin{align}\label{biobdf-5}
	k^{-1} \leq \nu(c)  \leq k, \quad  |\nu^{'}(c)|,  \leq \beta, \quad \forall \, c \in \mathbb{R}. 
\end{align}
\end{Assumption}

\begin{Assumption}
	The initial values and the force satisfy
	\begin{align}\label{biobdf-36}
		\u_0 \in \V_0 \cap \H^2(\Omega), \quad c_0 \in \tilde{H} \cap \H^2(\Omega), \quad \f \in \L^{\infty} (0,T;\L^2).
	\end{align}
\end{Assumption}

\begin{Assumption}\label{biobdf-46}
Suppose that solutions $(u, c)$ satisfy the following regularity
\begin{align}\label{biobdf-47}
	\begin{split}
	&(\mathbf{u}, c) \in L^\infty(0,T; W^{2,4}(\Omega)) \\
	&(\mathbf{u}_t, c_t) \in L^\infty(0,T; L^2(\Omega)) \cap L^2(0,T; H^1(\Omega)), \\
	&c_{tt} \in L^2(0,T;H^{1}), \quad \u_{tt} \in L^2(0,T;L^2),\\
	&(\mathbf{u}_{ttt}, c_{ttt}) \in L^2(0,T; H^{-1}(\Omega)), 
	\end{split}
\end{align}
\end{Assumption}

Let $N$ be a positive integer and $0 = t_0 < t_1< \cdots <t_N=T$ be a uniform partition of $[0,T]$, with $\tau = \Delta t= T/N $. For the sake of convenience, we introduce the following notation
\begin{align*}
	&D_\tau v^{n+1} = \frac{3 v^{n+1} - 4 v^n +v^{n-1}}{ 2\tau}, \quad d_\tau v^{n+1} = \frac{v^{n+1} -v^n}{\tau},\\
	&\hat{v}^n = 2 v^n-v^{n-1}.
\end{align*}
and the telescope formula \cite{liujie2013}
\begin{align}
		(D_{\tau} v^{n+1} , v^{n+1})  =& \frac{1}{4 \tau} ( \| v^{n+1} \| ^2_{L^2} - \| v^n\|^2_{L^2}  + \| \hat{v} ^{n+1} \|^2_{L^2} - \| \hat{v}^n \|^2_{L^2}   )  \nonumber \\
		&+ \frac{1}{4\tau} \| v^{n+1} -2v^n+v^{n-1}   \|^2_{L^2} \quad  \forall  \, 1 \leq n \leq N-1.
\end{align}

The exact solutions $( \u,p,c)$  at $t=t_{n+1}$ in the continuous system (\ref{biobdf-3}) will be denoted by  $(\u^{n+1},p^{n+1},c^{n+1})$ for all $0\leq n \leq N-1 $.  Further the exact solutions $(\u,p,c)$ at $t=t_{n+1}$ satisfy the following weak formulation

\begin{align}
	&(D_\tau \u^{n+1},\v) + (( \nu(\hat{c}^n + \alpha ) \nabla \u^{n+1}  ) , \nabla \v)+ B( \hat{\u}^n,  \u^{n+1},\v) +(\nabla p^{n+1}, q)  \nonumber\\
	= &(\f^{n+1},\v)-g(( 1+ \gamma \hat{c}^n)\textbf{i}_2,\v)+<R_\u^{n+1},\v> \quad \forall (\v,q) \in (\V \times M), \label{biobdf-48}\\
	&(D_\tau c^{n+1},r) +  \theta( \nabla c^{n+1} ,\nabla r)+ b( \hat{\u}^n,  c^{n+1},r) - U( \hat{c}^n  , \frac{\partial r}{\partial x_2})\nonumber \\
	=& <R^{n+1}_c,r> + U(\alpha , \frac{\partial r}{\partial x_2}) \quad \forall r \in \tilde{H}\label{biobdf-49}
\end{align}

where
\begin{align*}
	R^{n+1}_\u = D_\tau \u^{n+1} - \u_t(t_{n+1})+ \div ( \nu(  \hat{c}^n+\alpha) \nabla \u^{n+1})-\div ( \nu(  c^{n+1}+\alpha) \nabla \u^{n+1})\\
	+ \hat{\u}^{n}\cdot \nabla \u^{n+1}-\u^{n+1} \cdot \nabla \u^{n+1} + g(1+\gamma \hat{c}^{n})\textbf{i}_2-g(1+\gamma c^{n+1})\textbf{i}_2,\\
	R_c^{n+1} = D_{\tau}c^{n+1} -c_t (t_{n+1}) 
	+ \hat{\u}^{n} \cdot \nabla c^{n+1}-\u^{n+1} \cdot \nabla c^{n+1} + U \frac{\partial (\hat{c}^{n} + \alpha )}{\partial x_2}-U \frac{\partial (c^{n+1} + \alpha )}{\partial x_2}.
\end{align*}

By using \cite{rongy2020}
\begin{align}
	\| \frac{3 \u^{n+1} - 4 \u^{n} + \u^{n-1}}{2 \tau} - \u_t(t_{n+1}) \| ^2_{L^2} \leq C \tau^3 \int^{t_{n+1}}_{t_{n-1}}  \| \u_{ttt}\|^2_{L^2} dt,\\
	\| \u^{n+1}- 2 \u^{n}+\u^{n-1} \|^2_{L^2} \leq C \tau^3 \int^{t_{n+1}}_{t_{n-1}}  \| \u_{tt}\|^2_{L^2} dt,
\end{align}

Thus
\begin{align}
	\begin{split}
		\| R^{n+1}_\u \|_{H^{-1}} =& \underset{\v \in \V}{\sup} \frac{\langle R^{n+1}_\u, \v \rangle }{\|\nabla \v \|_{L^2}}\\
		\leq &C \tau^{\frac{3}{2}} ( \int^{t_{n+1}}_{t_{n-1}}  \| \u_{ttt}\|^2_{H^{-1}} dt)^{\frac{1}{2}} + C \tau^{\frac{3}{2}} \| \u(t_{n+1}) \|_{H^2} ( \int^{t_{n+1}}_{t_{n-1}}  \| \nabla c_{tt}\|^2_{L^2} dt)^{\frac{1}{2}}\\
		&+ C \tau^{\frac{3}{2}} \| \u(t_{n+1})\|_{H^2} ( \int^{t_{n+1}}_{t_{n-1}}  \|  \u_{tt}\|^2_{L^2}  dt)^{\frac{1}{2}}+C \tau^\frac{3}{2}( \int^{t_{n+1}}_{t_{n-1}}  \|  c_{tt}\|^2_{L^2}  dt)^{\frac{1}{2}},
	\end{split}
\end{align}
\begin{align}
	\begin{split}
		\| R^{n+1}_c \|_{H^{-1}} =& \underset{r \in  \tilde{H}}{\sup} \frac{\langle R^{n+1}_c, r \rangle}{\|\nabla r \|_{L^2}}\\
		\leq &C \tau^{\frac{3}{2}} ( \int^{t_{n+1}}_{t_{n-1}}  \| c_{ttt}\|^2_{L^2} dt)^{\frac{1}{2}} 
		+ C \tau^{\frac{3}{2}} \| c(t_{n+1})\|_{H^2} ( \int^{t_{n+1}}_{t_{n-1}}  \|  \u_{tt}\|^2_{L^2}  dt)^{\frac{1}{2}}+C \tau^\frac{3}{2}( \int^{t_{n+1}}_{t_{n-1}}  \|  \nabla c_{tt}\|^2_{L^2}  dt)^{\frac{1}{2}},
	\end{split}
\end{align}
Furthermore, by using the regularity assumption (\ref{biobdf-47}) we have
\begin{align}\label{biobdf-37}
	\tau \sum_{n=1}^{N} (  \| R^{n+1}_\u\|^2_{H^{-1}}  + \| R^{n+1}_c \|^2_{H^{-1}} ) \leq C \tau^4.
\end{align}

Especially, for $n=0$, the continuous system (\ref{biobdf-3}) at $t=t_0$ becomes
\begin{align}
	d_\tau \u^1-\div(\nu(c^0+\alpha) \nabla \u^1)  + \u^0\cdot \nabla \u^1+ \nabla p^1 + g ( 1+ \gamma c^0)\textbf{i}_2 = \f^1+ R^1_\u, \label{biobdf-31} \\
	d_\tau c^1 -\theta \Delta c^1 + \u^0\cdot \nabla c^1 + U \frac{\partial (c^0 + \alpha )}{\partial x_2}= R^{1}_c.\label{biobdf-38}
\end{align}

Further 
\begin{align}
	&(d_\tau \u^1,\v) + (( \nu(c^0 + \alpha ) \nabla \u^1  ) , \nabla \v)+ B( \u^0,  \u^1,\v)   \nonumber\\
	= &(\f^{1},\v)-g(( 1+ \gamma c^0)\textbf{i}_2,\v)+(R_\u^{1},\v), \label{biobdf-50}\\
	&(d_\tau c^1,r) +  \theta( \nabla c^1 ,\nabla r)+ b( \u^0, c^1,r) - U( c^0 , \frac{\partial r}{\partial x_2})\nonumber \\
	=& (R^{1}_c,r) + U(\alpha , \frac{\partial r}{\partial x_2}),\label{biobdf-51}
\end{align}
where 
\begin{align}
	\begin{split}
		R^1_\u =& d_\tau \u^1 - \u_t(t_1) + \u^0 \cdot \nabla \u^1 - \u^1 \cdot \nabla \u^1 - \div(\nu(c^0+\alpha) \nabla \u^1) \\
		+ &\div( \nu(c^1+\alpha) \nabla \u^1)
		+ g(1+\gamma c^0)\textbf{i}_2-  g(1+\gamma c^1)\textbf{i}_2,\\
		R_c^{1} =& d_{\tau}c^1  -c_t (t_{1})+\u^0 \cdot \nabla c^1-\u^1 \cdot \nabla c^1+ U \frac{\partial (c^0+ \alpha )}{\partial x_2}-U \frac{\partial (c^1 + \alpha )}{\partial x_2}.
	\end{split}
\end{align}

By noticing that 
\begin{align}
	\| \frac{\u^1 - \u^0}{\tau}  - \u_t(t_1) \|^2_{L^2}  \leq C \tau ^2 \int^{t_{n+1}}_{t_{n-1}}  \| \u_{tt}\|^2_{L^2} dt,
\end{align}

Applying the assumption (\ref{biobdf-36}) and the similar technique in (\ref{biobdf-37}), we have 
\begin{align}
	\| R^1_\u \|^2_{L^{2}} + \| R^1_c\|^2_{L^{2}}\leq C \tau^2.
\end{align}

\section{The decoupled and coupled BDF2 Finite Element schemes and their stabilities}

In this section, we will present the finite element  discrete method for the continuous system (\ref{biobdf-6})  and derive the optimal convergence rate in space.

Let $\mathcal{T}_h$ be a family of quasi-uniform triangular partition of $\Omega$. The corresponding ordered triangles are denoted by $\mathcal{K}_1,\mathcal{K}_2, \cdots \mathcal{K}_n$, let $h_i$= diam$(\mathcal{K}_i) ~i= 1,2 \cdots,n,$  and $h=\max\{h_1,h_2,\cdots h_n\}.$ For every $\mathcal{K} \in \mathcal{T}_h$, let $P_r(\mathcal{K})$ denote the space of polynomials on $\mathcal{K}$ of degree at most $r$.  We use the mini (P1b-P1) element to approximate the velocity and pressure, and the piecewise linear (P1) element to approximate the concentration, the corresponding finite element spaces are given by \cite{antwolevel2023,scott2008}
   \begin{align*}
   	&\V_h = \left\{ \mathbf{v}_h \in C(\overline{\Omega})^2 \cap \mathbf{V}, \ \mathbf{v}_h|_\mathcal{K} \in \left( P_1(\mathcal{K}) \oplus b(\mathcal{K}) \right)^2, \ \forall \mathcal{K} \in \mathcal{T}_h \right\},\\
   	&M_h = \left\{ q_h \in C(\overline{\Omega}) \cap H^1(\Omega), \ q_h|_\mathcal{K} \in P_1(\mathcal{K}), \ \forall \mathcal{K} \in \mathcal{T}_h, \ \int_{\Omega} q_h \, dx = 0 \right\},\\
   	&\tilde{H}_h = \left\{ r_h \in C(\overline{\Omega}) \cap \widetilde{H}, \ r_h|_\mathcal{K} \in P_1(\mathcal{K}), \ \forall \mathcal{K} \in \mathcal{T}_h \right\},
   \end{align*}
where $b(\mathcal{K})$ is the bubble function on each tetrahedron $\mathcal{K} \in \mathcal{T}_h$.

It's well known that $V_h$ and $M_h$ satisfy the discrete LBB condition \cite{babuvska1973,brezzi1974} for the mini element. i.e. there exists some $\beta >0$ being independent of mesh size $h$ such that
\begin{align*}
	\beta \|q_h \|_{L^2} \leq  \underset{\textbf{v}_h \in \textbf{V}}{\sup}  \frac{(\div \textbf{v}_h, q_h)}{\| \nabla \textbf{v}_h \|_{L^2}}, \quad \forall q_h \in M_h.
\end{align*}

The inverse inequalities will be used frequently
\begin{align}
	\| \u_h \|_{W^{m,q}} \leq C h^{  l-m + n( \frac{1}{q} -\frac{1}{p}  )} \| \u_h\| _{W^{l,p}}, \quad \forall~ \u_h \in V_h, \label{biobdf-16}\\
		\| c_h \|_{W^{m,q}} \leq C h^{  l-m + n( \frac{1}{q} -\frac{1}{p}  )} \| c_h\| _{W^{l,p}}, \quad \forall~ c_h \in \tilde{H}_h,\label{biobdf-21}
\end{align}

Let $\u^0_h = \I_h \u_0$ and $ c^0_h = \Pi_h c_0$, where $\I_h, \Pi$ are the classical interpolation operator onto $\V_h, \tilde{H}_h$,  then there holds \cite{scott2008}: 
\begin{align}
	&\| \u_0 -\u^0_h \|_{L^2} + \| c_0 - c^0_h \|_{L^2} \leq C h^2, \label{biobdf-26}\\
	&	\| \u_0 -\u^0_h \|_{H^1} + \| c_0 - c^0_h \|_{H^1} \leq C h,\label{biobdf-27}\\
	&	\| \u^0_h \|_{L^{\infty}} + \|c^0_h \|_{L^{\infty}} \leq C.
\end{align}

Furthermore, the following interpolation approximation holds
\begin{align}
	\begin{split}
	\| \u - \I_h \u \|_{L^2} + h \| \u- \I_h \u \|_{H^1} \leq C h^2 \| \u \|_{H^2},	\\
	\| p- J_h p \|_{L^2} \leq C h \| p\|_{H^1},\\
	\| c- \Pi_h c\| +h \|c- \Pi_h c\|_{H^1} \leq Ch^2 \|c\|_{H^2},
	\end{split}
\end{align}
where $J_h$ is the classical interpolation operator $M \rightarrow M_h.$

Next, we introduce a projection operator with variable coefficients onto finite element spaces, For $0\leq n\leq N-1$, we define the stokes projection operator $(\P_h^{n+1}, \rho_h^{n+1}) : \V \times M_h \rightarrow \V_h \times M_h$ with variable coefficient by \cite{liyuanmhd2024}
\begin{align}
	\nu(c) (\nabla(\u -\P_h^{n+1} \u ), \nabla \v_h   ) +( \nabla \cdot \v_h, p-\rho_h^{n+1} p)=0, \quad \forall \v_h \in \V_h,\\
	( \nabla \cdot (\u - \P_h^{n+1} \u),q_h   ) =0, \quad \forall q_h \in M_h,
\end{align}
and there holds 
\begin{align}\label{biobdf-56}
	\| \u - P_h^{n+1} \u \|_{L^2} + h ( \| \nabla (\u -\P_h^{n+1} \u)\|_{L^2} + \| p -\rho_h^{n+1} p\|_{L^2} \leq C h^2    (   \| \u\|_{H^2}  + \| p\|_{H^1} ).
\end{align}

The Ritz projection $R_h^{n+1} : \tilde{H} \rightarrow \tilde{H}_h$ with variable coefficient is defined by 
\begin{align}
	(   \nabla (c - R^{n+1}_h c) , \nabla r_h  ) =0, \quad \forall q_h \in \tilde{H}_h,
\end{align}
and there holds \cite{scott2008,heyinnian2015,ravindran2016,girault2012}
\begin{align}
	\| c -R^{n+1}_h c\|_{L^2} + h \| \nabla(c-R^{n+1}_h c )\| \leq C h^2 \| c\|_{H^2},\\
	\| R^{n+1}_h c \|_{L^{\infty}} + \|R^{n+1}_h c \|_{W^{1,4}} \leq C \| c\|_{H^2}.
\end{align}

Denote
\begin{align*}
	&\u^{n+1} -\u^{n+1}_h= \u^{n+1} - \P^{n+1}_h \u^{n+1} + \P^{n+1}_h \u^{n+1}-\u^{n+1}_h = \eta^{n+1}_\u+\e^{n+1}_{\u},\\
	&	p^{n+1} -p^{n+1}_h= p^{n+1} - \rho^{n+1}_h p^{n+1} + \rho^{n+1}_h p^{n+1}-p^{n+1}_h = \eta^{n+1}_p+e^{n+1}_{p},\\
	&	c^{n+1} -c^{n+1}_h= c^{n+1} - R^{n+1}_h c^{n+1} + R^{n+1}_h c^{n+1}-c^{n+1}_h = \eta^{n+1}_c+e^{n+1}_{c}.
\end{align*}

Based on the above foundation, we propose the following BDF2 finite element schemes for $0 \leq n\leq N-1$.
 The second-order backward difference formula is a widely studied and popular approach that significantly improves the accuracy of temporal discretization. This method demonstrates superior stability compared to other classical second-order schemes, such as the Crank-Nicolson/Adams-Bashforth scheme, as highlighted in the literature \cite{heyinnian2007,lixiaoli2022,dingqianqian2024}.

\subsection{The decoupled BDF2 finite element scheme and stability analysis}
\begin{Algorithm}
\textbf{(The decoupled BDF2 finite element schemes)}
\end{Algorithm}
 
 \textbf{Step I:} We find the first step iteration $(\u^1_h,p^1_h)  \in \V_h \times M_h$ by 
 \begin{align}\label{biobdf-53}
 	&(d_t \u^1_h ,\v_h) + \nu(c^0_h+\alpha) (\nabla \u_h^1 ,\nabla \v_h) 	 + B(\u^0_h,\u^{1}_h , \v_h) -(\nabla \cdot \v_h,p^{1}_h) +(\nabla \cdot \u^{1}_h,q_h)\nonumber\\
 	=& -g( (1+\gamma c_h^0)\i_2,\v_h )+(\f^{1},\v_h), \quad \forall ~(\v_h,q_h ) \in \V_h \times M_h.
 \end{align}
 
 and find $c_h^1 \in \tilde{H}_h$ by
 \begin{align}\label{biobdf-13a}
 	(d_\tau c^{1}_h,r_h) + \theta (\nabla c^{1}_h,\nabla r_h ) + b(\u^0_h,c^{1}_h, r_h ) - U(c^0_h ,\frac{\partial r_h}{\partial x_2}) = U\alpha(1,\frac{\partial r_h}{\partial x_2}), \quad \forall ~r_h \in \tilde{H}_h.
 \end{align}

 \textbf{Step II:}
 For $n\geq 2$, we find $(\u^{n+1}_h,p^{n+1}_h)\in \V_h \times M_h $ by 
 \begin{align}\label{biobdf-7}
 	&(D_\tau \u^{n+1}_h ,\v_h) + \nu(\hat{c}^n_h+\alpha) ( \nabla \u^{n+1}_h ,\nabla \v_h) + B(\hat{\u}^n_h,\u^{n+1}_h ,\v_h) -(\nabla \cdot \v_h,p^{n+1}_h) +(\nabla \cdot \u^{n+1}_h,q_h) \nonumber\\
 	=&-g( (1+\gamma \hat{c}^{n}_h) \i_2,\v_h )+(\f^{n+1},\v_h),   \quad \forall (\v_h,q_h ) \in \V_h \times M_h.
 \end{align}
 
 and find $c^{n+1}_h\in \tilde{H}_h$ by
 \begin{align}\label{biobdf-8} 
 	(D_\tau c^{n+1}_h,r_h) + \theta(\nabla c^{n+1}_h,\nabla r_h ) + b(\hat{\u}^n_h,c^{n+1}_h,r_h ) - U(\hat{c}^n_h ,\frac{\partial r_h}{\partial x_2}) = U\alpha(1,\frac{\partial r_h}{\partial x_2}), \quad \forall r_h \in \tilde{H}_h.
 \end{align}

\begin{theorem}\label{biobdf-12}
	Under the condition (\ref{biobdf-5}), we have 
	\begin{align}
		&	\| \u^{m+1}_h \|^2_{L^2} +\| \hat{\u}^{m+1}_h \|^2_{L^2} +	\| c^{m+1}_h \|^2_{L^2} +\| \hat{c}^{m+1}_h \|^2_{L^2} + 2 \tau k^{-1}\sum_{n=1}^{m} \|  \nabla \u^{n+1}_h\|^2_{L^2} +2 \tau \theta \sum_{n=1}^{m} \|  \nabla c^{n+1}_h\|^2_{L^2}  \nonumber \\
		&	\leq  \| \u^1_h\|^2_{L^2}  + \| \hat{\u}^1_h\|^2_{L^2} + \| c^1_h\|^2_{L^2}  + \| \hat{c}^1_h\|^2_{L^2}  +C \tau \sum_{n=0}^{N-1}( \| \f^{n+1}\|_{L^2}^2 + |\Omega|  ),\quad \forall~1\leq m\leq N-1.
	\end{align}
	where $|\Omega|$ is expressed as the measure of $\Omega$.
\end{theorem}
\begin{proof}
	Taking $(\v_h,q_h) = 4\tau (\u^{n+1}_h,p^{n+1}_h), r_h=4\tau c^{n+1}_h$in (\ref{biobdf-7})-(\ref{biobdf-8}), by using (\ref{biobdf-2})  and (\ref{biobdf-5}), summing up (\ref{biobdf-7}) and (\ref{biobdf-8}), we have
	\begin{align}
		&	\| \u^{n+1}_h \|^2_{L^2} - 	\| \u^{n}_h \|^2_{L^2}  + 	\| \hat{\u}^{n+1}_h \|^2_{L^2} -\| \hat{\u}^{n}_h \|^2_{L^2}  + \| \u^{n+1}_h- 2 \u^n_h + \u^{n-1}_h\|_{L^2} + 4 \tau k^{-1} \|  \nabla \u^{n+1}_h\|^2_{L^2}  \nonumber \\
		&+	\| c^{n+1}_h \|^2_{L^2} - 	\| c^{n}_h \|^2_{L^2}  +	\| \hat{c}^{n+1}_h \|^2_{L^2} -\| \hat{c}^{n}_h \|^2_{L^2}  + \| c^{n+1}_h- 2 c^n_h + c^{n-1}_h\|_{L^2} +  4 \tau k^{-1} \|  \nabla c^{n+1}_h\|^2_{L^2}   \label{biobdf-10}\\
		\leq & 4\tau | g( (1+\gamma\hat{c}^n_h)\i_2,\u^{n+1}_h )   | + 4 \tau | (\f^{n+1},\u^{n+1}_h) |+ 4\tau | 	U(\hat{c}^n_h ,\frac{\partial c^{n+1}_h}{\partial x_2})   | + 4 \tau | U\alpha(1,\frac{\partial c^{n+1}_h}{\partial x_2}) |.\nonumber
	\end{align}
	
	Applications of the H\"{o}lder inequality and Young inequality yield
	\begin{align}
		4\tau | g( (1+\gamma\hat{c}^n_h)\i_2,\u^{n+1}_h )   | &\leq C \tau |\Omega|+ C\tau \| \hat{c}^n_h \|^2_{L^2}  + \tau  k^{-1}\|  \nabla \u^{n+1}_h\|^2_{L^2} ,\\
		4 \tau | (\f^{n+1},\u^{n+1}_h) |& \leq C \tau \| \f^{n+1}\|_{L^2}^2 +  \tau  k^{-1}\|  \nabla \u^{n+1}_h\|^2_{L^2},\\
		4\tau | 	U(\hat{c}^n_h ,\frac{\partial c^{n+1}_h}{\partial x_2})   | &\leq C \tau \| \hat{ c}^n_h\|^2_{L^2} + \tau k^{-1} \| \nabla c^{n+1}_h \|_{L^2}^2,\\
		4 \tau | U\alpha(1,\frac{\partial c^{n+1}_h}{\partial x_2}) | &\leq C \tau |\Omega| +\tau k^{-1}  \| \nabla c^{n+1}\|_{L^2}^2.
	\end{align}
	
	Thus substituting the above inequalities into (\ref{biobdf-10}), summing up and by using the discrete Gronwall's inequality  in Lemma \ref{biobdf-11}, we complete the proof of Theorem \ref{biobdf-12}.
	
\end{proof}

\subsection{The fully coupled BDF2 finite element scheme and stability analysis}
\begin{Algorithm}
	\textbf{(The fully coupled BDF2 finite element schemes)}
\end{Algorithm}

\textbf{Step I:} We define the first step iteration $(\u^1_h,p^1_h, c_h^1)  \in \V_h \times M_h \times \tilde{H}_h$ by 
\begin{align}\label{biobdf-53f}
	&(d_t \u^1_h ,\v_h) + \nu(c^0_h+\alpha) (\nabla \u_h^1 ,\nabla \v_h) 	 + B(\u^0_h,\u^{1}_h , \v_h)  -(\nabla \cdot \v_h,p^{1}_h) +(\nabla \cdot \u^{1}_h,q_h)\nonumber\\
	=& -g( (1+\gamma c_h^1)\i_2,\v_h )+(\f^{1},\v_h), \quad \forall ~(\v_h,q_h ) \in \V_h \times M_h.
\end{align}

and 
\begin{align}\label{biobdf-13af}
	(d_\tau c^{1}_h,r_h) + \theta (\nabla c^{1}_h,\nabla r_h ) + b(\u^0_h,c^{1}_h, r_h ) - U(c^0_h ,\frac{\partial r_h}{\partial x_2}) = U\alpha(1,\frac{\partial r_h}{\partial x_2}), \quad \forall ~r_h \in \tilde{H}_h.
\end{align}

\textbf{Step II:}
For $n\geq 2$, we find $(\u^{n+1}_h,p^{n+1}_h,c^{n+1}_h)\in \V_h \times M_h \times \tilde{H}_h$ by 
\begin{align}\label{biobdf-7f}
	&(D_\tau \u^{n+1}_h ,\v_h) + \nu(\hat{c}^n_h+\alpha) ( \nabla \u^{n+1}_h ,\nabla \v_h) + B(\hat{\u}^n_h,\u^{n+1}_h ,\v_h) -(\nabla \cdot \v_h,p^{n+1}_h) +(\nabla \cdot \u^{n+1}_h,q_h) \nonumber\\
	=&-g( (1+\gamma c^{n+1}_h) \i_2,\v_h )+(\f^{n+1},\v_h),   \quad \forall (\v_h,q_h ) \in \V_h \times M_h.
\end{align}

and
\begin{align}\label{biobdf-8f} 
	(D_\tau c^{n+1}_h,r_h) + \theta  (\nabla c^{n+1}_h,\nabla r_h ) + b(\hat{\u}^n_h,c^{n+1}_h,r_h ) - U(c^{n+1}_h ,\frac{\partial r_h}{\partial x_2}) = U\alpha(1,\frac{\partial r_h}{\partial x_2}), \quad \forall r_h \in \tilde{H}_h.
\end{align}
\begin{theorem}\label{biobdf-12f}
	Under the condition (\ref{biobdf-5}), we have 
	\begin{align}
		&	\| \u^{m+1}_h \|^2_{L^2} +\| \hat{\u}^{m+1}_h \|^2_{L^2} +	\| c^{m+1}_h \|^2_{L^2} +\| \hat{c}^{m+1}_h \|^2_{L^2} + 2 \tau k^{-1}\sum_{n=1}^{m} \|  \nabla \u^{n+1}_h\|^2_{L^2} +2 \tau \theta \sum_{n=1}^{m} \|  \nabla c^{n+1}_h\|^2_{L^2}  \nonumber \\
		&	\leq  \| \u^1_h\|^2_{L^2}  + \| \hat{\u}^1_h\|^2_{L^2} + \| c^1_h\|^2_{L^2}  + \| \hat{c}^1_h\|^2_{L^2}  +C \tau \sum_{n=0}^{N-1}( \| \f^{n+1}\|_{L^2}^2 + |\Omega|  ),\quad \forall~1\leq m\leq N-1.
	\end{align}
	where $|\Omega|$ is expressed as the measure of $\Omega$.
\end{theorem}
\begin{proof}
	Taking $(\v_h,q_h) = 4\tau (\u^{n+1}_h,p^{n+1}_h), r_h=4\tau c^{n+1}_h$ in (\ref{biobdf-7f})-(\ref{biobdf-8f}), by using (\ref{biobdf-2})  and (\ref{biobdf-5}), summing up (\ref{biobdf-7f}) and (\ref{biobdf-8f}), we have
	\begin{align}
		&	\| \u^{n+1}_h \|^2_{L^2} - 	\| \u^{n}_h \|^2_{L^2}  + 	\| \hat{\u}^{n+1}_h \|^2_{L^2} -\| \hat{\u}^{n}_h \|^2_{L^2}  + \| \u^{n+1}_h- 2 \u^n_h + \u^{n-1}_h\|_{L^2} + 4 \tau k^{-1} \|  \nabla \u^{n+1}_h\|^2_{L^2}  \nonumber \\
		&+	\| c^{n+1}_h \|^2_{L^2} - 	\| c^{n}_h \|^2_{L^2}  +	\| \hat{c}^{n+1}_h \|^2_{L^2} -\| \hat{c}^{n}_h \|^2_{L^2}  + \| c^{n+1}_h- 2 c^n_h + c^{n-1}_h\|_{L^2} +  4 \tau k^{-1} \|  \nabla c^{n+1}_h\|^2_{L^2}   \label{biobdf-10f}\\
		\leq & 4\tau | g( (1+\gamma c^{n+1}_h)\i_2,\u^{n+1}_h )   | + 4 \tau | (\f^{n+1},\u^{n+1}_h) |+ 4\tau | 	U( c^{n+1}_h ,\frac{\partial c^{n+1}_h}{\partial x_2})   | + 4 \tau | U\alpha(1,\frac{\partial c^{n+1}_h}{\partial x_2}) |.\nonumber
	\end{align}
	
	Applications of the H\"{o}lder inequality and Young inequality yield
	\begin{align}
		4\tau | g( (1+\gamma c^{n+1}_h)\i_2,\u^{n+1}_h )   | &\leq C \tau |\Omega|+ C\tau \| c^{n+1}_h \|^2_{L^2}  + \tau  k^{-1}\|  \nabla \u^{n+1}_h\|^2_{L^2} ,\\
		4 \tau | (\f^{n+1},\u^{n+1}_h) |& \leq C \tau \| \f^{n+1}\|_{L^2}^2 + C \tau \|  \u^{n+1}_h\|^2_{L^2},\\
		4\tau | 	U(c^{n+1}_h ,\frac{\partial c^{n+1}_h}{\partial x_2})   | &\leq C \tau \| c^{n+1}_h\|^2_{L^2} + \tau k^{-1} \| \nabla c^{n+1}_h \|_{L^2}^2,\\
		4 \tau | U\alpha(1,\frac{\partial c^{n+1}_h}{\partial x_2}) | &\leq C \tau |\Omega| +\tau k^{-1}  \| \nabla c^{n+1}\|_{L^2}^2.
	\end{align}
	
	Thus substituting the above inequalities into (\ref{biobdf-10f}), summing up and by using the discrete Gronwall's inequality  in Lemma \ref{biobdf-11}, we complete the proof of Theorem \ref{biobdf-12f}.
	
\end{proof}

%

\section{Error estimates for both coupled and decoupled BDF2 finite element schemes}

\subsection{Error estimate for the fully coupled finite element scheme}

\begin{theorem}\label{biobdf-54}
	Under the assumption (\ref{biobdf-5}) and (\ref{biobdf-36}), let $(\u^{i}, p^{i},c^{i})$  and $(\u_h^{i}, p_h^{i},c_h^{i})$ be the solutions of the continuous model (\ref{biobdf-48})- (\ref{biobdf-49}) and the finite element discrete scheme (\ref{biobdf-53}) - (\ref{biobdf-8}) , respectively, there exists some $C_1>0$ independent of $\tau$ and $h$ such that 
	\begin{align}\label{biobdf-55}
		\| \e^{m+1}_{\u} \|^2_{L^2} +\| e^{m+1}_{c} \|^2_{L^2} +\tau k^{-1} \sum_{i=0}^{m}  \| \nabla \e^{i+1}_{\u} \|^2_{L^2} +\tau \theta  \sum_{i=0}^{m} \| \nabla e^{i+1}_{c} \|^2_{L^2} \leq C_1 (\tau^4+h^4), \quad \forall~0 \leq m \leq N-1.
	\end{align}
\end{theorem}

We prove Theorem \ref{biobdf-54} by the mathematical induction method. Firstly, we need to prove  (\ref{biobdf-55}) is valid for $i=1$. Setting $(\v,q) = (\v_h,q_h)$ and taking the difference between (\ref{biobdf-50}) - (\ref{biobdf-51}) and (\ref{biobdf-53}) - (\ref{biobdf-13a}), we get error equation for $n=0$.
\begin{align}
	&(d_\tau \e^1_{\u}, \v_h) + \nu (c^0 +\alpha)(\nabla \e^1_{\u},\nabla \v_h) -(\nabla \cdot \v_h, e^1_p)+(\nabla e^1_\u,q_h)\nonumber\\
	=& -(d_\tau \eta^1_\u, \v_h) -   (\nu(c^0+\alpha) - \nu(c^0_h+\alpha))  (\nabla (  \eta^1_\u + \e^1_{\u}), \nabla \v_h ) -( (\nu(c^0+\alpha) - \nu(c^0_h+\alpha))   \nabla \u^1, \nabla \v_h) \label{biobdf-24}\\
	&-  B ( \u^0-\u^0_h ,\u^1,\v_h) + B ( \u^0_h,\eta^1_\u + \e^1_{\u},\v_h) + g \gamma ( (c^0-c^0_h)\i_2,\v_h )+(R^1_\u,\v_h),\nonumber\\
		&(d_\tau e^1_{c}, r_h) + \theta (\nabla e^1_{c},\nabla r_h)\nonumber \\
		=& -(d_\tau \eta^1_c, r_h) -  b ( \u^0-\u^0_h ,c^1,r_h) + b( \u^0_h,\eta^1_c + e^1_{c},r_h) + U(c^0-c^0_h,\frac{\partial r_h}{\partial x_2})+(R^1_c ,r_h).\label{biobdf-25}
\end{align}

Setting $ (\v_h,q_h)=2 \tau  (\e^1_{\u},e^1_{p})$ in (\ref{biobdf-24}), by using (\ref{biobdf-26}), H\"{o}lder inequality and Young inequality, we have 
\begin{align}
	2\tau |( d_\tau \eta^1_\u,\e^1_\u )| \leq& C \tau h^4 \| d_t \u^1\|^2_{H^2} +\frac{\tau}{4k} \| \nabla \e^1_\u\|^2_{L^2},\label{biobdf-57}\\
	2\tau|B( \u^0-\u^0_h,\u^1, \e^1_\u)| \leq& C \tau h^4 \| \u^1\|_{H^2}^2 + C \tau \| \e^1_\u \|^2_{L^2},\\
		2\tau|B( \u^0,\eta^1_\u+\e^1_\u,\e^1_\u)| \leq& C\tau \|\e^1_\u\|^2_{L^2} + C \tau h^4 + \frac{\tau}{4k} \| \nabla \e^1_\u\|^2_{L^2},\\
			2\tau | g\gamma ((c^0-c^0_h)\i_2 ,\e^1_\u)  | \leq &C \tau \| \e^1_\u \|^2_{L^2} +C \tau h^4,\\
				2\tau | ((\nu(c^0+\alpha) - \nu(c^0_h+\alpha))  \nabla \u^1, \nabla \e^1_{\u}    ) | \leq& C \tau \| c^0-c^0_h\|_{L^2} \|  \nabla \u^1\|_{L^{\infty}}  \| \nabla \e^1_\u\|_{L^2}\nonumber\\
			\leq & C \tau \| c^0-c^0_h\|^2_{L^2} \| \u^1\|^2_{H^3} + \frac{\tau}{4k}\| \nabla \e^1_\u\|^2_{L^2}\\
			\leq &C \tau h^4 \| \u^1\|^2_{H^3} +  \frac{\tau}{4k}\| \nabla \e^1_\u\|^2_{L^2},\nonumber\\
			2\tau| (	R^{1}_\u, \e^1_\u)  | \leq & C \tau^4+ \frac{1}{4}\| \e^1_\u\|^2_{L^2}.
\end{align}

	By using interpolation inequality (\ref{biobdf-15}) and inverse inequality (\ref{biobdf-16}) , we have 
\begin{align}
	\begin{split}
	&2 \tau | (\nu(c^0+\alpha) - \nu(c^0_h+\alpha)) ( \nabla ( \eta^1_\u + \e^1_\u),\nabla \e^1_\u   )|  \\
	\leq& C \tau \| c^0-c^0_h \|_{L^4} \| \nabla \eta^1_\u\|_{L^2} \| \nabla \e^1_\u\|_{L^4} + C \tau \| c^0-c^0_h \|_{L^2} \| \nabla \e^1_\u \|_{L^{\infty}} \| \nabla \e^1_\u\|_{L^2} \\
	\leq & C \tau h^{\frac{1}{2}} \| c^0-c^0_h\|^{\frac{1}{2}}_{L^2} \| \nabla (c^0-c^0_h)\|^{\frac{1}{2}}_{L^2} \| \nabla \e^1_\u\|_{L^2} + C \tau h \| \| \nabla \e^1_\u \|_{L^2}^2\\
	\leq &C \tau h \| c^0-c^0_h \|_{L^2}\| \nabla (c^0-c^0_h) \|_{L^2} + C \tau h \| \nabla \e^1_\u \|^2_{L^2} + \frac{\tau}{4k} \| \nabla \e^1_\u \|^2_{L^2}\\
	\leq & C \tau h^4  + C \tau h^2 \| \nabla (c^0-c^0_h) \|^2_{L^2} + \frac{\tau}{4k} \| \nabla \e^1_\u \|^2_{L^2}+ C \tau h \| \nabla \e^1_\u \|^2_{L^2} ,
\end{split}
\end{align}

Setting $ r_h=2 \tau  e^1_{c}$ in (\ref{biobdf-25}), by using the same technique, we have
\begin{align}
	2\tau |( d_\tau \eta^1_c,e^1_{c} )| \leq& C \tau h^4 \| d_\tau c^1\|^2_{H^2} +\frac{\tau}{4k} \| \nabla e^1_{c}\|^2_{L^2},\\
	2\tau|b( \u^0-\u^0_h,c^1, e^1_{c})| \leq& C \tau h^4 \| c^1\|_{H^2}^2 + \frac{\tau}{4} \| e^1_{c} \|^2_{L^2},\\
	2\tau|b( \u^0,\eta^1_c+e^1_{c},e^1_{c})| \leq& C\tau \|e^1_{c}\|^2_{L^2} + C \tau h^4 + \frac{\tau}{4k} \| \nabla e^1_{c}\|^2_{L^2},\\
	2\tau| U(c^0-c^0_h,\frac{\partial e^1_c}{\partial x_2})| \leq &C \tau h^4+ \frac{\tau}{4k}\| \nabla e^1_{c}\|^2_{L^2},\\
			2\tau| (	R^{1}_c, e^1_c)  | \leq & C \tau^4+ \frac{1}{4}\| e^1_c\|^2_{L^2}.\label{biobdf-61}
\end{align}

By using (\ref{biobdf-5}) and interpolation error (\ref{biobdf-26}), (\ref{biobdf-27}), summing up (\ref{biobdf-24}), (\ref{biobdf-25}) and substituting the above inequalities into it, we have 

\begin{align}
	\begin{split}
		&\| \e^1_\u\|^2_{L^2} + 2 \tau k^{-1} \| \nabla \e^1_\u\|_{L^2}^2 
		+\| e^1_c\|^2_{L^2}  + 2 \tau \theta \| \nabla e^1_c\|_{L^2}^2\\
		\leq& C \tau h^4( 1+\| d_\tau \u^1\|^2_{H^2} +\| \u^1\|_{H^2}^2+\| d_\tau c^1\|^2_{H^2} +\| c^1\|_{H^2}^2) \\
		&+ C \tau h ( \| \nabla \e^1_\u \|^2_{L^2} +\| \nabla e^1_c \|^2_{L^2} )
		 + C \tau h^2 \| \nabla (c^0-c^0_h) \|^2_{L^2}+C\tau^4.
	\end{split}
\end{align}

Furthermore, for sufficiently small $h < h_1$ with $Ch_1 \leq \min\{ k^{-1}, \theta\} $, $\tau < \tau_1$ with $C\tau_1 \leq \frac{1}{2}$, we can derive 
\begin{align}
	\begin{split}
		&\| \e^1_\u\|^2_{L^2} + \tau k^{-1} \| \nabla \e^1_\u\|_{L^2}^2 
		+\| e^1_c\|^2_{L^2}  +  \tau \theta\| \nabla e^1_c\|_{L^2}^2 \leq C (\tau^4+h^4).
	\end{split}
\end{align}

Next, we will give the convergence analysis for $m\leq n$ with $1\leq n \leq N-1$, we assume that (\ref{biobdf-55}) is valid for $m\leq n-1$ with $1\leq n\leq N-1$ by the method of mathematical induction $i=N-1$, i.e.,
\begin{align}
	\| \e^n_\u \|^2_{L^2} +\| e^n_c\|^2_{L^2}+ 
	\tau \sum_{i=0}^{n}(  \| \nabla \e^{i}_\u \|^2_{L^2} + \| \nabla e^{i}_c\|^2_{L^2}  ) \leq C^2_0 (\tau^4+h^4). \label{biobdf-19}
\end{align}

Setting $(\v,r)=(\v_h,r_h)$ in (\ref{biobdf-48}) - (\ref{biobdf-49}) and by taking the difference between (\ref{biobdf-48}) - (\ref{biobdf-49}) and (\ref{biobdf-7}) - (\ref{biobdf-8}), we derive the following error equations:
\begin{align}
	\begin{split}
		&(D_\tau \e^{n+1}_{\u} , \v_h) + ( \nu(\hat{ c}^n+\alpha)   \nabla \e^{n+1}_{\u} , \nabla \v_h)-(\nabla \cdot \v_h, e^{n+1}_p)+(\nabla e^{n+1}_\u,q_h) \\
		=& - (D_{\tau} \eta^{n+1}_\u ,r_h  )-(  (\nu(\hat{c}^n+\alpha) - \nu(\hat{c}^n_h+\alpha))   \nabla (   \eta^{n+1}_\u +\e^{n+1}_{\u}   ), \nabla \v_h  ) -( (\nu(\hat{c}^n+\alpha) - \nu(\hat{c}^n_h+\alpha))   \nabla \u^{n+1}, \nabla \v_h ) \\
		&-B(\hat{\eta}^n_\u +\hat{e}^n_{\u},\u^{n+1}  , \v_h  ) -B(\hat{\u}^n_h,  \eta^{n+1}_\u +\e^{n+1}_{\u} ,  \v_h ) 
		- (g \gamma ( \hat{\eta}^n_c +\hat{e}^n_{c} )\i_2,\v_h)+<R^{n+1}_\u,v_h>,\\
	&(D_\tau e^{n+1}_{c}  ,r_h) + \theta(  \nabla e^{n+1}_{c} , \nabla r_h) \\
	=& - (D_{\tau}\eta^{n+1}_c,r_h  )-b(\hat{\eta}^n_\u +\hat{e}^n_{\u},c^{n+1}  , r_h  ) -b(\hat{\u}^n_h,  \eta^{n+1}_c +e^{n+1}_{c} ,  r_h ) + U( \hat{\eta}^n_c +\hat{e}^n_{c}  , \frac{\partial r_h}{\partial x_2})+<R^{n+1}_c,r_h>,
		\end{split}
\end{align}

	Setting $(\v_h,q_h)=4\tau (\e^{n+1} _{\u} , e^{n+1}_{p} ),  r_h =4 \tau e^{n+1}_{c} $  and noticing that  (\ref{biobdf-5}), we have 	
		\begin{align}
			&	\| \e^{n+1}_{\u} \|^2_{L^2} - \| \e^n_{\u}\|^2_{L^2} + \| \hat{\e} ^{n+1}_{\u} \|^2_{L^2} - \| \hat{\e}^n_{\u}\|^2_{L^2} + \| \e^{n+1}_{\u} -2 \e^n_{\u} + \e^{n-1}_{\u} \|^2_{L^2} + 4\tau k^{-1} \| \nabla \e^{n+1}_{\u} \|^2_{L^2} \label{biobdf-18}\\
			\leq &4\tau | (D_\tau \eta^{n+1}_\u, \e^{n+1}_{\u} )| 
			+ 4 \tau| (  (\nu(\hat{c}^n+\alpha) - \nu(\hat{c}^n_h+\alpha))   \nabla (   \eta^{n+1}_\u 
			+\e^{n+1}_{{\u}}   ), \nabla \e^{n+1}_{\u}  )| \nonumber
		\\
			&+  4 \tau|( (\nu(\hat{c}^n+\alpha) - \nu(\hat{c}^n_h+\alpha))    \nabla \u^{n+1}, \nabla \e^{n+1}_{\u} ) |+ 4 \tau|  B(\hat{\eta}^n_\u +\hat{e}^n_{\u},\u^{n+1}  , \e^{n+1}_{\u}  )| |\nonumber\\
			&
			+  4 \tau|  B(\hat{\u}^n_h,  \eta^{n+1}_\u +\e^{n+1}_{\u} , \e^{n+1}_{\u}  ) 
			+ 4\tau|(g \gamma ( \hat{\eta}^n_c +\hat{e}^n_{c} )\i_2,\e^{n+1}_{\u} )|+ 4\tau|  <R^{n+1}_\u, \e^{n+1}_\u>| \nonumber,\\
	&	\| e^{n+1}_{c}  \|^2_{L^2} - \| e^n_{c}\|^2_{L^2} + \| \hat{e} ^{n+1}_{c} \|^2_{L^2} - \| \hat{e}^n_{c}\|^2_{L^2} + \| e^{n+1}_{c}  -2 e^n_{c} + e^{n-1}_{c} \|^2_{L^2} + 4\tau \theta \| \nabla e^{n+1}_{c}  \|^2_{L^2}\label{biobdf-22}\\
		\leq &4\tau | (D_\tau \eta^{n+1}_c, e^{n+1}_{c} )| 
+ 4 \tau|     b(\hat{\eta}^n_\u +\hat{e}^n_{\u},c^{n+1}  , e^{n+1}_{c}   ) | +  4 \tau|   b(\hat{\u}^n_h,  \eta^{n+1}_c +e^{n+1}_{c} ,  e^{n+1}_{c} ) |
\nonumber\\
&+ 4\tau U|( \hat{\eta}^n_c +\hat{e}^n_{c}  , \frac{\partial e^{n+1}_{c} }{\partial x_2})|+4\tau|  <R^{n+1}_c, e^{n+1}_c>|.\nonumber
	\end{align}

	Applications of the H\"{o}lder inequality and Young inequality yield
	\begin{align}
		4\tau | (D_\tau \eta^{n+1}_\u, \e^{n+1}_{\u} )|  \leq C \tau h^4 \| D_{\tau} \u^{n+1} \|^2_{H^2} + C \tau \| \e^{n+1}_{\u} \|^2_{L^2},\label{biobdf-59}\\
			4\tau|(g \gamma ( \hat{\eta}^n_c +\hat{e}^n_{c} )\i_2,\e^{n+1}_{\u} )|\leq C \tau \| \hat{\eta}^n_c 
		+\hat{e}^n_{c}\|_{L^2}^2 + C \tau \|\e^{n+1}_\u\|^2_{L^2},\\
		4 \tau | <R^{n+1}_\u,\e^{n+1}_\u>| \leq C \tau \| R^{n+1}_\u \|_{H^{-1}}^2 + \frac{\tau}{4k} \| \nabla \e^{n+1}_\u|^2_{L^2}.
	\end{align}
	
	By using (\ref{biobdf-5}), (\ref{biobdf-56}), interpolation inequality (\ref{biobdf-15}) and inverse inequality (\ref{biobdf-16}) , we have 
	\begin{align}
		\begin{split}
&4 \tau| (    (\nu(\hat{c}^n+\alpha) - \nu(\hat{c}^n_h+\alpha))  \nabla (   \eta^{n+1}_\u 
+\e^{n+1}_{{\u}}   ), \nabla \e^{n+1}_{\u}  )| \\
\leq& C \tau \| \hat{\eta}^n_c +\hat{e}^n_{c}  \|_{L^4} \|  \nabla  \eta^{n+1}_\u \|_{L^2}\| \nabla \e^{n+1}_{{\u}} \|_{L^4}
+C \tau \| \hat{\eta}^n_c +\hat{e}^n_{c}\|_{L^2}\|\nabla  \e^{n+1}_{{\u}} \|_{L^{\infty}}  \|\e^{n+1}_{{\u}} \|_{L^2}\\
\leq &C \tau h^{\frac{1}{2}} \| \hat{\eta}^n_c +\hat{e}^n_{c}  \|^{\frac{1}{2}}_{L^2}\|   \nabla (\hat{\eta}^n_c +\hat{e}^n_{c} ) \|^{\frac{1}{2}}_{L^2} \| \nabla \e^{n+1}_{\u} \|_{L^2} + C \tau h^{-1} \| \hat{\eta}^n_c +\hat{e}^n_{c}\|_{L^2}\| \nabla \e^{n+1}_{{\u}}\|^2_{L^2}\\
\leq &\frac{\tau}{4k} \| \nabla \e^{n+1}_{\u} \|^2_{L^2} + C \tau h \|\hat{\eta}^n_c +\hat{e}^n_{c}\|_{L^2} \| \nabla(\hat{\eta}^n_c +\hat{e}^n_{c})\|_{L^2}+C \tau h^{-1} \| \hat{\eta}^n_c +\hat{e}^n_{c}\|_{L^2}\| \nabla \e^{n+1}_{\u} \|^2_{L^2}\\
\leq &  \frac{\tau}{4k} \| \nabla \e^{n+1}_{\u} \|^2_{L^2} + C \tau \| \hat{\eta}^n_c +\hat{e}^n_{c}\|^2_{L^2} + C \tau h^2 \| \nabla(\hat{\eta}^n_c +\hat{e}^n_{c}) \|^2_{L^2} +C \tau h^{-1} \| \hat{\eta}^n_c +\hat{e}^n_{c}\|_{L^2}\| \nabla \e^{n+1}_{\u} \|^2_{L^2}.
		\end{split}
	\end{align}
	
Applying (\ref{biobdf-5}), (\ref{biobdf-52}), one has 
\begin{align}
	\begin{split}
&4 \tau|(   (\nu(\hat{c}^n+\alpha) - \nu(\hat{c}^n_h+\alpha))  \nabla \u^{n+1}, \nabla \e^{n+1}_{\u} ) | +4 \tau|  B(\hat{\eta}^n_\u +\hat{e}^n_{\u},\u^{n+1}  , \e^{n+1}_{\u}  )|  \\
\leq& C \tau \|  \hat{\eta}^n_c 
+\hat{e}^n_{c}  \|_{L^2} \| \nabla \u^{n+1}\|_{L^{\infty}} \|  \nabla \e^{n+1}_{\u} \|_{L^2}+ C \tau \|  \hat{\eta}^n_c 
+\hat{e}^n_{c}   \|_{L^2} \| \nabla \u^{n+1}\|_{L^{\infty}} \|  \nabla \e^{n+1}_{\u} \|_{L^2}\\
\leq &C \tau \| \hat{\eta}^n_c 
+\hat{e}^n_{c}\|_{L^2}^2+ \frac{\tau}{4k}\|  \nabla \e^{n+1}_{\u} \|_{L^2}.
		\end{split}
\end{align}

Noticing that (\ref{biobdf-2}), (\ref{biobdf-17}) we can get
\begin{align}
		\begin{split}
 &4 \tau|  B(\hat{\u}^n_h,  \eta^{n+1}_\u +\e^{n+1}_{\u} , \e^{n+1}_{\u} )|\\
 \leq &4 \tau |B (\hat{\eta}^{n}_\u +\hat{\e}^{n}_{\u},   \eta^{n+1}_\u +\e^{n+1}_{\u} ,\e^{n+1}_{\u} )| 
 +4 \tau |B ( \hat{\u}^n,   \eta^{n+1}_\u +\e^{n+1}_{\u} ,\e^{n+1}_{\u} )|\\
 \leq &C \tau \| \hat{\eta}^{n}_\u +\hat{\e}^{n}_{\u} \|_{L^4} \| \nabla \eta^{n+1}_\u\|_{L^2}\|\e^{n+1}_{\u} \|_{L^4} + C\tau \|\hat{\u}^n\|_{L^{\infty}} \| \nabla \e^{n+1}_{\u} \|_{L^2} \| \eta^{n+1}_\u\|_{L^2}\\
 \leq & \frac{\tau}{4k} \| \nabla \e^{n+1}_{\u} \|^2_{L^2}+ C\tau h^2 \|  \hat{\eta}^{n}_\u +\hat{\e}^{n}_{\u} \|_{L^2}\| \nabla(\hat{\eta}^{n}_\u +\hat{\e}^{n}_{\u})\|_{L^2} +C\tau h^4\\
 \leq & \frac{\tau}{4k} \| \nabla \e^{n+1}_{\u} \|^2_{L^2} +C\tau \|  \hat{\eta}^{n}_\u +\hat{\e}^{n}_{\u} \|^2_{L^2} + C\tau h^4 \| \nabla (  \hat{\eta}^{n}_\u +\hat{\e}^{n}_{\u} )\|^2_{L^2} + C\tau h^4.
 		\end{split}
\end{align}

Substituting the above inequalities into (\ref{biobdf-18}), we can derive 
\begin{align}\label{biobdf-22}
			\begin{split}
	&	\| \e^{n+1}_{\u} \|^2_{L^2} - \| \e^n_{\u}\|^2_{L^2} + \| \hat{\e} ^{n+1}_{\u} \|^2_{L^2} - \| \hat{\e}^n_{\u}\|^2_{L^2} + \| \e^{n+1}_{\u} -2 \e^n_{\u} + \e^{n-1}_{\u} \|^2_{L^2} + 4\tau k^{-1} \| \nabla \e^{n+1}_{\u} \|^2_{L^2}\\
		\leq &C \tau h^4(1+ \| D_{\tau} \u^{n+1} \|^2_{H^2} )+ C \tau \| \e^{n+1}_{\u} \|_{L^2}+C \tau h^{-1} \| \hat{\eta}^n_c +\hat{e}^n_{c}\|_{L^2}\| \nabla \e^{n+1}_{\u} \|^2_{L^2}+C \tau \| R^{n+1}_\u \|_{H^{-1}}^2 \\
		& + C\tau (\| \hat{\eta}^n_c +\hat{e}^n_{c}\|^2_{L^2} +\|  \hat{\eta}^{n}_\u +\hat{\e}^{n}_{\u}\|^2_{L^2})+C \tau h^2 \| \nabla(\hat{\eta}^n_c +\hat{e}^n_{c}) \|^2_{L^2}+C\tau h^4 \| \nabla ( \hat{\eta}^{n}_\u +\hat{\e}^{n}_{\u}  )\|^2_{L^2}.
		 		\end{split}
\end{align}

Next, we estimate $e^{n+1}_c$ by the same method, applications of the H\"{o}lder inequality and Young inequality yield
\begin{align}
	4\tau | (D_\tau \eta^{n+1}_c, e^{n+1}_{c} )|  \leq C \tau h^4 \| D_{\tau} c^{n+1} \|^2_{H^2} + C \tau \| e^{n+1}_{c}  \|_{L^2},\\
4\tau|( \hat{\eta}^n_c +\hat{e}^n_{c}  , \frac{\partial e^{n+1}_{c} }{\partial x_2})| \leq C \tau \| \hat{\eta}^n_c +\hat{e}^n_{c} \|^2_{L^2} + \frac{\tau}{4k} \| \nabla e^{n+1}_c \|^2_{L^2},\\
4 \tau  | <R^{n+1}_c , e^{n+1}_c> | \leq C \tau \| R^{n+1}\|_{H^{-1}}^2 + \frac{\tau}{4k} \| \nabla e^{n+1}_c\|^2_{L^2}. 
\end{align}
%

Applying (\ref{biobdf-47}), one has 
\begin{align}\label{biobdf-62}
	\begin{split}
		&4 \tau|  b(\hat{\eta}^n_\u +\hat{e}^n_{\u},c^{n+1}  , e^{n+1}_{c}   )|  \\
		\leq&  C \tau \|  \hat{\eta}^n_c 
		+\hat{e}^n_{c}   \|_{L^2} \| \nabla c^{n+1}\|_{L^{\infty}} \|  \nabla e^{n+1}_{c} \|_{L^2}\\
		\leq &C \tau \| \hat{\eta}^n_c 
		+\hat{e}^n_{c}\|_{L^2}^2+ \frac{\tau}{4k}\|  \nabla e^{n+1}_{c} \|^2_{L^2}.
	\end{split}
\end{align}
%

Substituting the above inequalities into (\ref{biobdf-22}), we can derive 
\begin{align}\label{biobdf-23}
	\begin{split}
		&	\| e^{n+1}_{c}  \|^2_{L^2} - \|e^n_{c}\|^2_{L^2} + \| \hat{e} ^{n+1}_{c} \|^2_{L^2} - \| \hat{e}^n_{c}\|^2_{L^2} + \| e^{n+1}_{c}  -2 e^n_{c} + e^{n-1}_{c} \|^2_{L^2} + 4\tau k^{-1} \| \nabla e^{n+1}_{c}  \|^2_{L^2}\\
		\leq &C \tau h^4(1+ \| D_{\tau} c^{n+1} \|^2_{H^2} )+ C \tau \| e^{n+1}_{c}  \|_{L^2}+C \tau h^{-1} \| \hat{\eta}^n_c +\hat{e}^n_{c}\|_{L^2}\| \nabla e^{n+1}_{c} \|^2_{L^2}\\
		& + C\tau (\| \hat{\eta}^n_c +\hat{e}^n_{c}\|^2_{L^2} +\|  \hat{\eta}^{n}_\u +\hat{\e}^{n}_{\u}\|^2_{L^2})+C \tau h^2 \| \nabla(\hat{\eta}^n_c +\hat{e}^n_{c}) \|^2_{L^2}+C\tau h^4 \| \nabla ( \hat{\eta}^{n}_\u +\hat{\e}^{n}_{\u}  )\|^2_{L^2}.
	\end{split}
\end{align}

Thus summing (\ref{biobdf-22}) and (\ref{biobdf-23}), by using (\ref{biobdf-19}), for sufficiently small $h < h_1$ with $Ch_1 \leq k^{-1}$, we can get 
$$C \tau h^{-1} \| \hat{\eta}^n_c +\hat{e}^n_{c}\|_{L^2} (\| \nabla \e^{n+1}_{\u} \|_{L^2}+\| \nabla e^{n+1}_{c} \|_{L^2})\leq  \tau k^{-1} (\| \nabla \e^{n+1}_{\u} \|_{L^2}+\| \nabla e^{n+1}_{c} \|_{L^2}).$$

Furthermore
\begin{align}
		\begin{split}
&	\| \e^{n+1}_{\u} \|^2_{L^2} - \| \e^n_{\u}\|^2_{L^2} + \| \hat{\e} ^{n+1}_{\u} \|^2_{L^2} - \| \hat{\e}^n_{\u}\|^2_{L^2} + \| \e^{n+1}_{\u} -2 \e^n_{\u} + \e^{n-1}_{\u} \|^2_{L^2} + 4\tau k^{-1} \| \nabla \e^{n+1}_{\u} \|^2_{L^2}\\
		&+\| e^{n+1}_{c}  \|^2_{L^2} - \| e^n_{c}\|^2_{L^2} + \| \hat{e} ^{n+1}_{c} \|^2_{L^2} - \| \hat{e}^n_{c}\|^2_{L^2} + \| e^{n+1}_{c}  -2 e^n_{c} + e^{n-1}_{c} \|^2_{L^2} + 4\tau k^{-1} \| \nabla e^{n+1}_{c}  \|^2_{L^2}\\
		\leq & C\tau(\|\e^{n+1}_{\u} \|^2_{L^2}+  \|e^{n+1}_{c}  \|^2_{L^2}  )+C \tau h^4(1+ \| D_{\tau} \u^{n+1} \|^2_{H^2} +\| D_{\tau} c^{n+1} \|^2_{H^2}) \\
		&+ C\tau h^2( \| \nabla \hat{e}^n_c\|^2_{L^2}  + \|\nabla \hat{\e}^n_\u\|^2_{L^2}  ) +C \tau (\|  \hat{e}^n_c\|^2_{L^2}  + \| \hat{\e}^n_\u\|^2_{L^2})
		\end{split}
\end{align}

Taking the sum and applying the discrete Gronwall's inequality in Lemma \ref{biobdf-11}, there exists some $C_1>0$ independent of $\tau$ and $h$ such that 
\begin{align}
		\| \e^{n+1}_{\u} \|^2_{L^2} +\| e^{n+1}_{c}  \|^2_{L^2} +\tau k^{-1} \sum_{m=0}^{n}  \| \nabla \e^{n+1}_{\u} \|^2_{L^2}+  \tau \theta  \sum_{m=0}^{n} \| \nabla e^{n+1}_{c}  \|^2_{L^2}  \leq C_1 (\tau^4+h^4).
\end{align}

\subsection{Error estimate for fully coupled finite element scheme}

\begin{theorem}\label{biobdf-54f}
	Under the assumption (\ref{biobdf-5}) and (\ref{biobdf-36}), let $(\u^{i}, p^{i},c^{i})$  and $(\u_h^{i}, p_h^{i},c_h^{i})$ are the solutions of the continuous model (\ref{biobdf-48})- (\ref{biobdf-49}) and the finite element discrete scheme (\ref{biobdf-53f}) - (\ref{biobdf-8f}) , respectively, there exists some $C_1>0$ independent of $\tau$ and $h$ such that 
	\begin{align}\label{biobdf-55f}
		\| \e^{m+1}_{\u} \|^2_{L^2} +\| e^{m+1}_{c} \|^2_{L^2} +\tau k^{-1} \sum_{i=0}^{m} (  \| \nabla \e^{i+1}_{\u} \|^2_{L^2}+ \| \nabla e^{i+1}_{c} \|^2_{L^2}) \leq C_1 (\tau^4+h^4), \quad \forall~0 \leq m \leq N-1.
	\end{align}
\end{theorem}

We prove Theorem \ref{biobdf-54f} by the mathematical induction method. Firstly, we need to prove  (\ref{biobdf-55f}) is valid for $i=1$. Taking the difference between (\ref{biobdf-31}) - (\ref{biobdf-38}) and (\ref{biobdf-53f}) - (\ref{biobdf-13af}), we get error equation for $n=0$.
\begin{align}
&(d_\tau \e^1_{\u}, \v_h) + \nu (c^0 +\alpha)(\nabla \e^1_{\u},\nabla \v_h) -(\nabla \cdot \v_h, e^1_p)+(\nabla e^1_\u,q_h)\nonumber\\
=& -(d_\tau \eta^1_\u, \v_h) - (\nu(c^0+\alpha) - \nu(c^0_h+\alpha))  (\nabla (  \eta^1_\u + \e^1_{\u}), \nabla \v_h ) -(  (\nu(c^0+\alpha) - \nu(c^0_h+\alpha))   \nabla \u^1, \nabla \v_h) \label{biobdf-24f}\\
&-  B ( \u^0-\u^0_h ,\u^1,\v_h) + B ( \u^0_h,\eta^1_\u + \e^1_{\u},\v_h) + g \gamma ( (\eta^1_c+e^1_c)\i_2,\v_h )+(R^1_\u,\v_h),\nonumber\\
&(d_\tau e^1_{c}, r_h) + \theta (\nabla e^1_{c},\nabla r_h)\nonumber \\
=& -(d_\tau \eta^1_c, r_h) -  b ( \u^0-\u^0_h ,c^1,r_h) + b( \u^0_h,\eta^1_c + e^1_{c},r_h) + U(\eta^1_c+e^1_c,\frac{\partial r_h}{\partial x_2})+(R^1_c ,r_h).\label{biobdf-25f}
\end{align}

Based the above inequalities (\ref{biobdf-57}) - (\ref{biobdf-61}) and 
\begin{align}
	2\tau | g\gamma ((\eta^1_c+e^1_c)\i_2 ,\e^1_\u)  | \leq C \tau ( \| \e^1_\u \|^2_{L^2} + \| e^1_c \|^2_{L^2})+C \tau h^4,\\
		2\tau| U(\eta^1_c+e^1_c,\frac{\partial e^1_c}{\partial x_2})| \leq C \tau h^4+ C\tau \|e^1_c\|^2_{L^2}+ \frac{\tau}{4k}\| \nabla e^1_{c}\|^2_{L^2},
\end{align}

By using (\ref{biobdf-5}) and interpolation error (\ref{biobdf-26}), (\ref{biobdf-27}), summing up (\ref{biobdf-24f}), (\ref{biobdf-25f}) and substituting the above inequalities into it, we have 

\begin{align}
	\begin{split}
		&\| \e^1_\u\|^2_{L^2} + 2 \tau k^{-1} \| \nabla \e^1_\u\|_{L^2}^2 
		+\| e^1_c\|^2_{L^2}  + 2 \tau \theta \| \nabla e^1_c\|_{L^2}^2\\
		\leq& C \tau h^4( 1+\| d_\tau \u^1\|^2_{H^2} +\| \u^1\|_{H^2}^2+\| d_\tau c^1\|^2_{H^2} +\| c^1\|_{H^2}^2) +C \tau (\| \e^1_\u \|^2_{L^2}
		+\| e^1_c \|^2_{L^2})\\
		&+ C \tau h ( \| \nabla \e^1_\u \|^2_{L^2} +\| \nabla e^1_c \|^2_{L^2} )
		+ C \tau h^2 \| \nabla (c^0-c^0_h) \|^2_{L^2}+C\tau^4.
	\end{split}
\end{align}

Furthermore, for sufficiently small $h < h_1$ with $Ch_1 \leq k^{-1}$, $\tau < \tau_1$ with $C\tau_1 \leq \frac{1}{2}$, we can derive 
\begin{align}
	\begin{split}
		&\| \e^1_\u\|^2_{L^2} + \tau k^{-1} \| \nabla \e^1_\u\|_{L^2}^2 
		+\| e^1_c\|^2_{L^2}  +  \tau \theta \| \nabla e^1_c\|_{L^2}^2 \leq C (\tau^4+h^4).
	\end{split}
\end{align}

Setting $(\v,r)=(\v_h,r_h)$ in (\ref{biobdf-48}) - (\ref{biobdf-49}) and by taking the difference between (\ref{biobdf-48}) - (\ref{biobdf-49}) and (\ref{biobdf-7f}) - (\ref{biobdf-8f}), we derive the following error equations:
\begin{align}
	\begin{split}
		&(D_\tau \e^{n+1}_{\u} , \v_h) + ( \nu(\hat{ c}^n+\alpha)   \nabla \e^{n+1}_{\u} , \nabla \v_h) -(\nabla \cdot \v_h, e^{n+1}_p)+(\nabla e^{n+1}_\u,q_h)\\
		=& - (D_{\tau} \eta^{n+1}_\u ,r_h  )-((\nu(\hat{c}^n+\alpha) - \nu(\hat{c}^n_h+\alpha)) \nabla (   \eta^{n+1}_\u +\e^{n+1}_{\u}   ), \nabla \v_h  ) -(  (\nu(\hat{c}^n+\alpha) - \nu(\hat{c}^n_h+\alpha))  \nabla \u^{n+1}, \nabla \v_h ) \\
		&-B(\hat{\eta}^n_\u +\hat{e}^n_{\u},\u^{n+1}  , \v_h  ) -B(\hat{\u}^n_h,  \eta^{n+1}_\u +\e^{n+1}_{\u} ,  r_h ) 
		- (g \gamma ( \eta^{n+1}_c +e^{n+1}_{c} )\i_2,\v_h)+<R^{n+1}_\u,v_h>,\\
		&(D_\tau e^{n+1}_{c}  ,r_h) + \theta (\nabla e^{n+1}_{c} , \nabla r_h) \\
		=& - (D_{\tau}\eta^{n+1}_c,r_h  )-b(\hat{\eta}^n_\u +\hat{e}^n_{\u},c^{n+1}  , r_h  ) -b(\hat{\u}^n_h,  \eta^{n+1}_c +e^{n+1}_{c} ,  r_h ) + ( \eta^{n+1}_c +e^{n+1}_{c}, \frac{\partial r_h}{\partial x_2})+<R^{n+1}_c,r_h>,
	\end{split}
\end{align}

Setting $(\v_h,q_h)=4\tau (\e^{n+1} _{\u} , e^{n+1}_{p} ),  r_h =4 \tau e^{n+1}_{c} $  and noticing that  (\ref{biobdf-5}), we have 	
\begin{align}
	&	\| \e^{n+1}_{\u} \|^2_{L^2} - \| \e^n_{\u}\|^2_{L^2} + \| \hat{\e} ^{n+1}_{\u} \|^2_{L^2} - \| \hat{\e}^n_{\u}\|^2_{L^2} + \| \e^{n+1}_{\u} -2 \e^n_{\u} + \e^{n-1}_{\u} \|^2_{L^2} + 4\tau k^{-1} \| \nabla \e^{n+1}_{\u} \|^2_{L^2} \label{biobdf-18f}\\
	\leq &4\tau | (D_\tau \eta^{n+1}_\u, \e^{n+1}_{\u} )| 
	+ 4 \tau| (\nu (\hat{\eta}^n_c +\hat{e}^n_{c}  )  \nabla (   \eta^{n+1}_\u 
	+\e^{n+1}_{{\u}}   ), \nabla \e^{n+1}_{\u}  )|
	+  4 \tau|(  \nu (\hat{\eta}^n_c 
	+\hat{e}^n_{c}  )   \nabla \u^{n+1}, \nabla \e^{n+1}_{\u} ) |\nonumber\\
	&+ 4 \tau|  B(\hat{\eta}^n_\u +\hat{e}^n_{\u},\u^{n+1}  , \e^{n+1}_{\u}  )| 
	+  4 \tau|  B(\hat{\u}^n_h,  \eta^{n+1}_\u +\e^{n+1}_{\u} , \e^{n+1}_{\u}  ) |\notag \\
	&
	+ 4\tau|(g \gamma ( \eta^{n+1}_c +e^{n+1}_{c} )\i_2,\e^{n+1}_{\u} )|+ 4\tau|  <R^{n+1}_\u, \e^{n+1}_\u>| \nonumber,\\
	&	\| e^{n+1}_{c}  \|^2_{L^2} - \| e^n_{c}\|^2_{L^2} + \| \hat{e} ^{n+1}_{c} \|^2_{L^2} - \| \hat{e}^n_{c}\|^2_{L^2} + \| e^{n+1}_{c}  -2 e^n_{c} + e^{n-1}_{c} \|^2_{L^2} 
	+ 4\tau \theta \| \nabla e^{n+1}_{c}  \|^2_{L^2}\label{biobdf-22f}\\
	\leq &4\tau | (D_\tau \eta^{n+1}_c, e^{n+1}_{c} )|  + 4 \tau|     b(\hat{\eta}^n_\u +\hat{e}^n_{\u},c^{n+1}  , e^{n+1}_{c}   ) | +  4 \tau|   b(\hat{\u}^n_h,  \eta^{n+1}_c +e^{n+1}_{c} ,  e^{n+1}_{c} ) |\nonumber\\
	&+ 4\tau|( \eta^{n+1}_c +e^{n+1}_{c}, \frac{\partial e^{n+1}_{c} }{\partial x_2})|+4\tau|  <R^{n+1}_c, e^{n+1}_c>|.\nonumber
\end{align}

Based the above inequalities (\ref{biobdf-59}) - (\ref{biobdf-62}) and 
\begin{align}
	2\tau | g\gamma (\eta^{n+1}_c +e^{n+1}_{c})\i_2 ,\e^{n+1}_\u)  | \leq C \tau ( \| \e^{n+1}_\u \|^2_{L^2} + \| e^{n+1}_c \|^2_{L^2})+C \tau h^4,\\
	2\tau| U(\eta^{n+1}_c +e^{n+1}_{c},\frac{\partial e^{n+1}_c}{\partial x_2})| \leq C \tau h^4+ C\tau \|e^{n+1}_c\|^2_{L^2}+ \frac{\tau}{4k}\| \nabla e^{n+1}_{c}\|^2_{L^2},
\end{align}

Thus summing (\ref{biobdf-18f}) and (\ref{biobdf-22f}), by using (\ref{biobdf-19}), for sufficiently small $h < h_1$ with $Ch_1 \leq k^{-1}$, we can get 
$$C \tau h^{-1} \| \hat{\eta}^n_c +\hat{e}^n_{c}\|_{L^2} (\| \nabla \e^{n+1}_{\u} \|_{L^2}+\| \nabla e^{n+1}_{c} \|_{L^2})\leq  \tau k^{-1} (\| \nabla \e^{n+1}_{\u} \|_{L^2}+\| \nabla e^{n+1}_{c} \|_{L^2}).$$

Furthermore
\begin{align}
	\begin{split}
		&	\| \e^{n+1}_{\u} \|^2_{L^2} - \| \e^n_{\u}\|^2_{L^2} + \| \hat{\e} ^{n+1}_{\u} \|^2_{L^2} - \| \hat{\e}^n_{\u}\|^2_{L^2} + \| \e^{n+1}_{\u} -2 \e^n_{\u} + \e^{n-1}_{\u} \|^2_{L^2} + 4\tau k^{-1} \| \nabla \e^{n+1}_{\u} \|^2_{L^2}\\
		&+\| e^{n+1}_{c}  \|^2_{L^2} - \| e^n_{c}\|^2_{L^2} + \| \hat{e} ^{n+1}_{c} \|^2_{L^2} - \| \hat{e}^n_{c}\|^2_{L^2} + \| e^{n+1}_{c}  -2 e^n_{c} + e^{n-1}_{c} \|^2_{L^2} + 4\tau \theta \| \nabla e^{n+1}_{c}  \|^2_{L^2}\\
		\leq & C\tau(\|\e^{n+1}_{\u} \|^2_{L^2}+  \|e^{n+1}_{c}  \|^2_{L^2}  )+C \tau h^4(1+ \| D_{\tau} \u^{n+1} \|^2_{H^2} +\| D_{\tau} c^{n+1} \|^2_{H^2}) \\
		&+ C\tau h^2( \| \nabla \hat{e}^n_c\|^2_{L^2}  + \|\nabla \hat{\e}^n_\u\|^2_{L^2}  ) +C \tau (\|  \hat{e}^n_c\|^2_{L^2}  + \| \hat{\e}^n_\u\|^2_{L^2})
	\end{split}
\end{align}

Taking the sum and applying the discrete Gronwall's inequality in Lemma \ref{biobdf-11}, there exists some $C_1>0$ independent of $\tau$ and $h$ such that 
\begin{align}
	\| \e^{n+1}_{\u} \|^2_{L^2} +\| e^{n+1}_{c}  \|^2_{L^2} +\tau k^{-1} \sum_{m=0}^{n}   \| \nabla \e^{n+1}_{\u} \|^2_{L^2}+ \tau \theta \sum_{m=0}^{n}  \| \nabla e^{n+1}_{c}  \|^2_{L^2} \leq C_1 (\tau^4+h^4).
\end{align}

\section{Numerical Experiments}

In this section, we will present the numerical results to verify our theoretical analysis. All programs are implemented using the free finite element software FreeFem++ \cite{Hecht2012NewDI}. For the sake of simplicity, we solve the following coupled system with the artificial functions \(g\) and \(\mathbf{f}\):

 \begin{align}
	\frac{\partial \u}{\partial t} - \div ( \nu(c) D(\u)) + \u \cdot \nabla \u + \nabla p = - g (1+\gamma c) \textbf{i}_2 + \textbf{f}, \quad  x \in \Omega, \, t>0,\\
	\nabla \cdot \u = 0, \quad  x \in \Omega, \, t>0, \\
	\frac{\partial c}{\partial t} -  \div (\theta(c) \nabla c) + \u \cdot \nabla c + U \frac{\partial c}{\partial x_2}=g, \quad  x \in \Omega, \, t>0.
\end{align}

in $\Omega \times [0, T]$, where $\Omega$ is the unit  square:
$$ \Omega = \{ (x,y)\in \mathbb{R}^2 : 0< x <1, 0<y<1   \}.$$

We set the final time $T=1$, $\theta=1$. To select the approximate functions $g$ and $\f$, we determine the exact solution  $(\sigma,\u,p)$ as follows
 \begin{align}\left\{\begin{aligned}
		&\u(x,y,t)=(yexp(-t)(2y - 1)(y - 1),  -xexp(-t)(2x - 1)(x - 1))^{T},\\
		&p(x,y,t)= exp(-t)(2x - 1)(2y - 1),\\
		& c(x,y,t)=exp(-t)sin(\pi x)sin(\pi y),
	\end{aligned}\right.\end{align}
Denote
\begin{align*}
	\|r - r_h\|_{L^2} &= \|r(t_N) - r_h^N\|_{L^2}, \\
	\|\mathbf{v} - \mathbf{v}_h\|_{L^2} &= \|\mathbf{v}(t_N) - \mathbf{v}_h^N\|_{L^2}, \\
	\|q - q_h\|_{l_2(L^2)} &= \left(\tau \sum_{n=1}^{N} \|q(t_n) - q_h^n\|_{L^2}^2\right)^{1/2}.
\end{align*}

Numerical results are showed by taking different grid size $ h = \frac{1}{4}, \frac{1}{8}, \dots, \frac{1}{128}$, the meshes are given from the uniform triangles meshes. In addition, we adapt different forms for $\nu$ by $\nu =1, 1+0.1 c, exp(c)$.  For the decoupled BDF2 finite element method, Stability results are presented in Table \ref{stability1}, \ref{stability-2}, \ref{stability3}. The convergent rate for the velocity, the concentration and pressure in $L^2$-norm are shown in Table \ref{l2rate1}, \ref{l2rate2}, \ref{l2rate3} and illustrated graphically in Figure \ref{plot1},\ref{plot2},\ref{plot3}, the convergent rate for the velocity and the concentration in $H^1$-norm are shown in Table \ref{h1rate1}, \ref{h1rate2}, \ref{h1rate3}, it is clear that the second-order convergence rate $\mathcal{O}(h^2)$ in $L^2$-norm and $\mathcal{O}(h^1)$ in $H^1$-norm. The relative error and convergent rate for the velocity, the concentration and pressure  are shown in \ref{relative1},  \ref{relative2}, \ref{relative3}. In addition, we present the numerical solutions for  the velocity, the concentration and pressure with $\nu=1+0.1c$ in Figure \ref{decoupledvelocity}, \ref{decoupledcentration}, \ref{decoupledpressure} and $\nu=exp(c)$ in Figure  \ref{coupledvelocity}, \ref{coupledcentration}, \ref{coupledpressure}. In conclusion, all numerical results and tests have well verified the effectiveness and accuracy of the proposed algorithm.

	\begin{figure}
	\centering
	\includegraphics[width=0.6\textwidth]{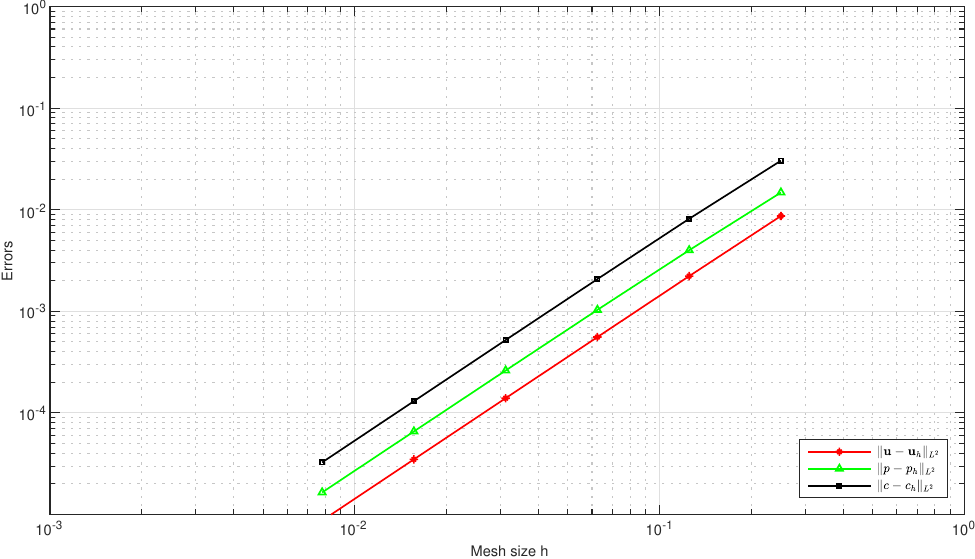}
	\caption{Convergence history of $(\u,p,c)$ for $\nu=1$.}
	\label{plot1}
\end{figure}
	\begin{figure}
	\centering
	\includegraphics[width=0.6\textwidth]{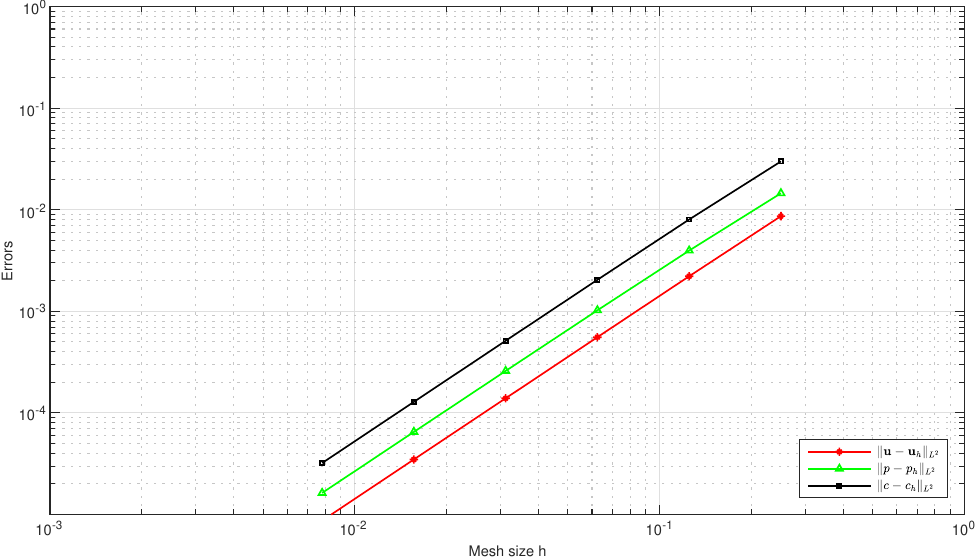}
	\caption{Convergence history of $(\rho,\sigma,\u,p)$ for $\nu=1+0.1c$.}
	\label{plot2}
\end{figure}
	\begin{figure}
	\centering
	\includegraphics[width=0.6\textwidth]{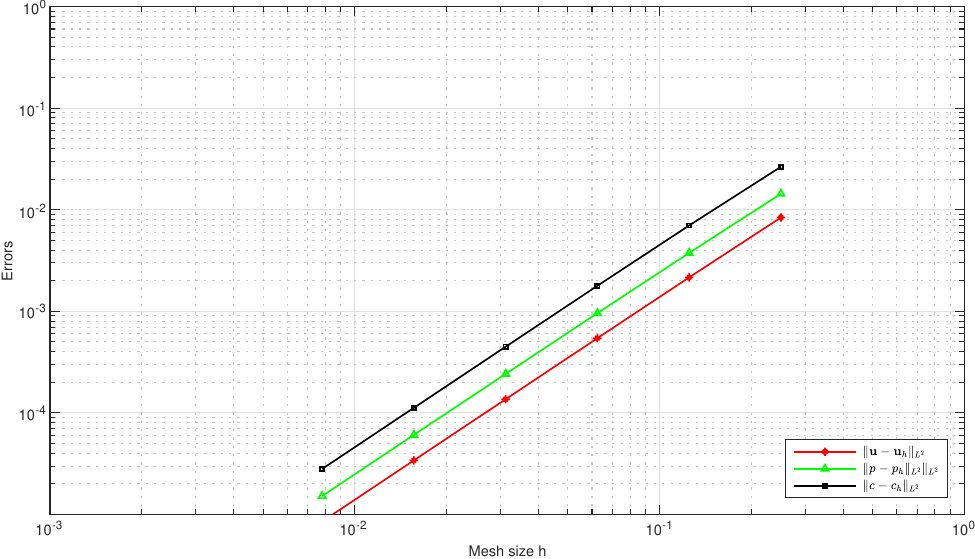}
	\caption{Convergence history of $(\rho,\sigma,\u,p)$ for $\nu=exp(c)$.}
	\label{plot3}
\end{figure}
\begin{table}[htbp]  
	\centering  
	\caption{Stability result with $\nu=1$ }  
		 \setlength{\tabcolsep}{5.8mm}{  
	\begin{tabular}{cccccc}         
		   	                \hline\hline
		 $\tau=h$& $\| \u_h\|_{L^2} $   & $\| \u_h\|_{H^1}$   & $\| c_h\|_{L^2} $  & $\| c_h\|_{H^1}$   & $\| p_h\|_{L^2} $  \\          
		 \hline 
		 1/4   & 0.0291664 & 0.0350738 & 0.156795 & 0.749873 & 0.123623  \\    
		 1/8   & 0.0342092 & 0.0177303 & 0.176622 & 0.799815 & 0.1227\\    
		 1/16  & 0.035478 & 0.0088942 & 0.182072 & 0.812822 & 0.122632 \\    
		 1/32  & 0.0357955 & 0.0044511 & 0.18347 & 0.81612 & 0.122627 \\    
		 1/64  & 0.0358749 & 0.0022261 & 0.183822 & 0.816947 & 0.122627 \\    
		 1/128 & 0.0358947 & 0.0011131 & 0.18391 & 0.817154 & 0.122626    \\  
		      	                \hline\hline
		  \end{tabular}  }
		   \label{stability1}%
		   \end{table}%
		   
		   \begin{table}[htbp] 
		   	 \centering 
		   	  \caption{Numerical errors and convergence rates of $(\u,p,c)$  in $L^2$-norm with $\nu=1$}  
		   	  \label{l2rate1}
		   	  		 \setlength{\tabcolsep}{5mm}{  
		   	    \begin{tabular}{ccccccc}         
		   	    	\hline\hline 
		   	    	$\tau=h$&  $\|\u -\u_h\|_{L^2}$ & rate  &$\| c -c_h\|_{L^2}$ & rate  & $\| p -p_h\|_{L^2}$   & rate \\   
		   	    	\hline
		   	    	  1/4   & 0.008674 &       & 0.0303567 &       & 0.0147581 &  \\   
		   	    	   1/8   & 0.0022186 & 1.97  & 0.0081386 & 1.90  & 0.0039903 & 1.89  \\   
		   	    	    1/16  & 0.0005575 & 1.99  & 0.0020749 & 1.97  & 0.0010329 & 1.95  \\ 
		   	    	       1/32  & 0.0001395 & 2.00  & 0.0005213 & 1.99  & 0.0002614 & 1.98  \\   
		   	    	        1/64  & 3.49E-05 & 2.00  & 0.0001305 & 2.00  & 6.57E-05 & 1.99  \\ 
		   	    	           1/128 & 8.72E-06 & 2.00  & 3.26E-05 & 2.00  & 1.65E-05 & 2.00  \\  
		   	    	    \hline\hline
		   	    	    \end{tabular}}
		   	    	    \end{table}%

   \begin{table}[htbp] 
    	    	 \centering 
 	    	  \caption{Numerical errors and convergence rates of $(\u,p,c)$ in $H^1$-norm with $\nu=1$} 
 	    	  \label{h1rate1}  
    	    	  	 \setlength{\tabcolsep}{9mm}{  
 	    	  \begin{tabular}{ccccc}          
  	    	    	  	   	    	\hline\hline 
  	    	    	  $\tau=h$	& $\|\u -\u_h\|_{H^1}$ & rate  & $\|c -c_h\|_{H^1}$& rate \\    
  	    	    	  	   	    	\hline
  	    	    	  1/4   & 0.116999 &       & 0.308563 &  \\    1/8   & 0.0595905 & 0.97  & 0.158861 & 0.96  \\    1/16  & 0.0299324 & 0.99  & 0.0800286 & 0.99  \\    1/32  & 0.0149835 & 1.00  & 0.04009 & 1.00  \\    1/64  & 0.0074939 & 1.00  & 0.0200545 & 1.00  \\    1/128 & 0.0037472 & 1.00  & 0.0100284 & 1.00  \\   
  	    	    	  	    	    	\hline\hline 
   	  	   \end{tabular}}
    	    	  	 \end{table}%

\begin{table}[htbp]  
	\centering  
	\caption{Relative errors and convergence rates of $(\u,p,c)$ with $\nu=1$}   
	\label{relative1}
	 \setlength{\tabcolsep}{5.3mm}{  
	 \begin{tabular}{ccccccc}    
	 		\hline\hline       
	 $\tau=h$	& $\frac{\| \u-\u_h\|_{L^2}}{\| \u\|_{L^2}}$ & rate  &$\frac{\| c-c_h\|_{L^2}}{\| c \|_{L^2}}$   & rate  &$\frac{\| p-p_h\|_{L^2}}{\| p\|_{L^2}}$   & rate \\   
	 		\hline 
	 	1/4   & 0.297397 &       & 0.193607 &       & 0.11938 &  \\    
	 	1/8   & 0.0648539 & 2.20  & 0.0460789 & 2.07  & 0.0325203 & 1.88  \\    
	 	1/16  & 0.0157137 & 2.05  & 0.0113959 & 2.02  & 0.0084228 & 1.95  \\   
	 	1/32  & 0.0038978 & 2.01  & 0.0028411 & 2.00  & 0.0021313 & 1.98  \\    
	 	1/64  & 0.0009725 & 2.00  & 0.0007098 & 2.00  & 0.0005357 & 1.99  \\    
	 	1/128 & 0.000243 & 2.00  & 0.0001774 & 2.00  & 0.0001343 & 2.00  \\   
	 		\hline\hline  
	 		\end{tabular}}
	 	\end{table}%

	 	\begin{table}[htbp]
	 		\centering  
	 				 \setlength{\tabcolsep}{5.8mm}{  
	 		\caption{Stability result with $\nu=1+0.1c$.}    
	 		 \label{stability-2}
	 		\begin{tabular}{cccccc}      
	 				\hline\hline           
	 			 $\tau=h$& $\| \u_h\|_{L^2} $   & $\| \u_h\|_{H^1}$   & $\| c_h\|_{L^2} $  & $\| c_h\|_{H^1}$   & $\| p_h\|_{L^2} $  \\          
	 				\hline 
	 			1/4   & 0.02921 & 0.0350731 & 0.157375 & 0.75264 & 0.123597 \\    
	 			1/8   & 0.0342185 & 0.0177302 & 0.176771 & 0.800485 & 0.1227 \\    
	 			1/16  & 0.0354801 & 0.0088942 & 0.182111 & 0.812995 & 0.122632 \\    1/32  & 0.035796 & 0.0044511 & 0.18348 & 0.816163 & 0.122627 \\    
	 			1/64  & 0.035875 & 0.0022261 & 0.183825 & 0.816958 & 0.122627 \\    1/128 & 0.0358948 & 0.0011131 & 0.183911 & 0.817156 & 0.122626 \\   
	 				\hline\hline  
	 				 \end{tabular}}
	 			\end{table}%

	 			\begin{table}[htbp] 
	 				\centering 
	 				\caption{Numerical errors and convergence rates of $(\u,p,c)$  in $L^2$-norm with $\nu=1+0.1c$.}  
	 					   	  \label{l2rate2}
	 				\setlength{\tabcolsep}{5mm}{  
	 					\begin{tabular}{ccccccc}         
	 						\hline\hline 
	 						$\tau=h$&  $\|\u -\u_h\|_{L^2}$ & rate  &$\| c -c_h\|_{L^2}$ & rate  & $\| p -p_h\|_{L^2}$   & rate \\   
	 						\hline
	 					  1/4   & 0.0086415 &       & 0.0298615 &       & 0.0145104 &  \\    
	 					  1/8   & 0.0022125 & 1.97  & 0.0080102 & 1.90  & 0.0039585 & 1.87  \\  
	 					    1/16  & 0.0005562 & 1.99  & 0.0020413 & 1.97  & 0.0010226 & 1.95  \\  
	 					      1/32  & 0.0001392 & 2.00  & 0.0005128 & 1.99  & 0.0002587 & 1.98  \\   
	 					       1/64  & 3.48E-05 & 2.00  & 0.0001284 & 2.00  & 6.50E-05 & 1.99  \\   
	 					        1/128 & 8.70E-06 & 2.00  & 3.21E-05 & 2.00  & 1.63E-05 & 2.00  \\  
	 						\hline\hline
	 				\end{tabular}}
	 			\end{table}%

\begin{table}[htbp] 
	\centering 
	\caption{Numerical errors and convergence rates of $(\u,p,c)$ in $H^1$-norm with $\nu=1+0.1c$.}   
	 \label{h1rate2}  
	\setlength{\tabcolsep}{9mm}{  
		\begin{tabular}{ccccc}          
			\hline\hline 
			$\tau=h$	& $\|\u -\u_h\|_{H^1}$ & rate  & $\|c -c_h\|_{H^1}$& rate \\    
			\hline
		1/4   & 0.116986 &       & 0.308516 &  \\    1/8   & 0.0595896 & 0.97  & 0.158855 & 0.96  \\    1/16  & 0.0299323 & 0.99  & 0.0800279 & 0.99  \\    1/32  & 0.0149835 & 1.00  & 0.0400899 & 1.00  \\    1/64  & 0.0074939 & 1.00  & 0.0200545 & 1.00  \\    1/128 & 0.0037472 & 1.00  & 0.0100284 & 1.00  \\  
			\hline\hline 
	\end{tabular}}
\end{table}%

\begin{table}[htbp]  
	\centering  
	\caption{Relative errors and convergence rates of $(\u,p,c)$ with $\nu=1+0.1c$.}   
		\label{relative2}
	\setlength{\tabcolsep}{5.3mm}{  
		\begin{tabular}{ccccccc}    
			\hline\hline       
			$\tau=h$	& $\frac{\| \u-\u_h\|_{L^2}}{\| \u\|_{L^2}}$ & rate  &$\frac{\| c-c_h\|_{L^2}}{\| c \|_{L^2}}$   & rate  &$\frac{\| p-p_h\|_{L^2}}{\| p\|_{L^2}}$   & rate \\   
			\hline 
		  1/4   & 0.295839 &       & 0.189747 &       & 0.117401 &  \\    1/8   & 0.0646581 & 2.19  & 0.045314 & 2.07  & 0.0322619 & 1.86  \\    1/16  & 0.0156756 & 2.04  & 0.011209 & 2.02  & 0.0083389 & 1.95  \\    1/32  & 0.0038891 & 2.01  & 0.0027949 & 2.00  & 0.0021095 & 1.98  \\    1/64  & 0.0009704 & 2.00  & 0.0006983 & 2.00  & 0.0005302 & 1.99  \\    1/128 & 0.0002425 & 2.00  & 0.0001745 & 2.00  & 0.0001329 & 2.00  \\  
			\hline\hline  
	\end{tabular}}
\end{table}%

	\begin{table}[htbp]  
	\centering  
	\setlength{\tabcolsep}{5.8mm}{  
		\caption{Stability result with $\nu = exp(c)$.}    
			 \label{stability3}%
		\begin{tabular}{cccccc}      
			\hline\hline           
			$\tau=h$& $\| \u_h\|_{L^2} $   & $\| \u_h\|_{H^1}$   & $\| c_h\|_{L^2} $  & $\| c_h\|_{H^1}$   & $\| p_h\|_{L^2} $  \\          
			\hline 
		 1/4   & 0.0296026 & 0.0352432 & 0.161541 & 0.772549 & 0.123604 \\    1/8   & 0.0343059 & 0.0177448 & 0.177988 & 0.805997 & 0.122697 \\    1/16  & 0.0355003 & 0.0088958 & 0.182422 & 0.814382 & 0.122632 \\    1/32  & 0.0358009 & 0.0044513 & 0.183558 & 0.816509 & 0.122627 \\    1/64  & 0.0358762 & 0.0022261 & 0.183844 & 0.817044 & 0.122627 \\    1/128 & 0.0358951 & 0.0011131 & 0.183916 & 0.817178 & 0.122626 \\ 
			\hline\hline  
	\end{tabular}}
\end{table}%

\begin{table}[htbp] 
	\centering 
	\caption{Numerical errors and convergence rates of $(\u,p,c)$  in $L^2$-norm with $\nu =  exp(c)$.}  
		   	  \label{l2rate3}
	\setlength{\tabcolsep}{5mm}{  
		\begin{tabular}{ccccccc}         
			\hline\hline 
			$\tau=h$&  $\|\u -\u_h\|_{L^2}$ & rate  &$\| c -c_h\|_{L^2}$ & rate  & $\| p -p_h\|_{L^2}$   & rate \\   
			\hline
		1/4   & 0.0083743 &       & 0.0264466 &       & 0.0143339 &  \\    
		1/8   & 0.0021562 & 1.96  & 0.0069973 & 1.92  & 0.0037548 & 1.93  \\   
		 1/16  & 0.0005439 & 1.99  & 0.0017822 & 1.97  & 0.0009591 & 1.97  \\    
		 1/32  & 0.0001363 & 2.00  & 0.000448 & 1.99  & 0.000242 & 1.99  \\   
		  1/64  & 3.41E-05 & 2.00  & 0.0001122 & 2.00  & 6.08E-05 & 1.99  \\  
		    1/128 & 8.53E-06 & 2.00  & 2.81E-05 & 2.00  & 1.52E-05 & 2.00  \\ 
			\hline\hline
	\end{tabular}}
\end{table}%

\begin{table}[htbp] 
	\centering 
	\caption{Numerical errors and convergence rates of $(\u,p,c)$ in $H^1$-norm with $\nu = exp(c)$.}   
	 \label{h1rate3}  
	\setlength{\tabcolsep}{9mm}{  
		\begin{tabular}{ccccc}          
			\hline\hline 
			$\tau=h$	& $\|\u -\u_h\|_{H^1}$ & rate  & $\|c -c_h\|_{H^1}$& rate \\    
			\hline
		1/4   & 0.117002 &       & 0.309078 &  \\    1/8   & 0.059592 & 0.97  & 0.158943 & 0.96  \\    1/16  & 0.0299327 & 0.99  & 0.0800392 & 0.99  \\    1/32  & 0.0149835 & 1.00  & 0.0400913 & 1.00  \\    1/64  & 0.0074939 & 1.00  & 0.0200547 & 1.00  \\    1/128 & 0.0037472 & 1.00  & 0.0100285 & 1.00  \\ 
			\hline\hline 
	\end{tabular}}
\end{table}%

\begin{table}[htbp]  
	\centering  
	\caption{Relative errors and convergence rates of $(\u,p,c)$ with $\nu = exp(c)$.}   
		\label{relative3}
	\setlength{\tabcolsep}{5.3mm}{  
		\begin{tabular}{ccccccc}    
			\hline\hline       
			$\tau=h$	& $\frac{\| \u-\u_h\|_{L^2}}{\| \u\|_{L^2}}$ & rate  &$\frac{\| c-c_h\|_{L^2}}{\| c \|_{L^2}}$   & rate  &$\frac{\| p-p_h\|_{L^2}}{\| p\|_{L^2}}$   & rate \\   
			\hline 
		1/4   & 0.282891 &       & 0.163715 &       & 0.115967 &  \\    1/8   & 0.0628515 & 2.17  & 0.039313 & 2.06  & 0.0306021 & 1.92  \\    1/16  & 0.0153206 & 2.04  & 0.0097695 & 2.01  & 0.0078207 & 1.97  \\    1/32  & 0.0038077 & 2.01  & 0.0024406 & 2.00  & 0.0019735 & 1.99  \\    1/64  & 0.0009507 & 2.00  & 0.0006102 & 2.00  & 0.0004958 & 1.99  \\    1/128 & 0.0002376 & 2.00  & 0.0001526 & 2.00  & 0.0001243 & 2.00  \\  
			\hline\hline  
	\end{tabular}}
\end{table}%

\begin{figure}[htbp] 
	
	\centering 
		\begin{minipage}{0.32\textwidth} 
			\centering  
			\includegraphics[width=\textwidth]{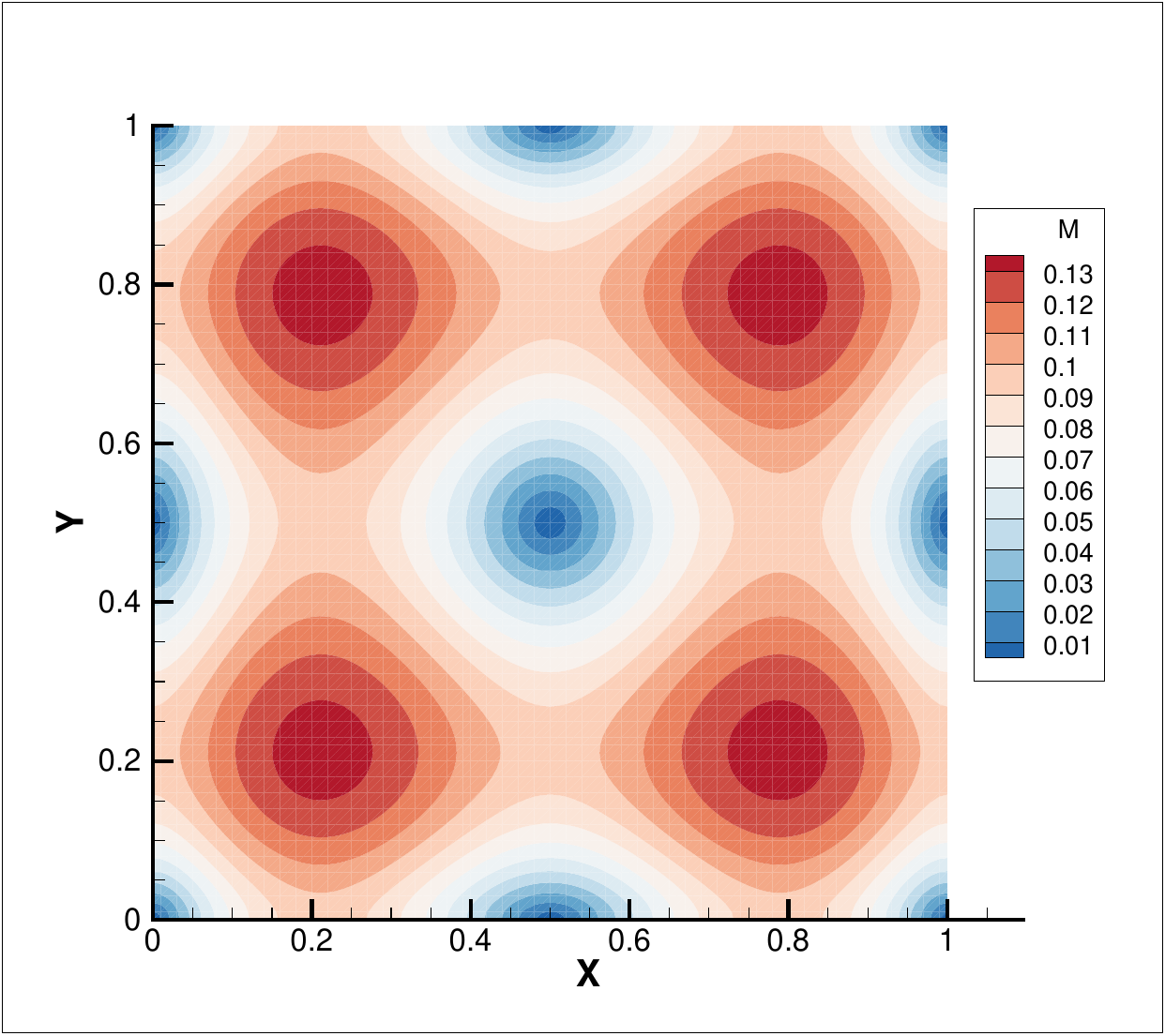}  
		\end{minipage}  
		\begin{minipage}{0.32\textwidth} 
			\centering  
			\includegraphics[width=\textwidth]{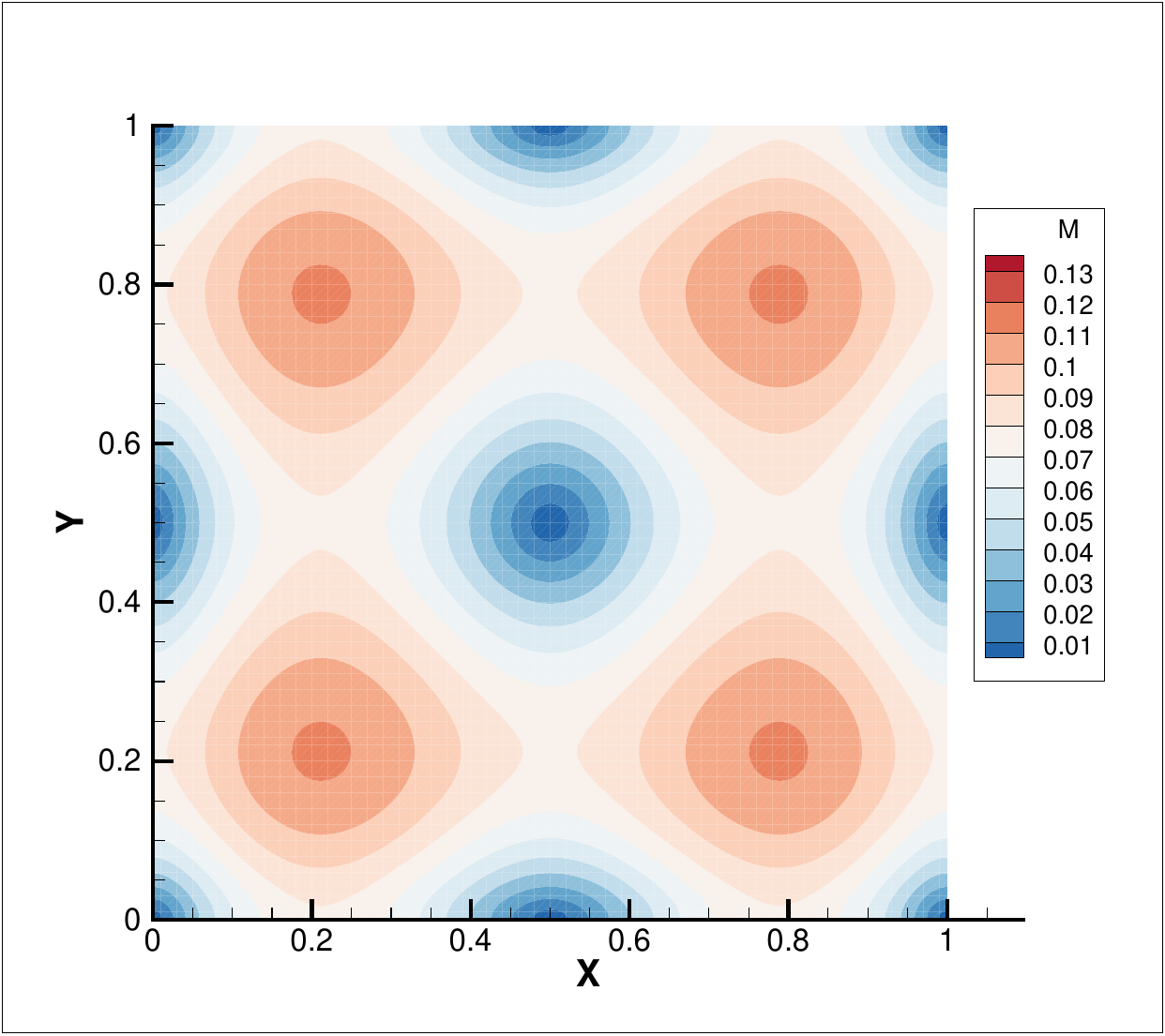} 
		\end{minipage}  
		\begin{minipage}{0.32\textwidth}  
			\centering  
			\includegraphics[width=\textwidth]{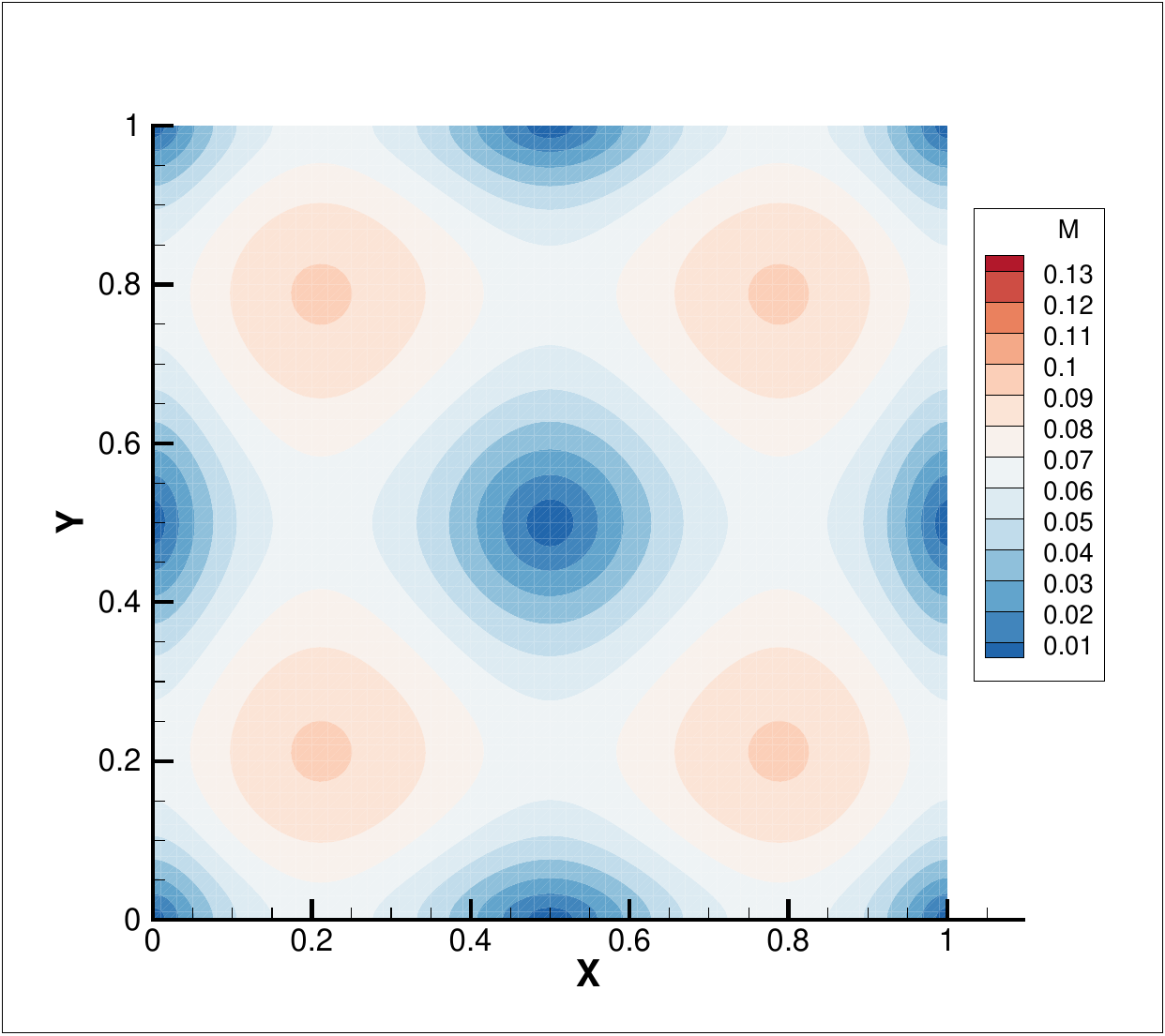}  
		\end{minipage}  
		\begin{minipage}{0.32\textwidth} 
			\centering  
			\includegraphics[width=\textwidth]{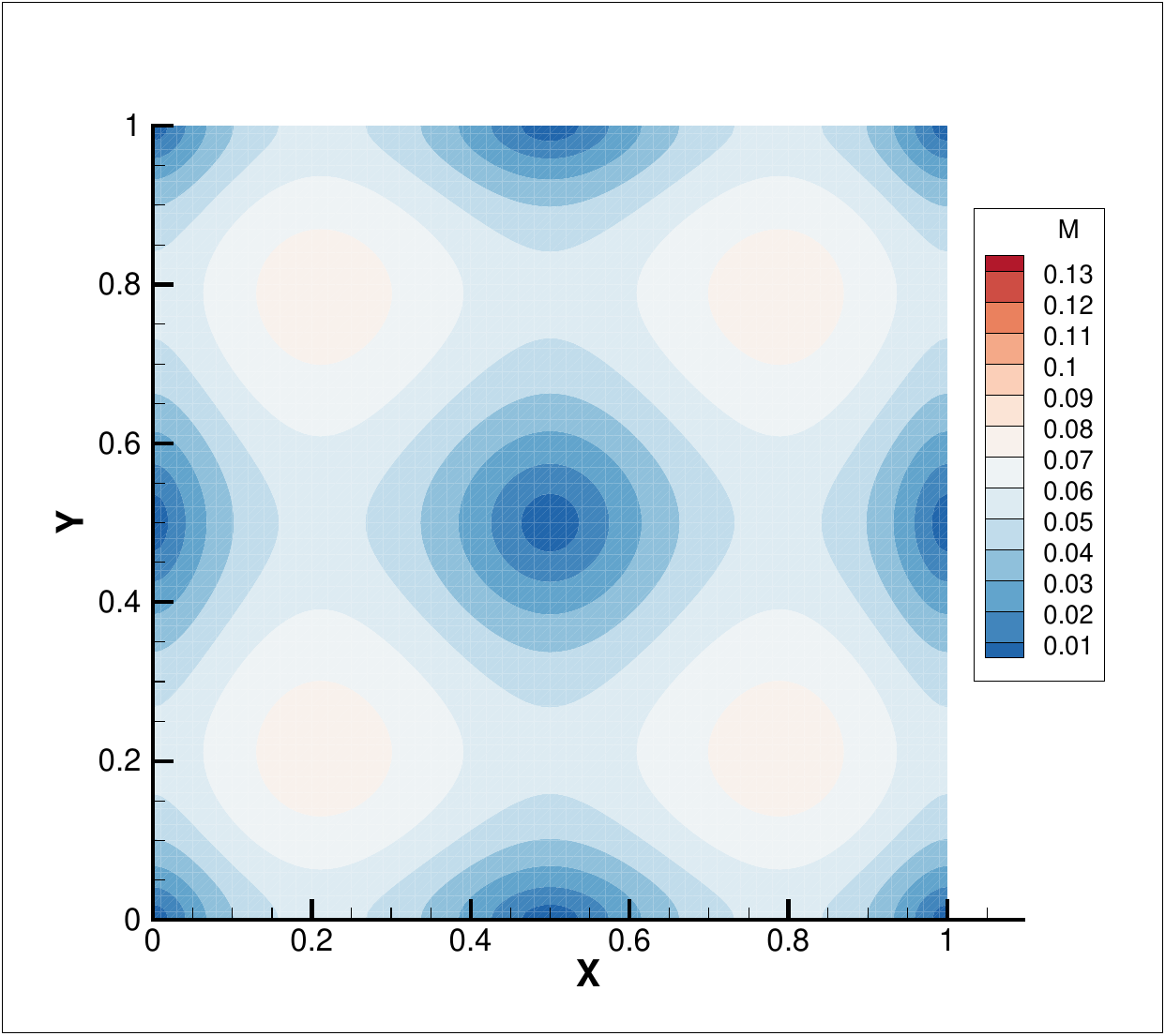}  
		\end{minipage}  
		\begin{minipage}{0.32\textwidth} 
			\centering  
			\includegraphics[width=\textwidth]{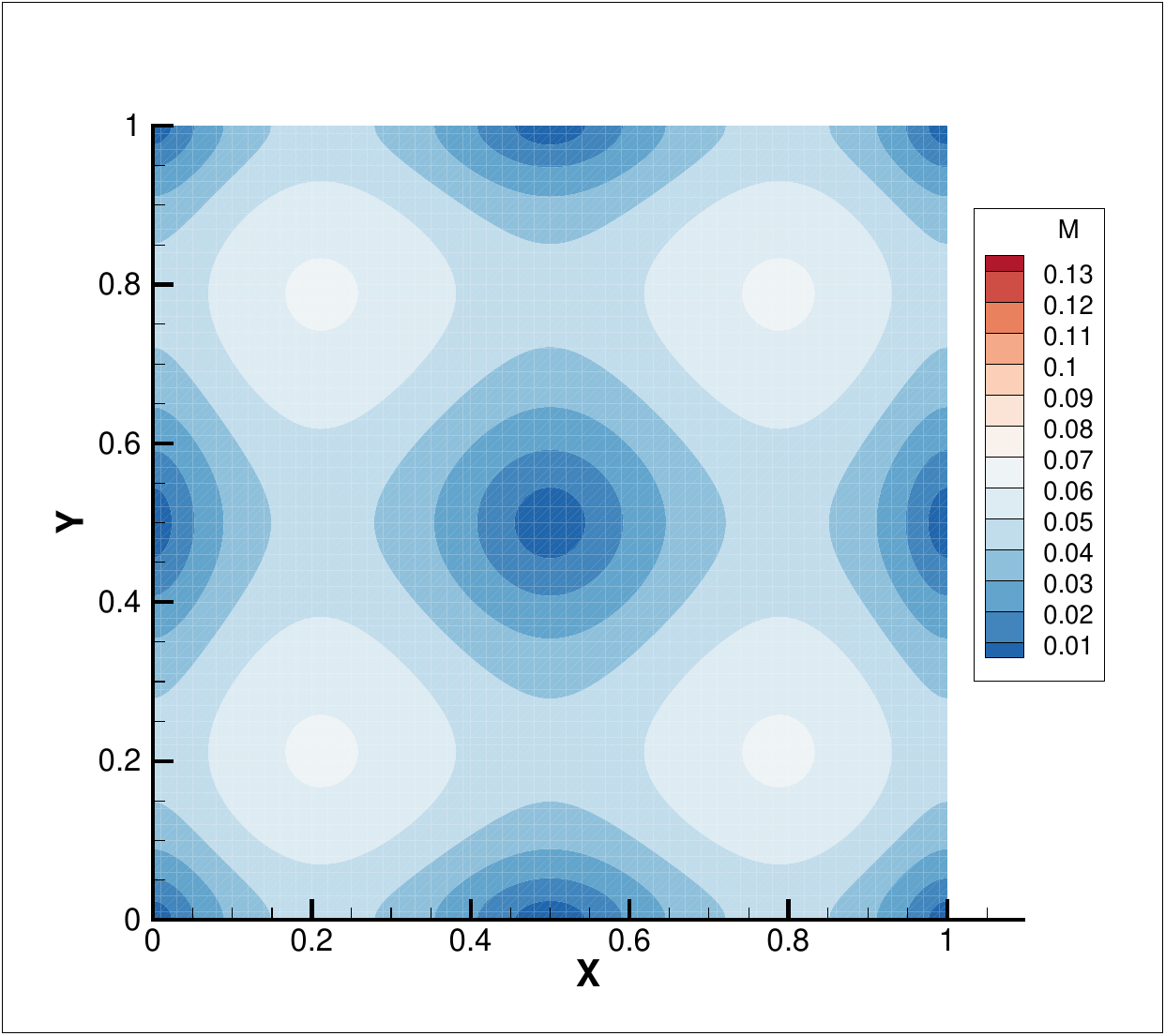} 
		\end{minipage}  
		\begin{minipage}{0.32\textwidth}  
			\centering  
			\includegraphics[width=\textwidth]{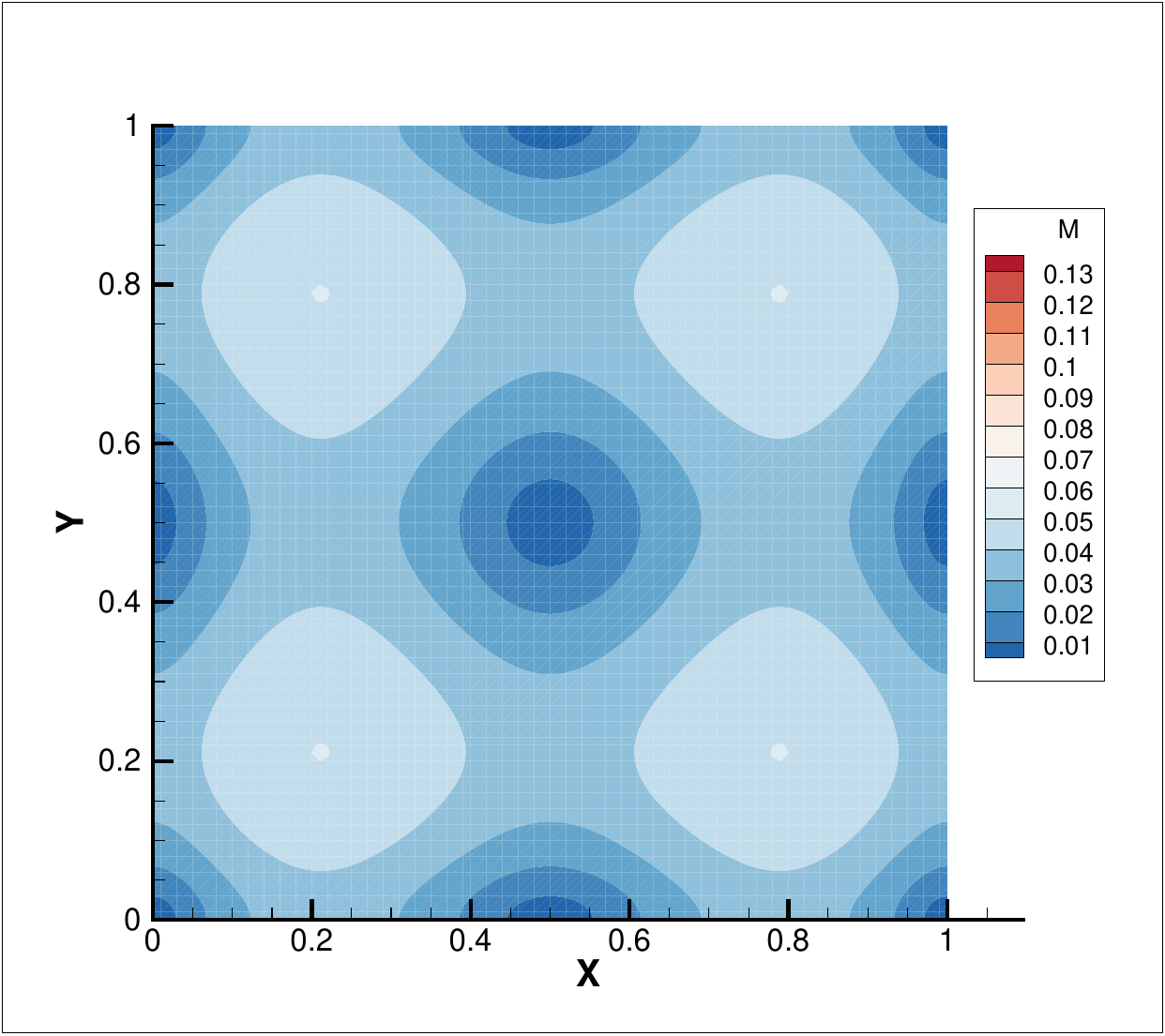}  
		\end{minipage}  
	\caption{Numerical solutions of velocity at times t = 0, 0.2, 0.4, 0.6, 0.8, 1.0 with $\nu=1+0.1c$ for the decoupled finite element method.}  
	\label{decoupledvelocity}  
\end{figure}

\begin{figure}[htbp] 
	
	\centering 

		\begin{minipage}{0.32\textwidth} 
			\centering  
			\includegraphics[width=\textwidth]{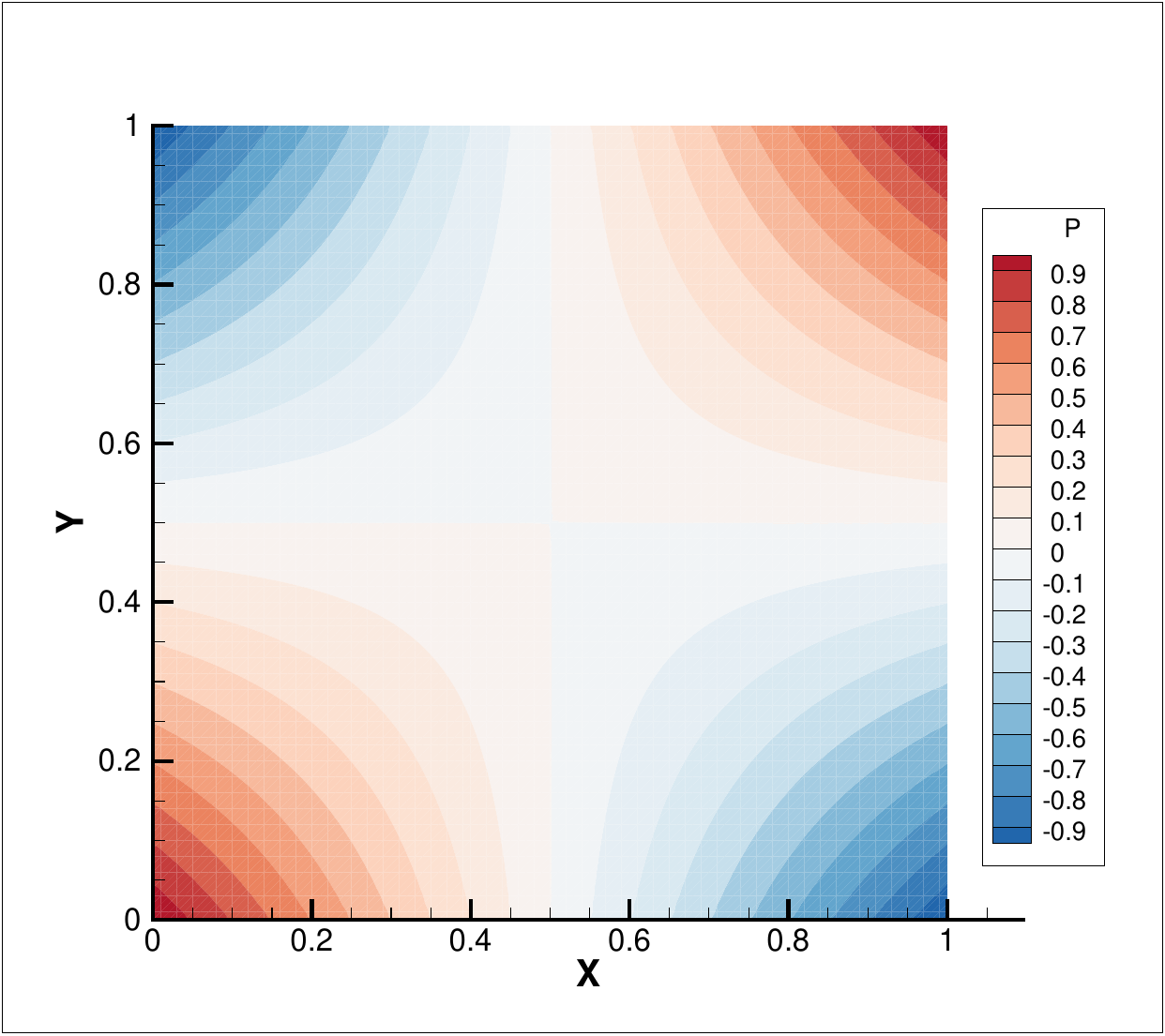}  
		\end{minipage}  
		\begin{minipage}{0.32\textwidth} 
			\centering  
			\includegraphics[width=\textwidth]{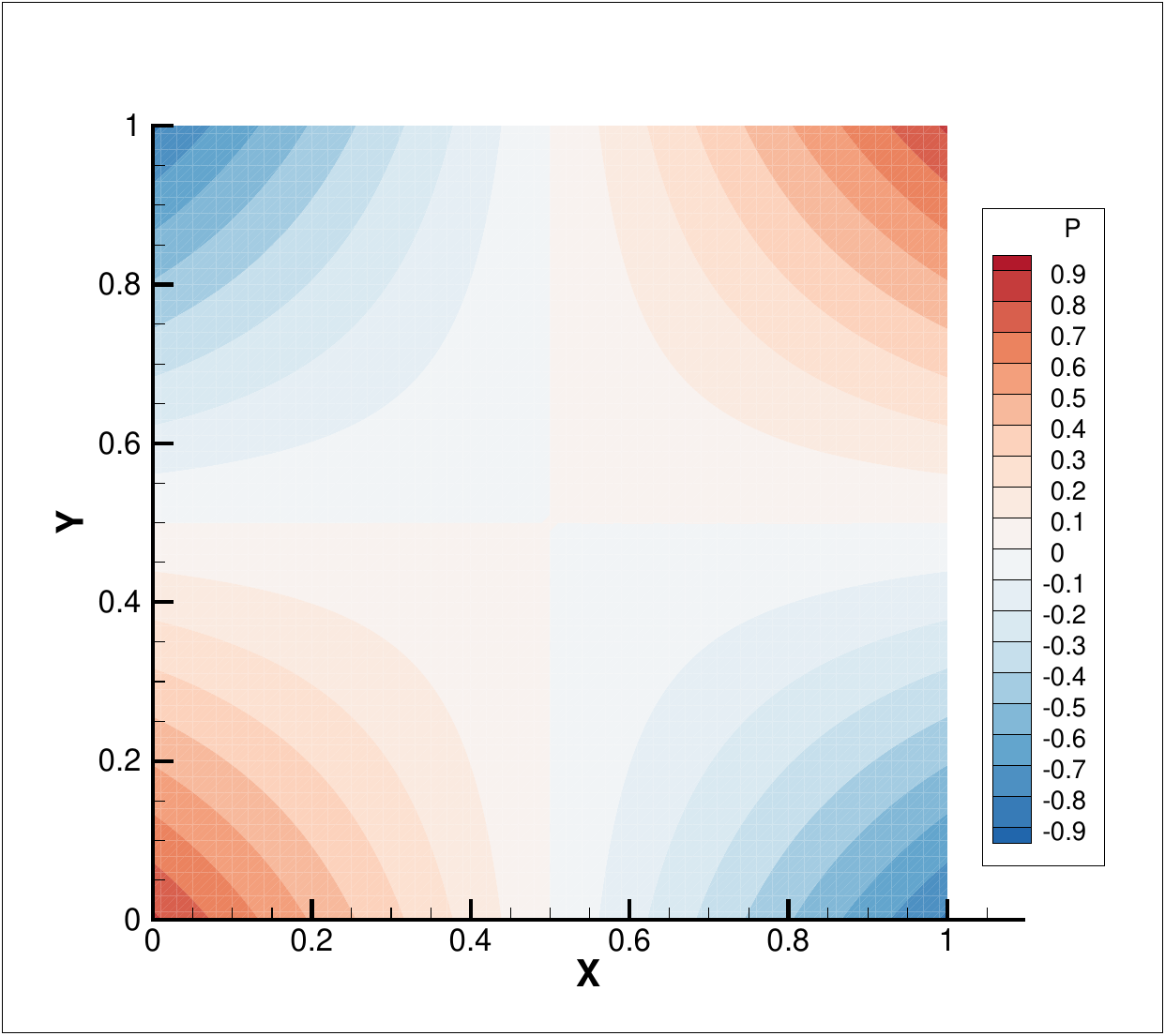} 
		\end{minipage}  
		\begin{minipage}{0.32\textwidth}  
			\centering  
			\includegraphics[width=\textwidth]{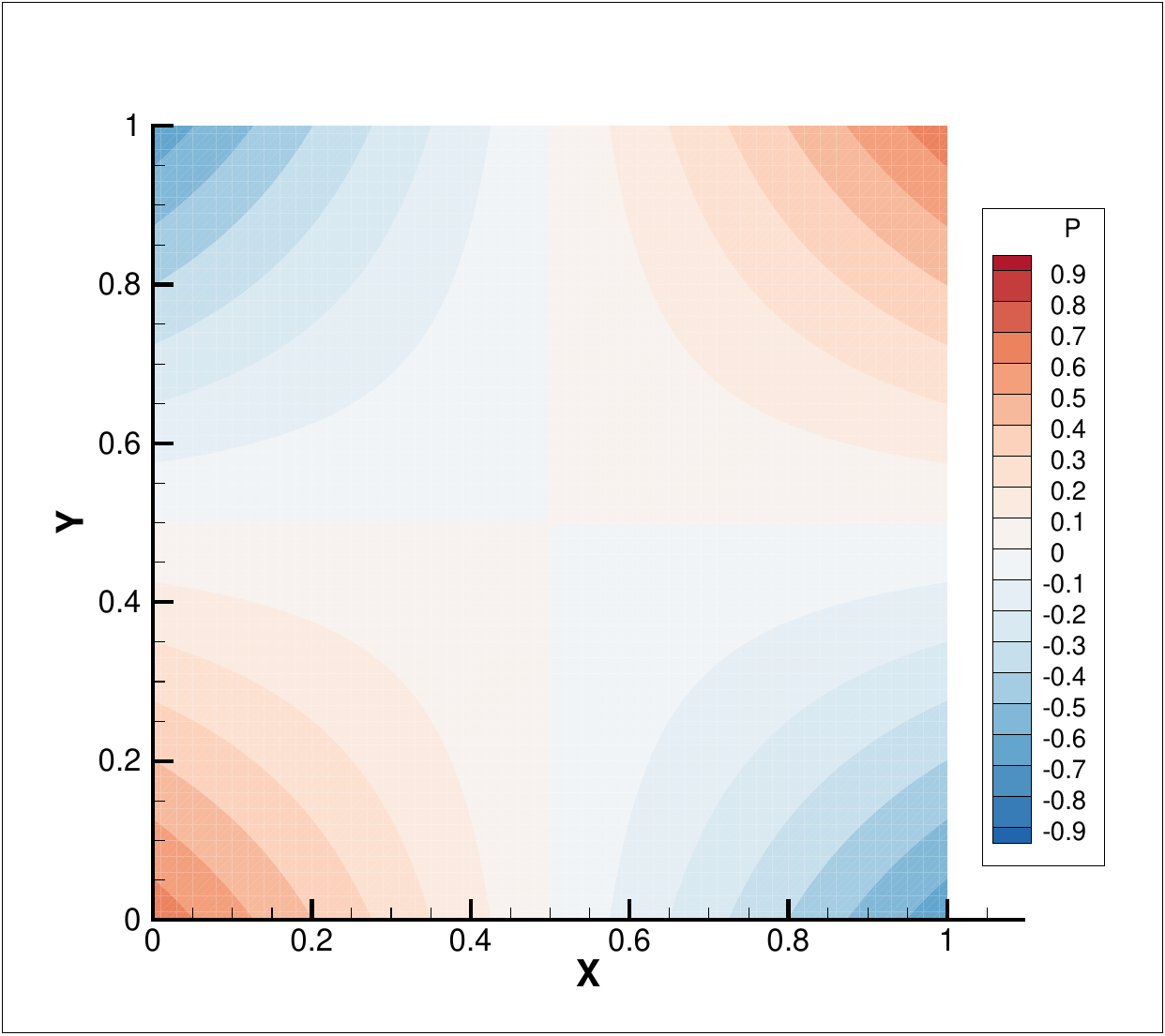}  
		\end{minipage}
		
		\begin{minipage}{0.32\textwidth} 
			\centering  
			\includegraphics[width=\textwidth]{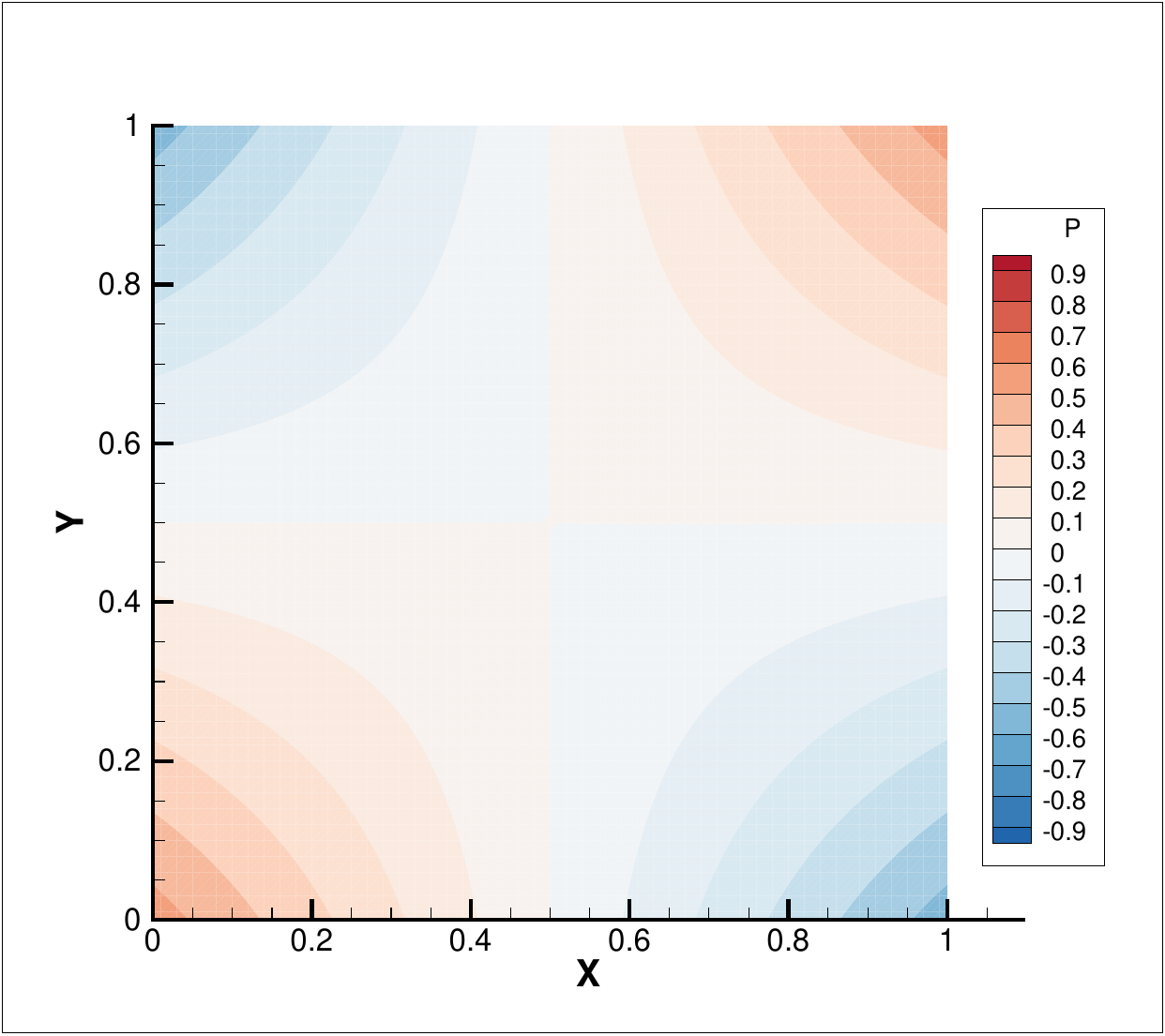}  
		\end{minipage}  
		\begin{minipage}{0.32\textwidth} 
			\centering  
			\includegraphics[width=\textwidth]{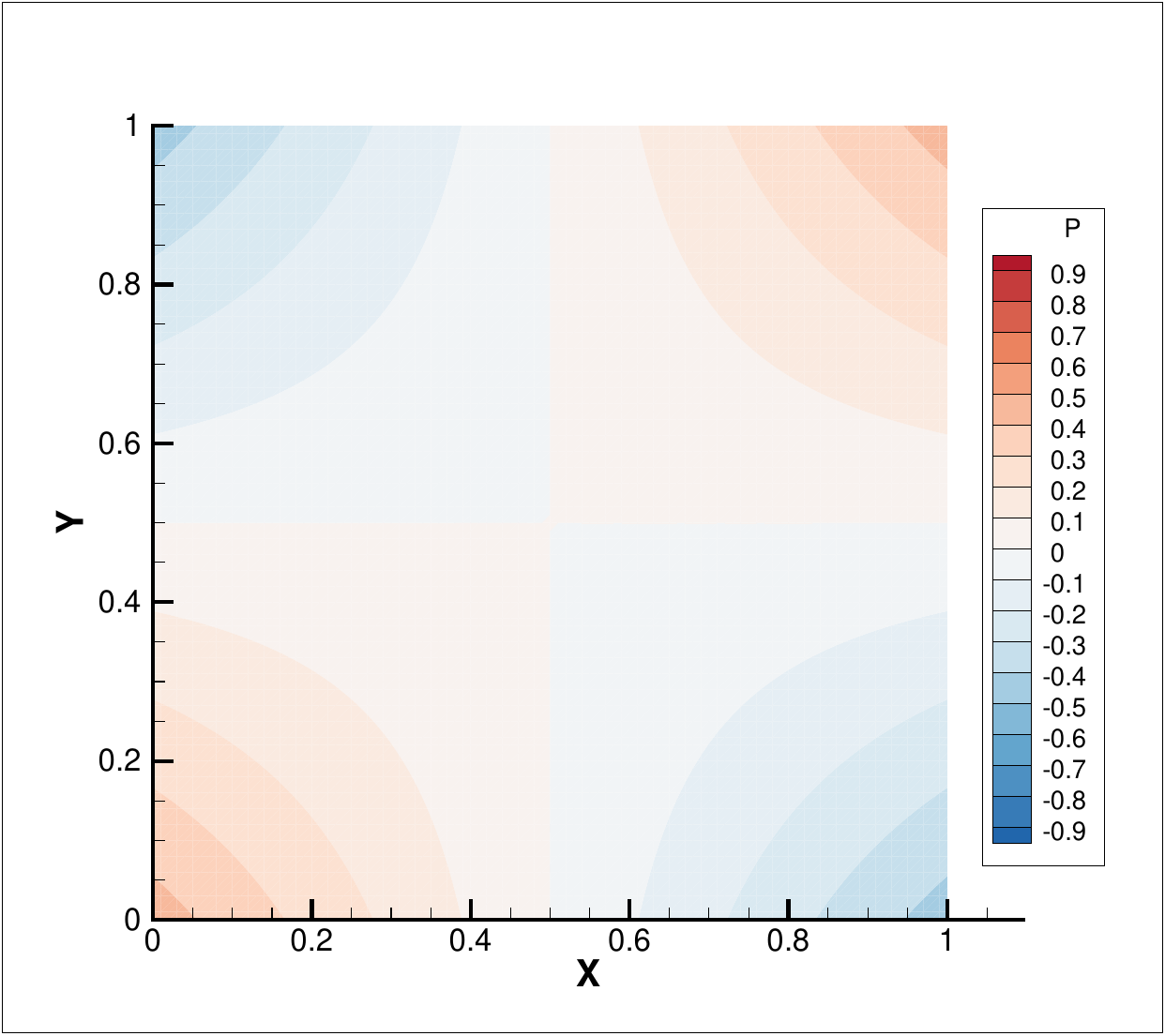} 
		\end{minipage}  
		\begin{minipage}{0.32\textwidth}  
			\centering  
			\includegraphics[width=\textwidth]{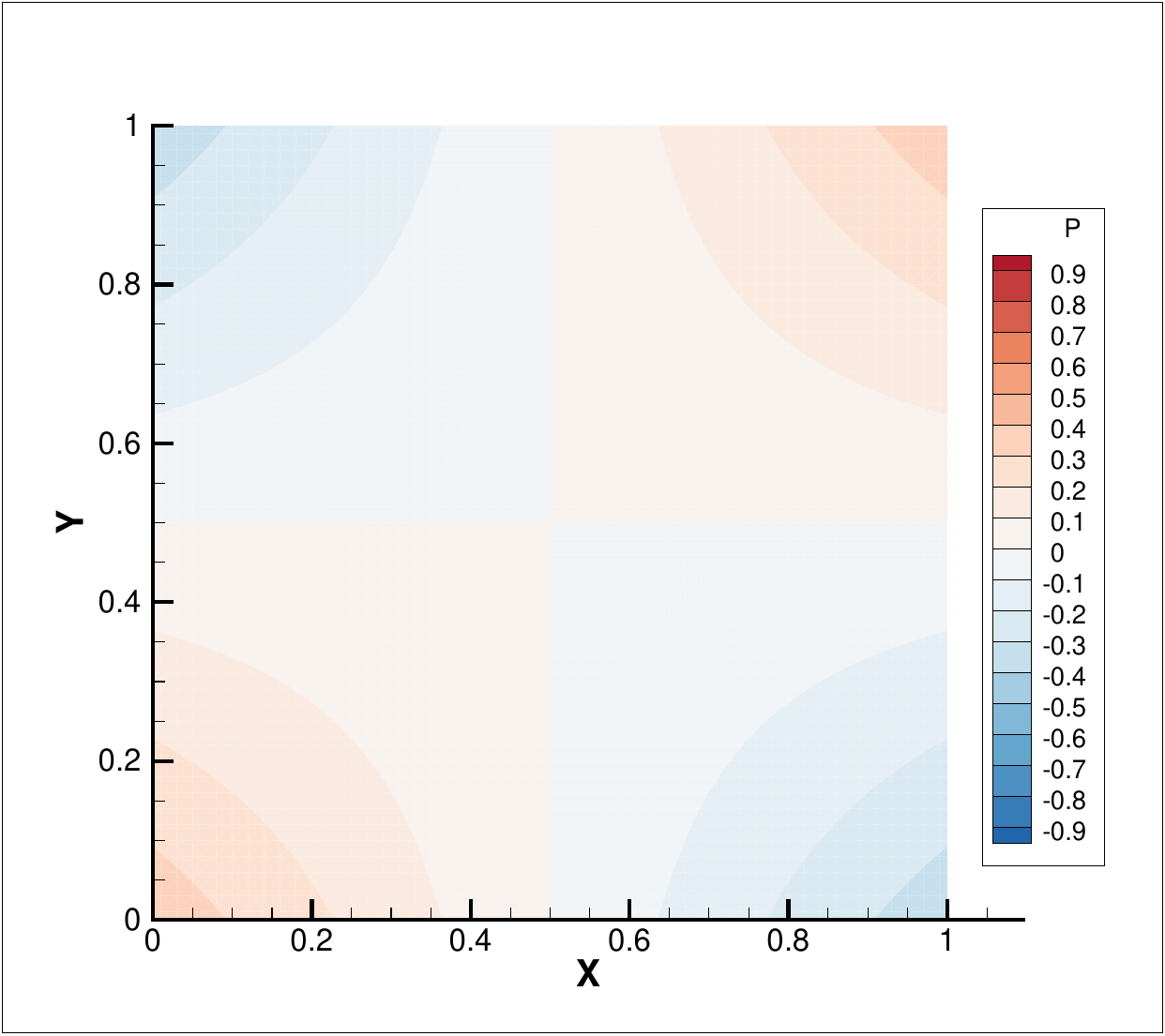}  
		\end{minipage}   
	\caption{Numerical solutions of pressure at times t = 0, 0.2, 0.4, 0.6, 0.8, 1.0 with $\nu=1+0.1c$ for the decoupled finite element method.}  
	\label{decoupledpressure}  
\end{figure}

	\begin{figure}[htbp] 
		
		\centering 
	
			\begin{minipage}{0.32\textwidth} 
				\centering  
				\includegraphics[width=\textwidth]{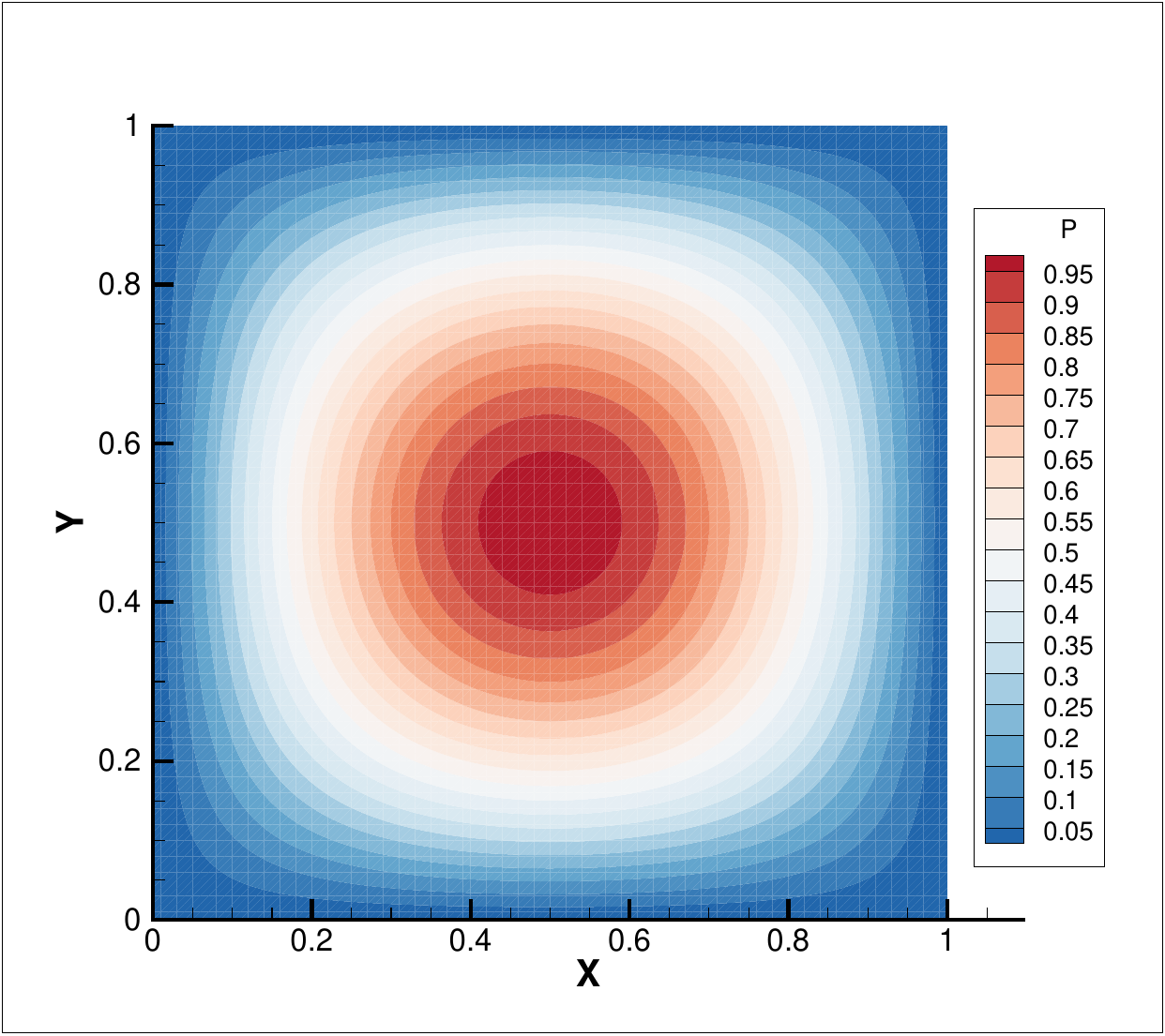}  
			\end{minipage}  
			\begin{minipage}{0.32\textwidth} 
				\centering  
				\includegraphics[width=\textwidth]{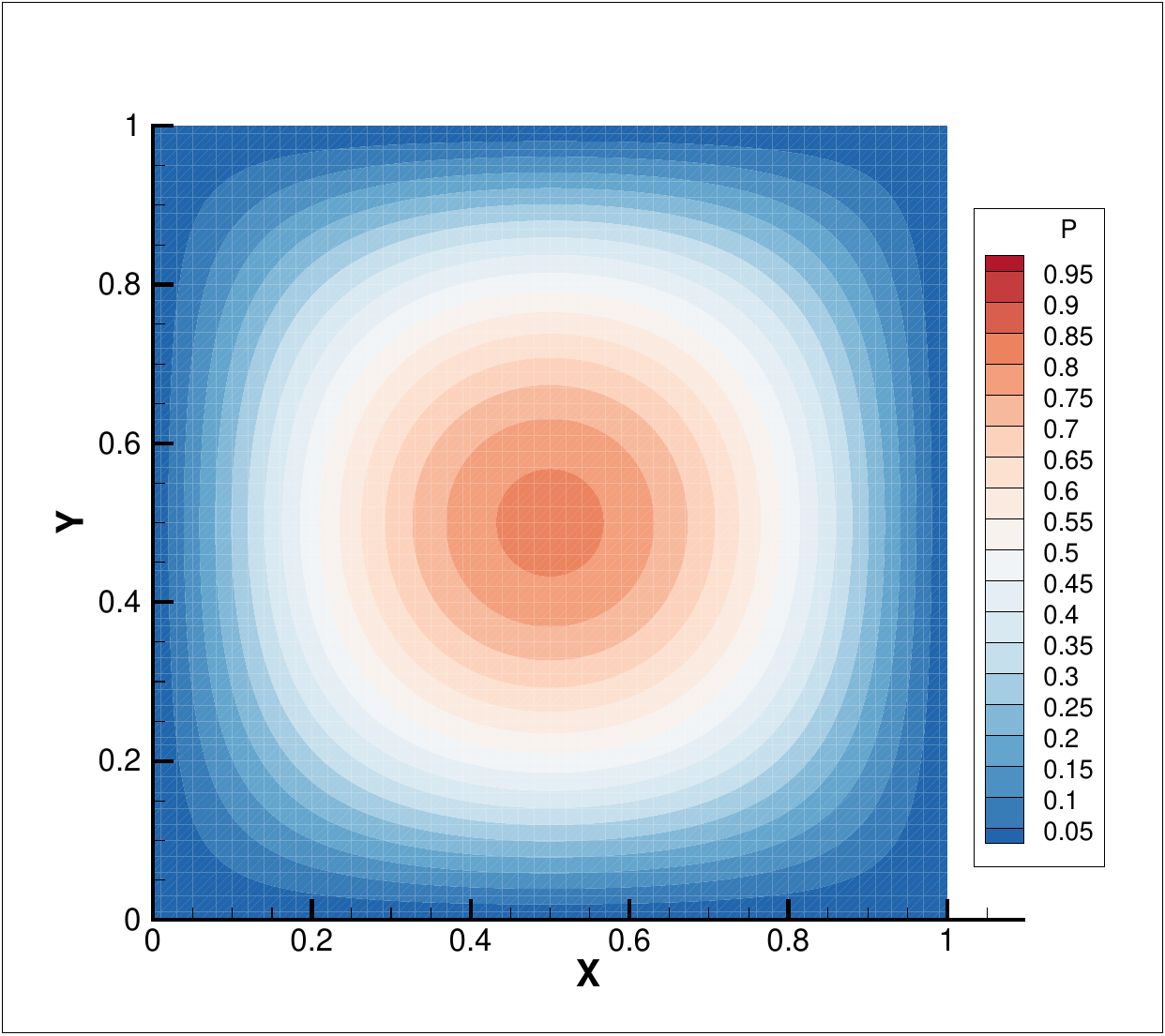} 
			\end{minipage}  
			\begin{minipage}{0.32\textwidth}  
				\centering  
				\includegraphics[width=\textwidth]{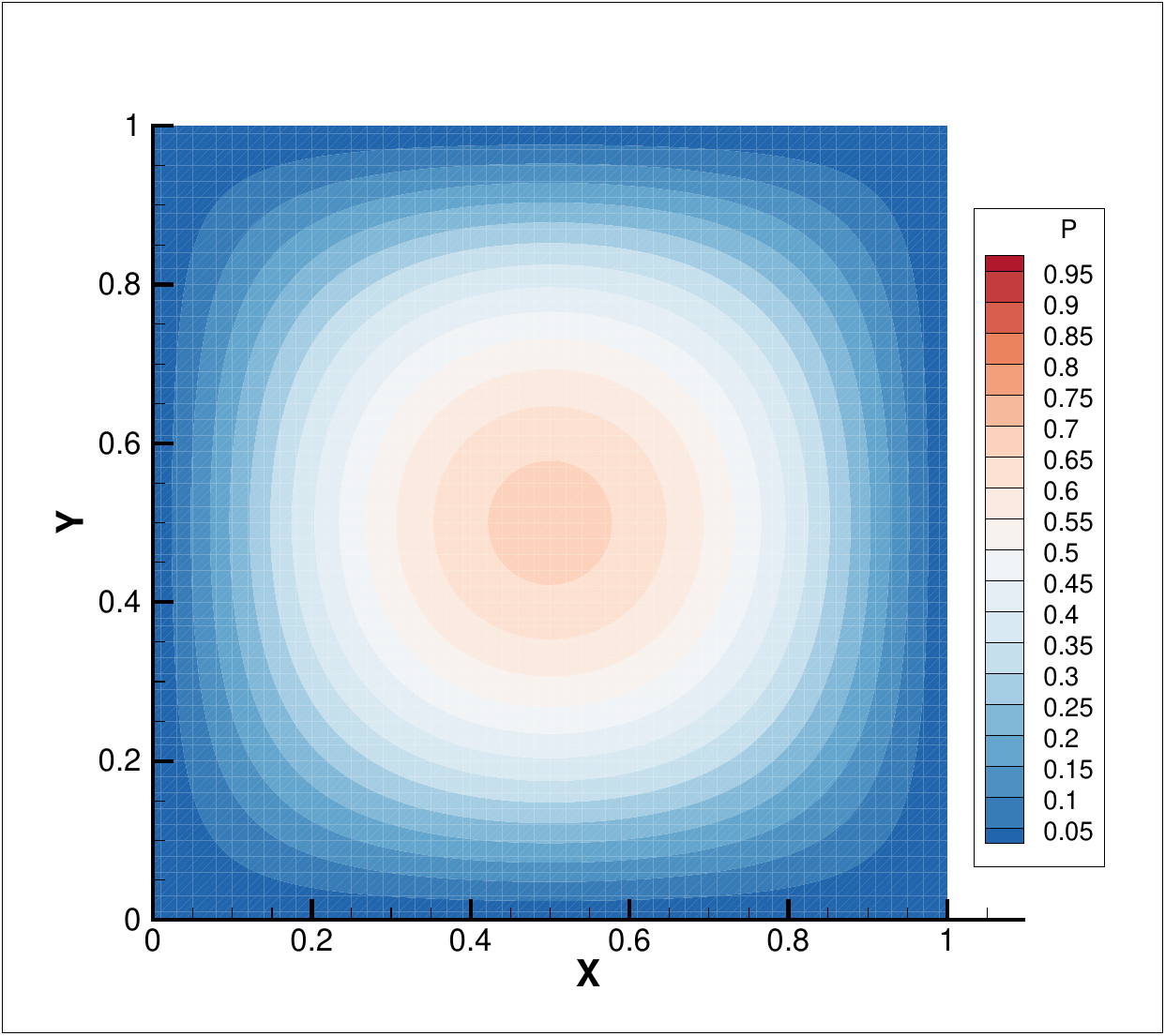}  
			\end{minipage}  
			
			\begin{minipage}{0.32\textwidth} 
				\centering  
				\includegraphics[width=\textwidth]{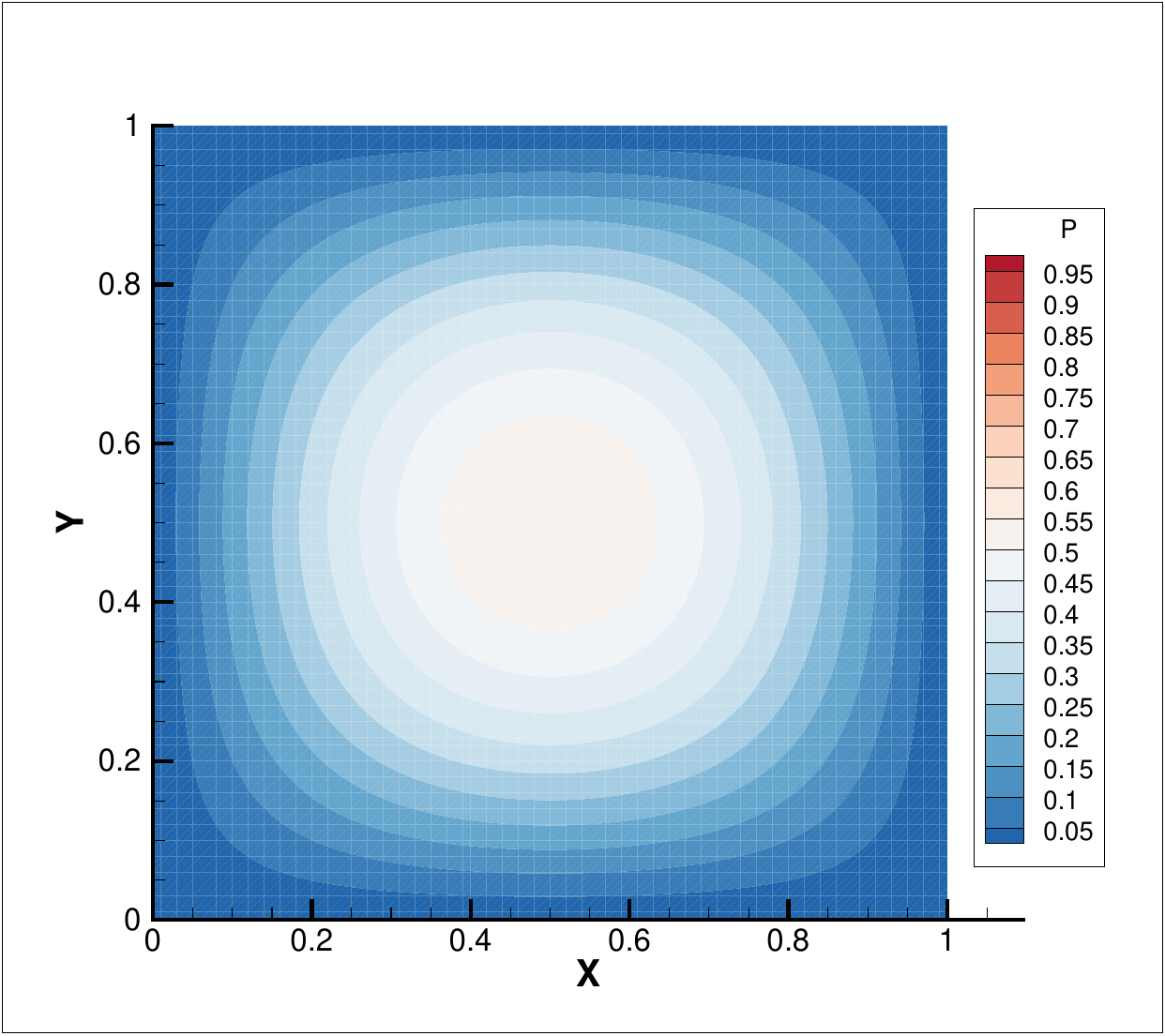}  
			\end{minipage}  
			\begin{minipage}{0.32\textwidth} 
				\centering  
				\includegraphics[width=\textwidth]{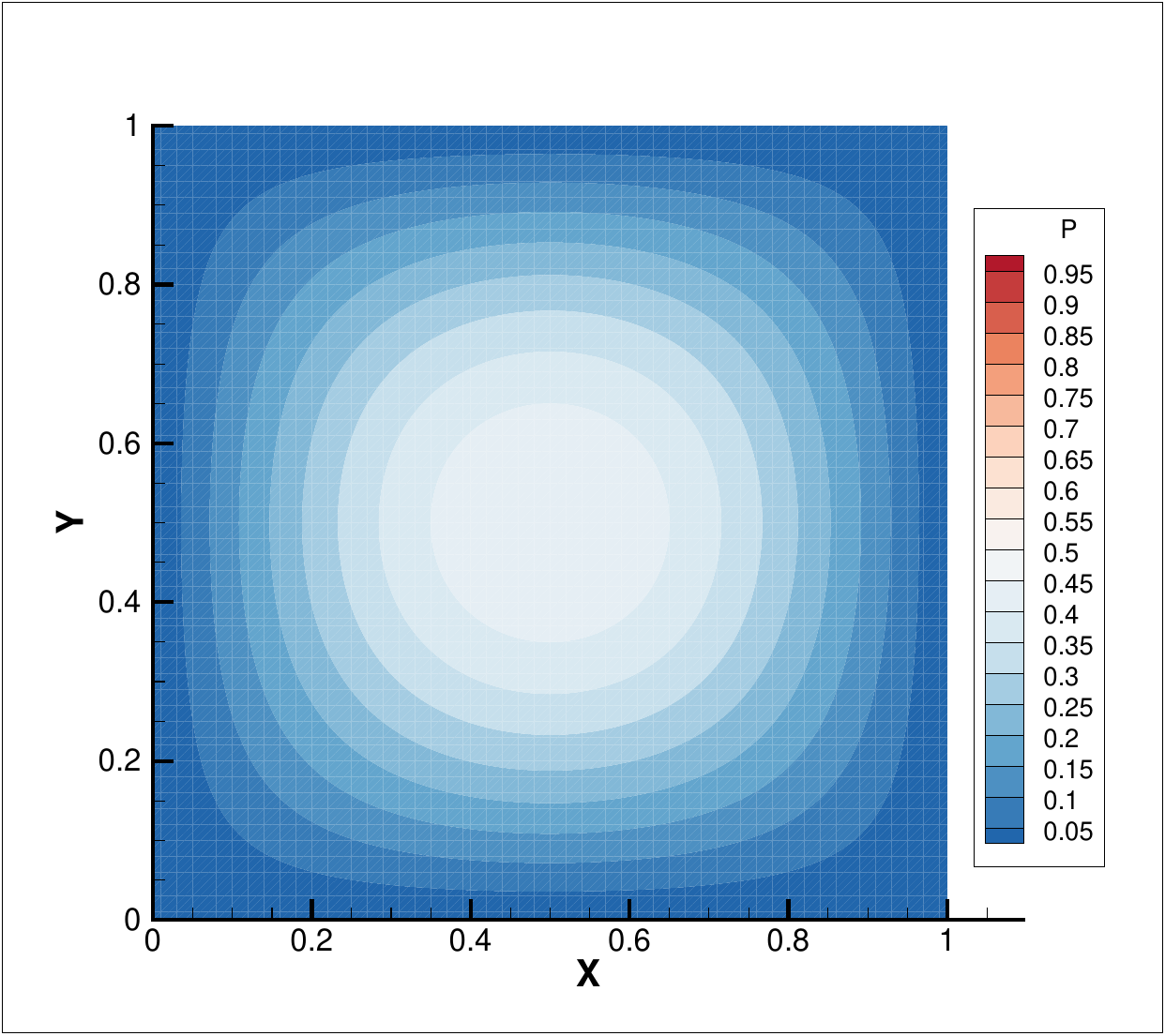} 
			\end{minipage}  
			\begin{minipage}{0.32\textwidth}  
				\centering  
				\includegraphics[width=\textwidth]{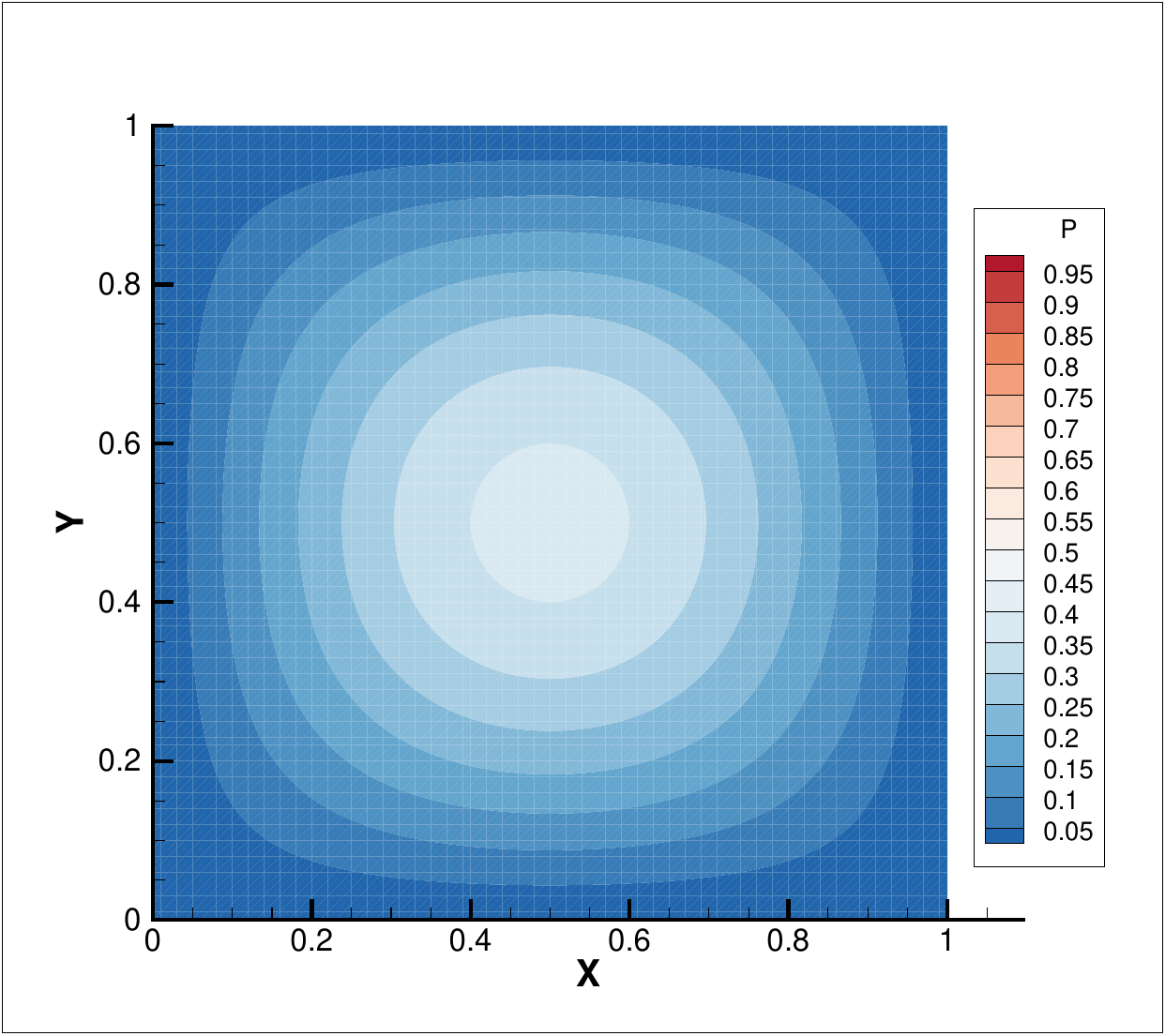}  
			\end{minipage}  
		\caption{Numerical solutions of centration at times t = 0, 0.2, 0.4, 0.6, 0.8, 1.0 with $\nu=1+0.1c$ for the decoupled finite element method.}  
		\label{decoupledcentration}  
	\end{figure}

\begin{figure}[htbp] 
	
	\centering 

		\begin{minipage}{0.32\textwidth} 
			\centering  
			\includegraphics[width=\textwidth]{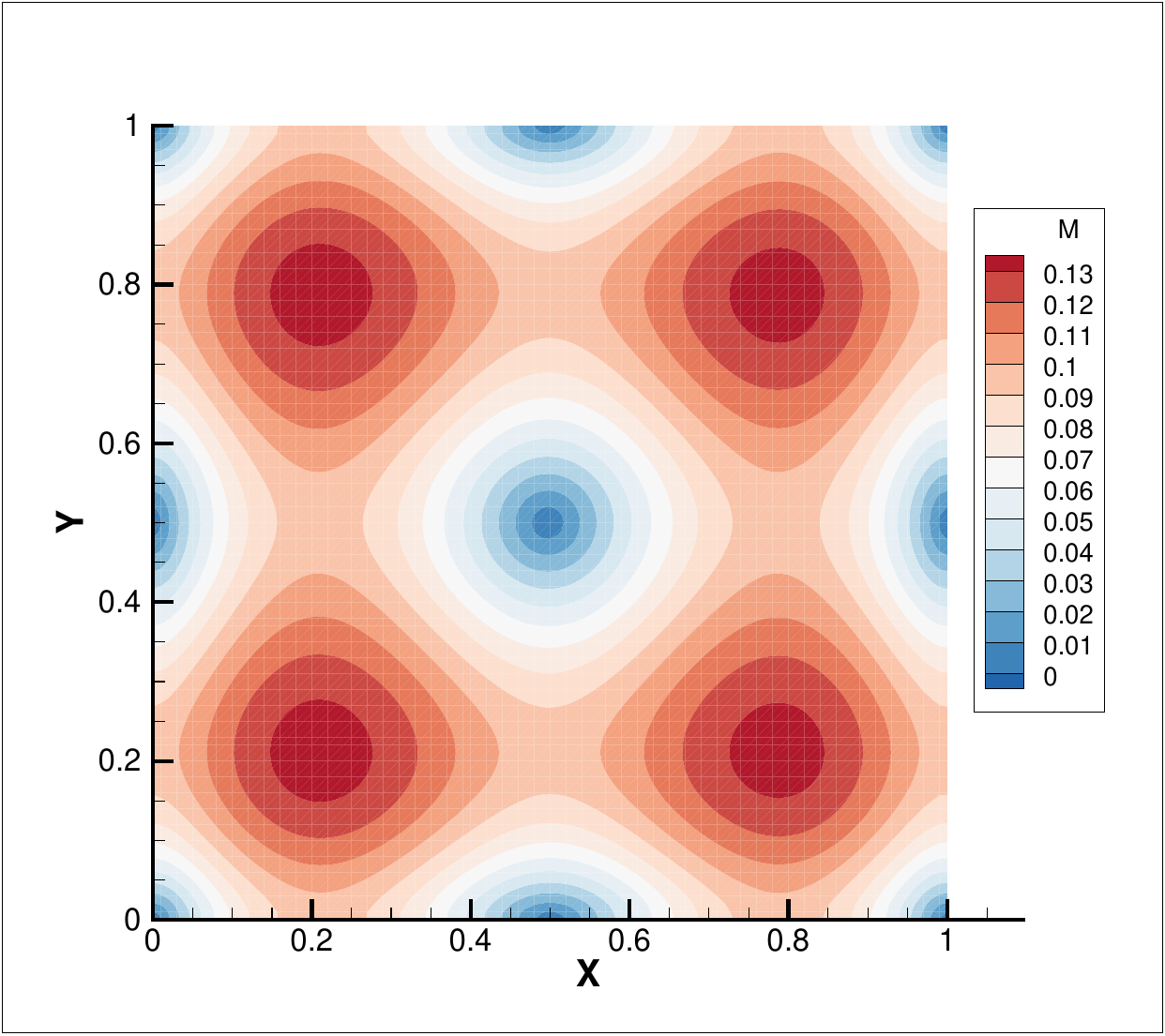}  
		\end{minipage}  
		\begin{minipage}{0.32\textwidth} 
			\centering  
			\includegraphics[width=\textwidth]{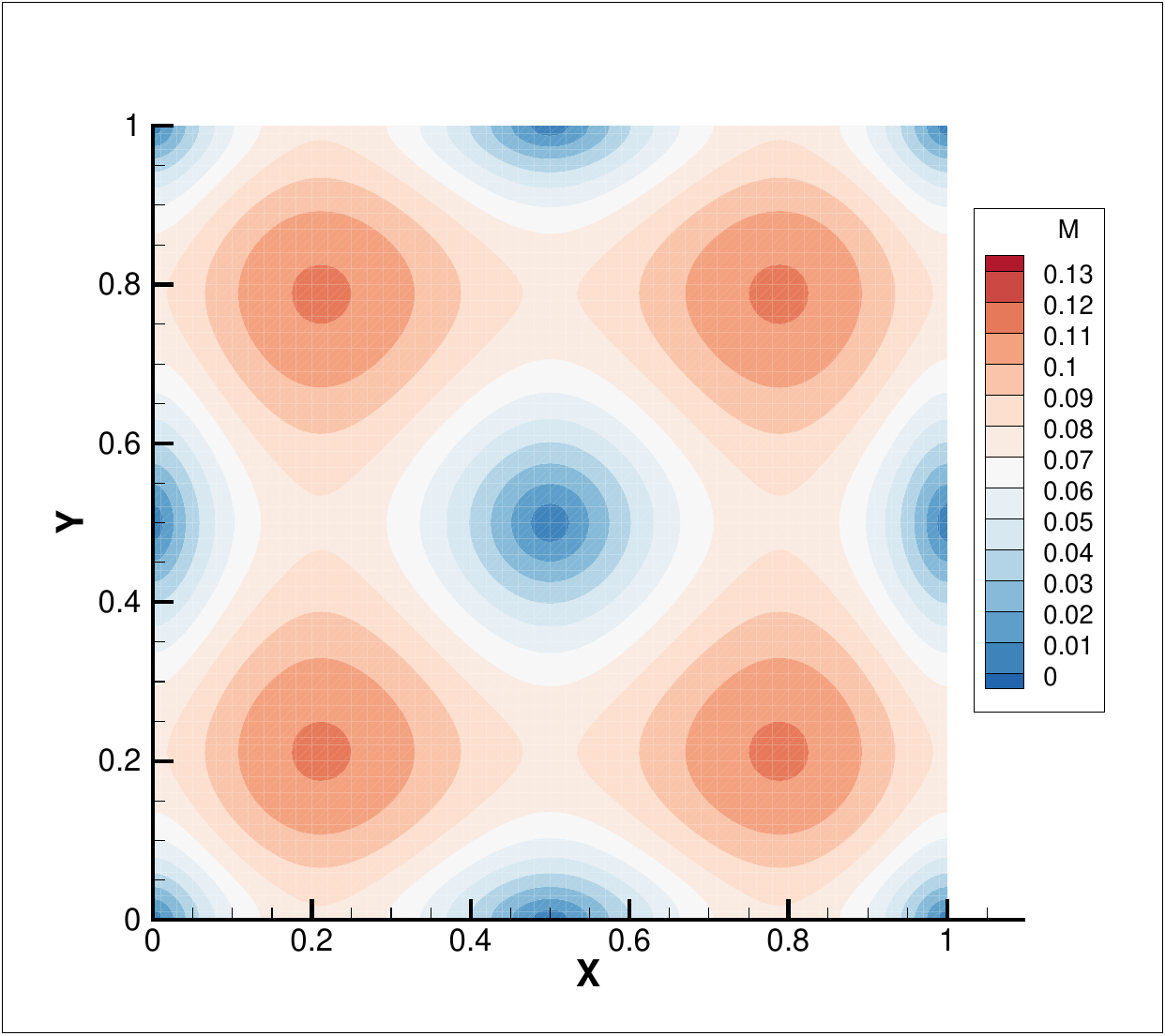} 
		\end{minipage}  
				\begin{minipage}{0.32\textwidth}  
						\centering  
						\includegraphics[width=\textwidth]{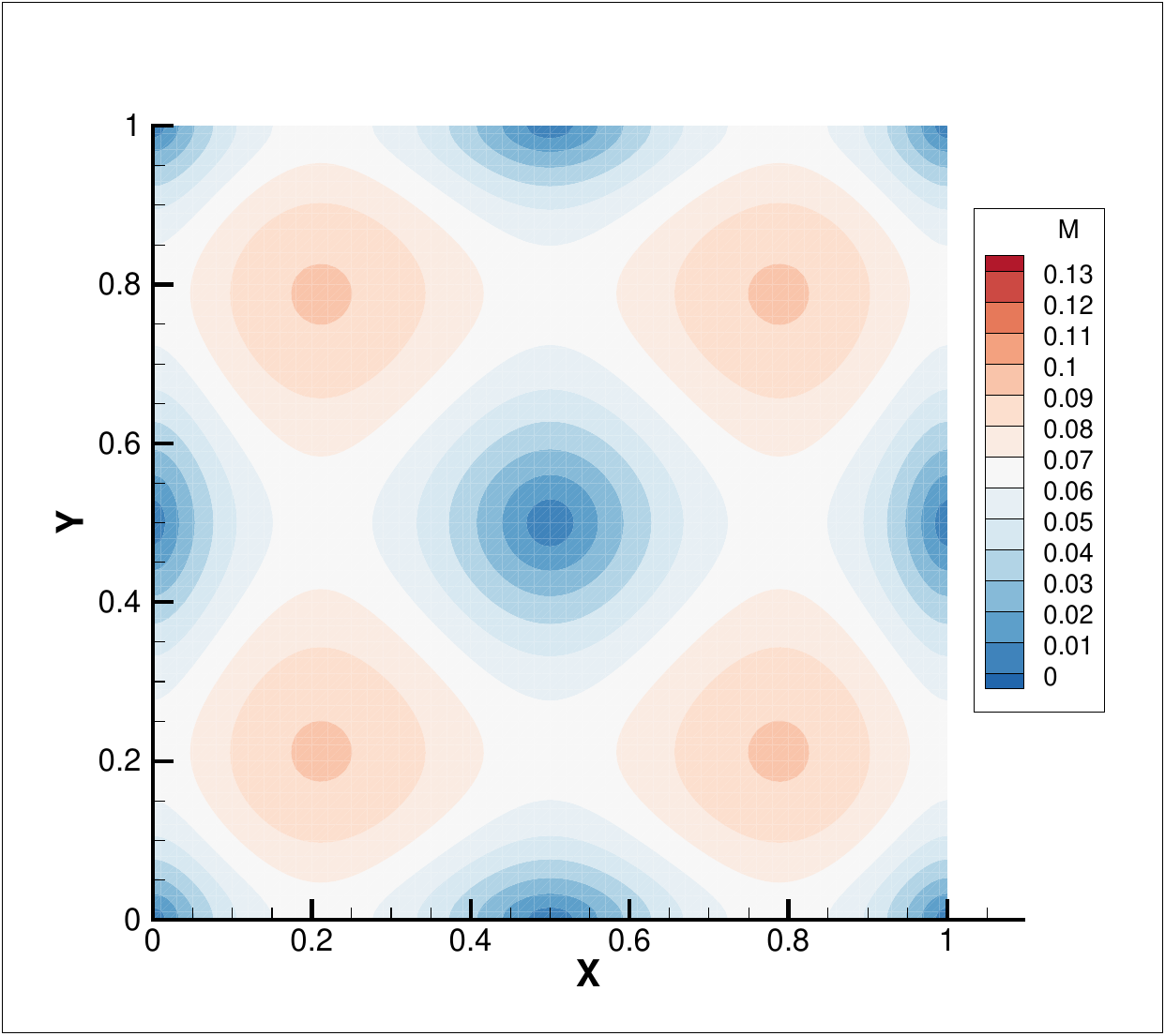}  
					\end{minipage}  
						\begin{minipage}{0.32\textwidth} 
						\centering  
						\includegraphics[width=\textwidth]{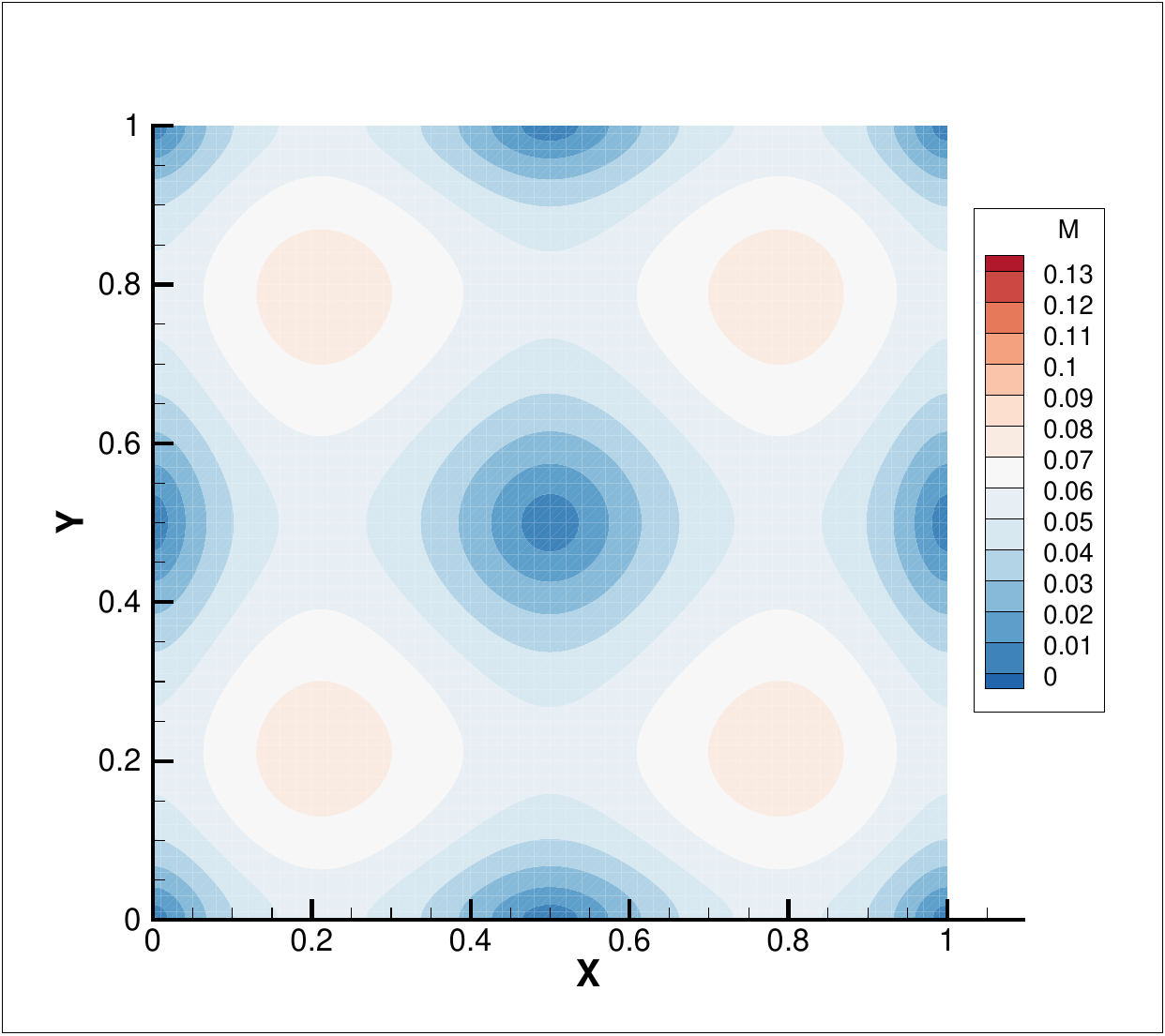}  
					\end{minipage}  
					\begin{minipage}{0.32\textwidth} 
						\centering  
						\includegraphics[width=\textwidth]{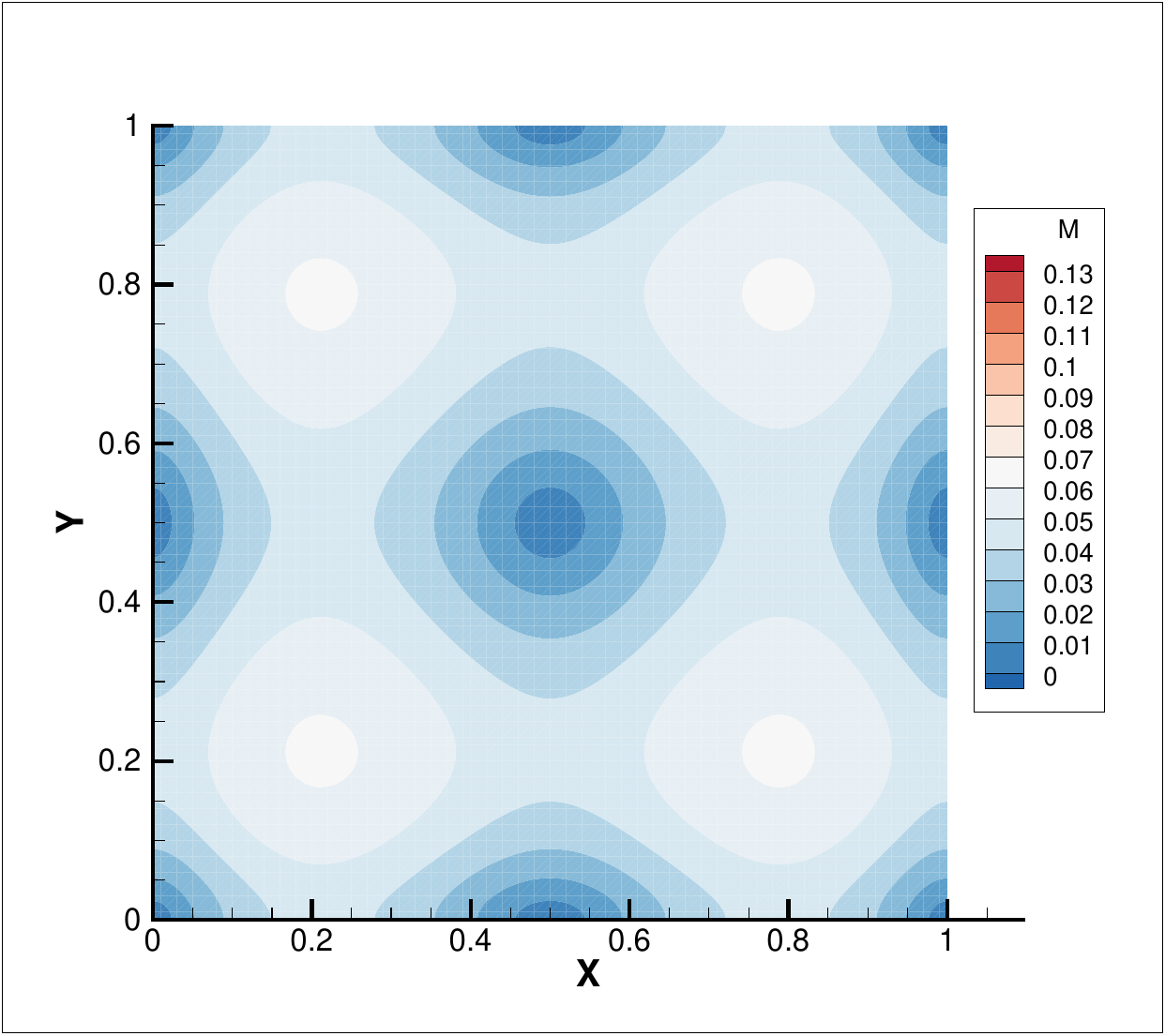} 
					\end{minipage}  
					\begin{minipage}{0.32\textwidth}  
						\centering  
						\includegraphics[width=\textwidth]{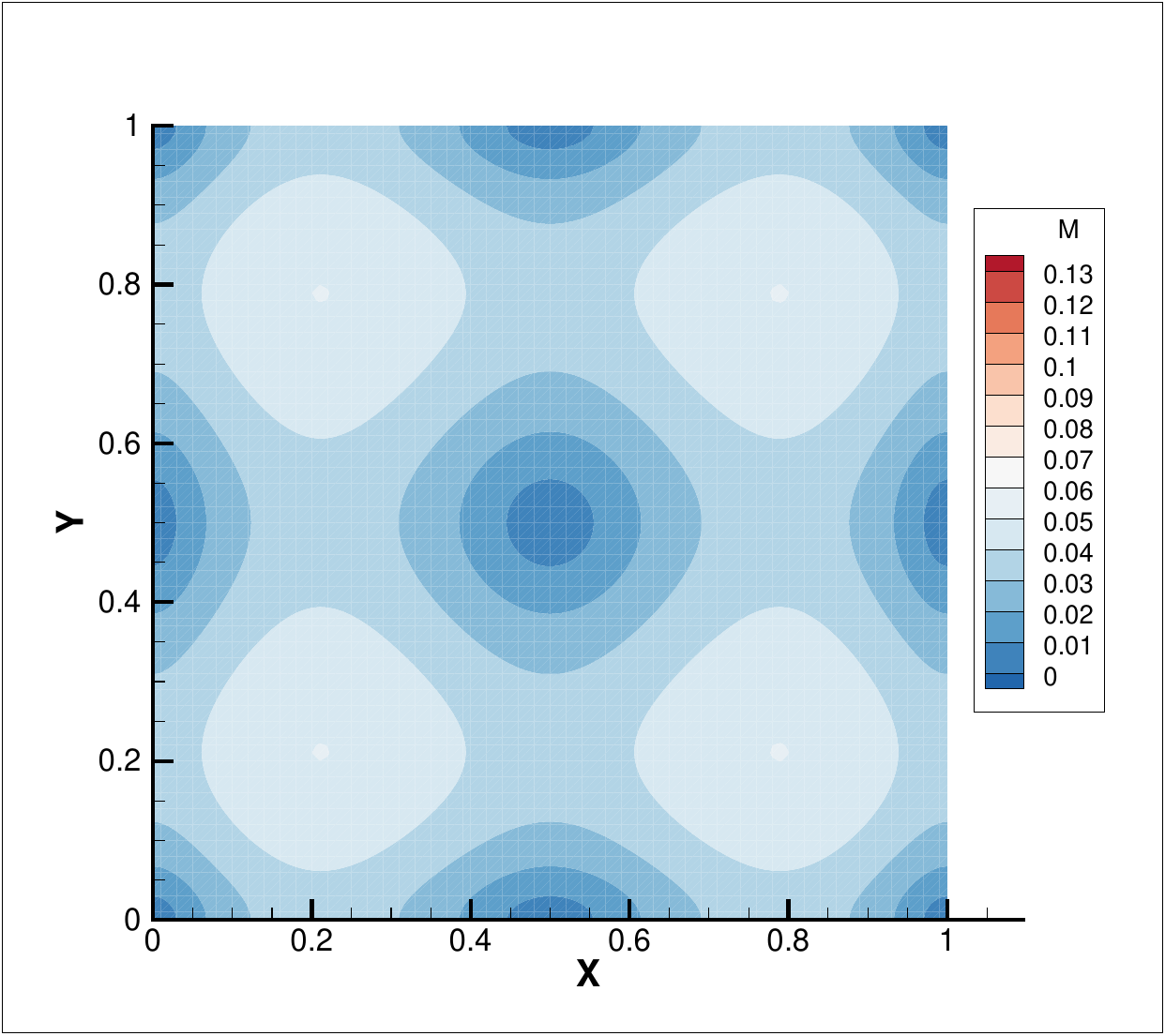}  
					\end{minipage} 
	\caption{Numerical solutions of velocity at times t = 0, 0.2, 0.4, 0.6, 0.8, 1.0 with $\nu = exp(c)$ for the coupled finite element method.}  
	\label{coupledvelocity}  
\end{figure}

\begin{figure}[htbp] 
	
	\centering 
	
		\begin{minipage}{0.32\textwidth} 
			\centering  
			\includegraphics[width=\textwidth]{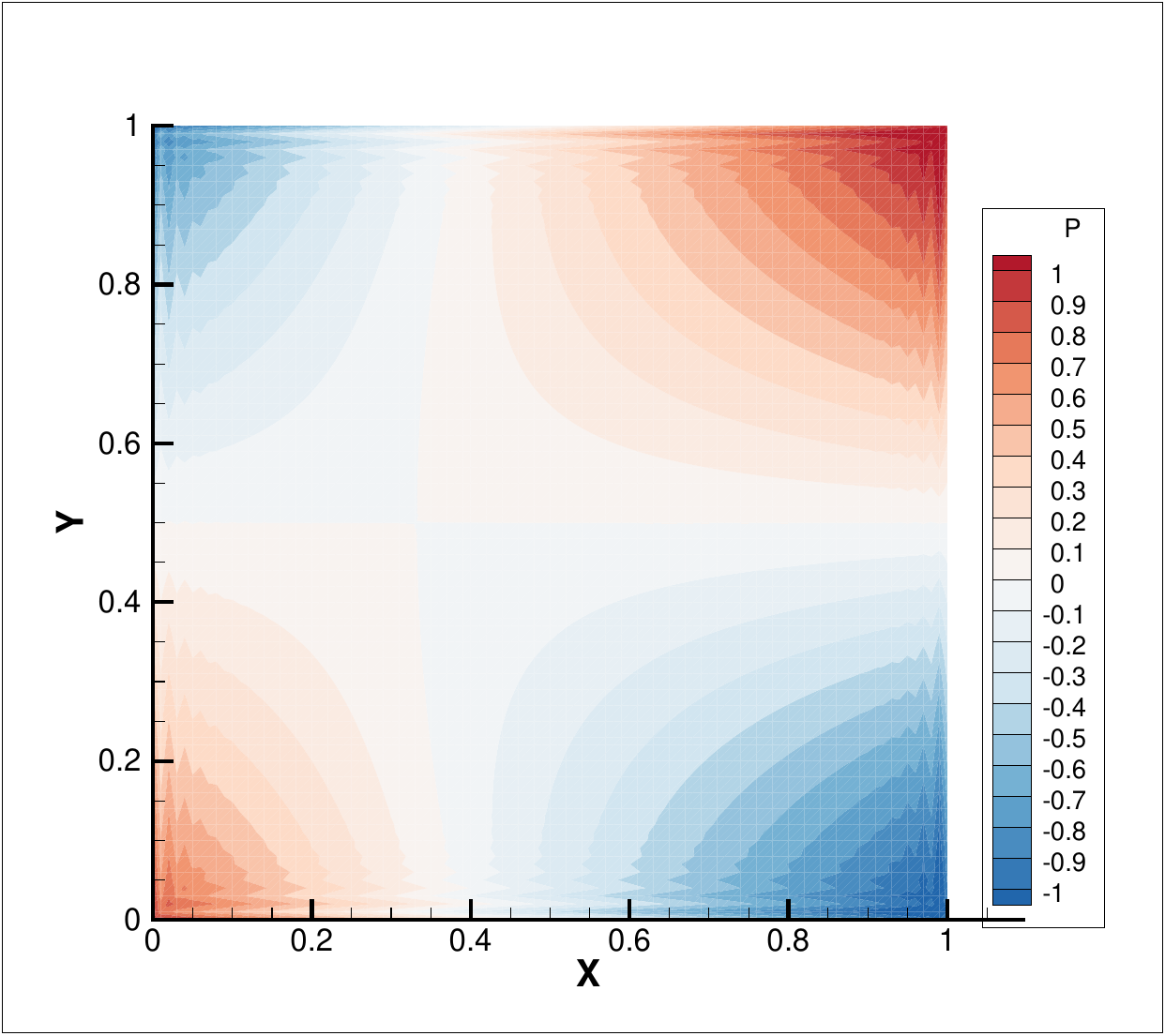}  
		\end{minipage}  
		\begin{minipage}{0.32\textwidth} 
			\centering  
			\includegraphics[width=\textwidth]{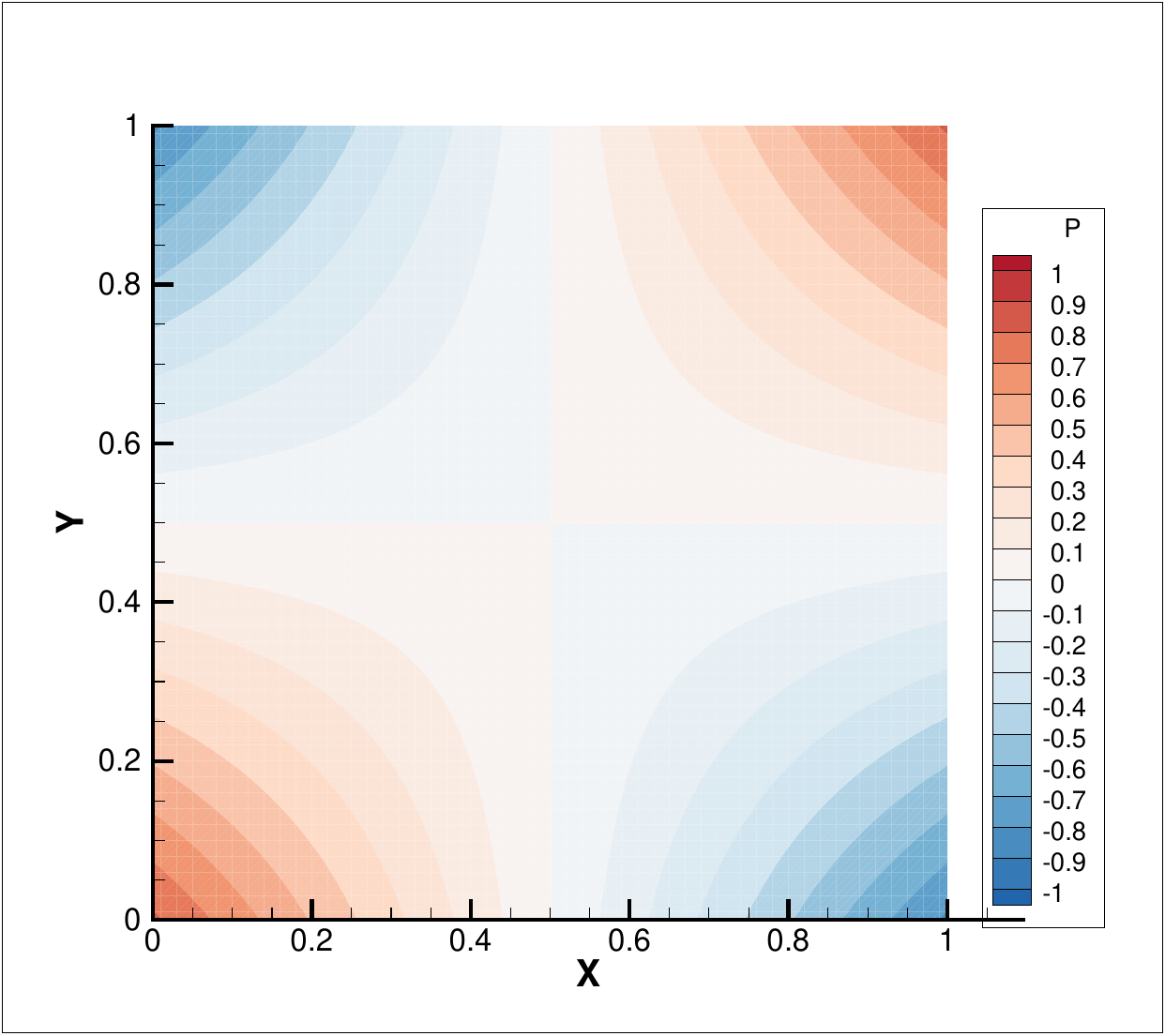} 
		\end{minipage}  
		\begin{minipage}{0.32\textwidth}  
			\centering  
			\includegraphics[width=\textwidth]{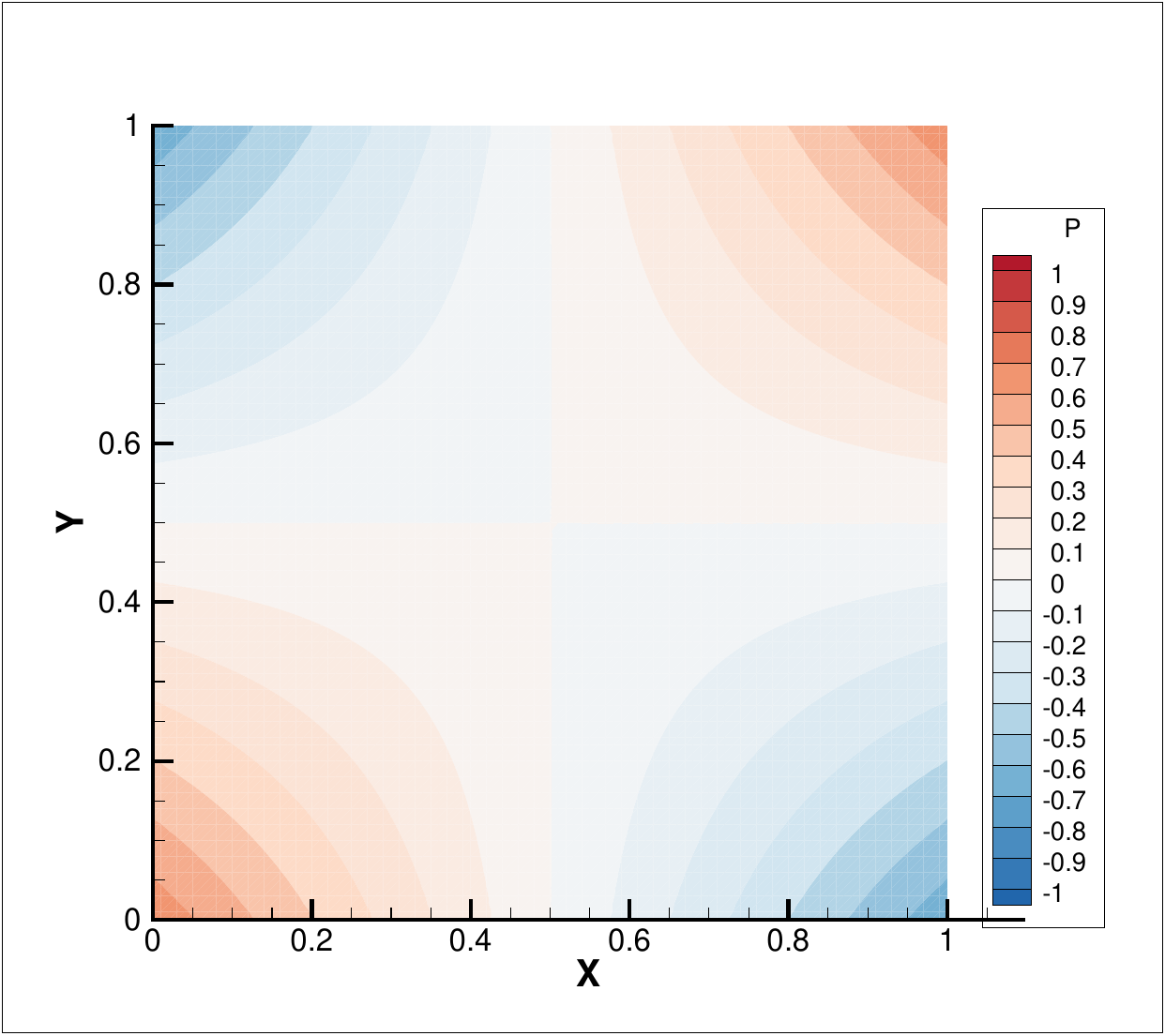}  
		\end{minipage}  	\begin{minipage}{0.32\textwidth} 
		\centering  
		\includegraphics[width=\textwidth]{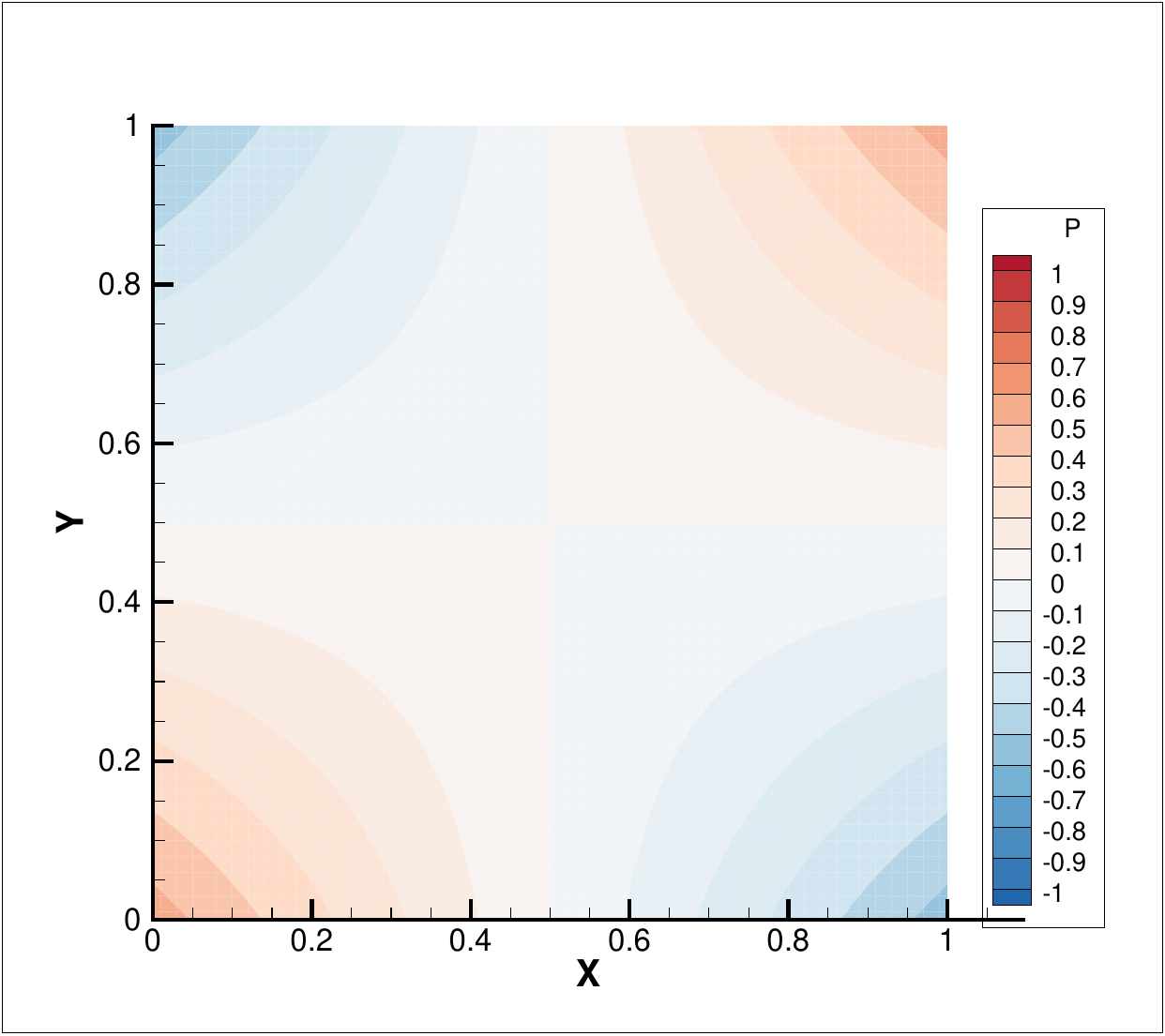}  
		\end{minipage}  
		\begin{minipage}{0.32\textwidth} 
		\centering  
		\includegraphics[width=\textwidth]{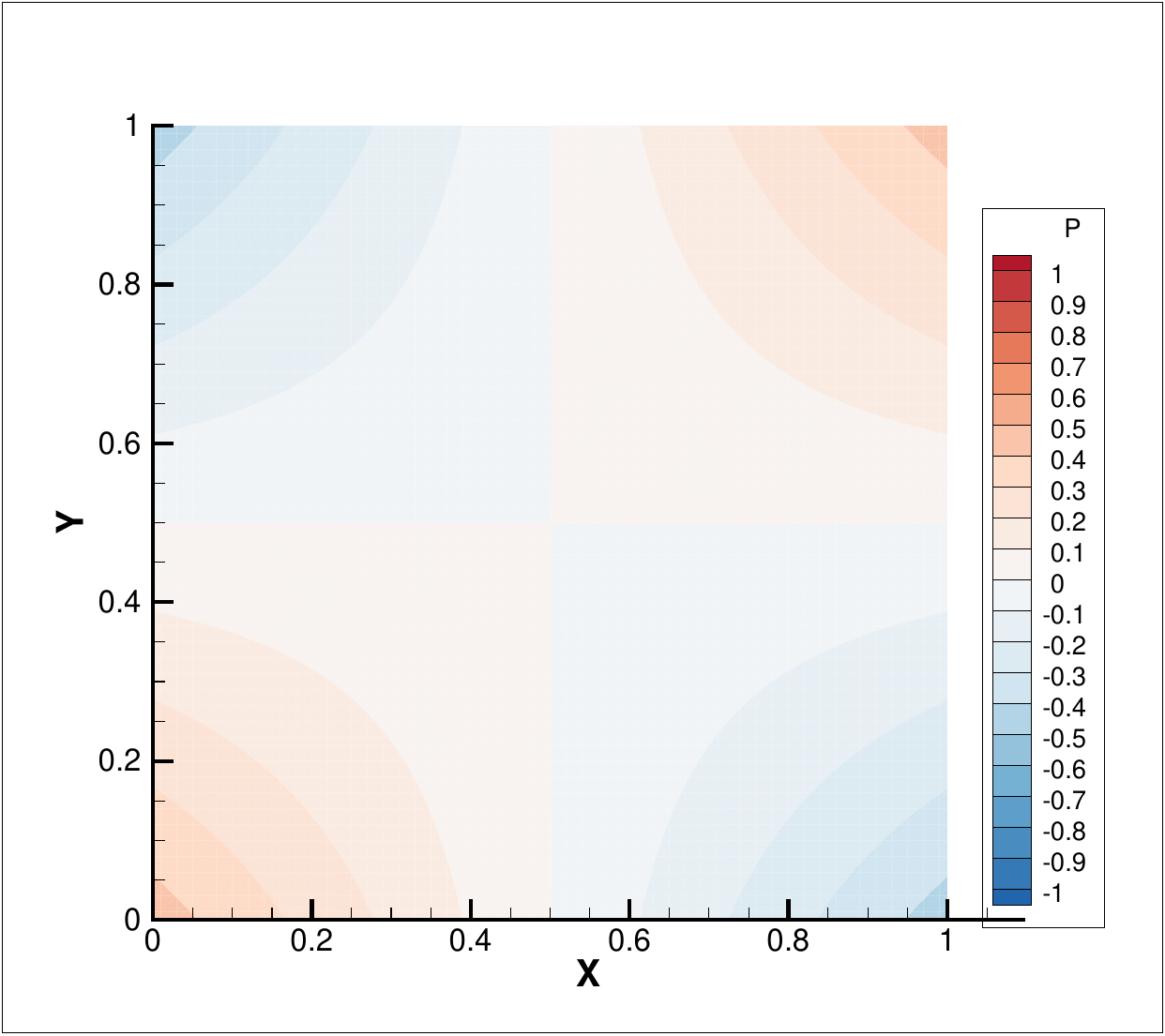} 
		\end{minipage}  
		\begin{minipage}{0.32\textwidth}  
		\centering  
		\includegraphics[width=\textwidth]{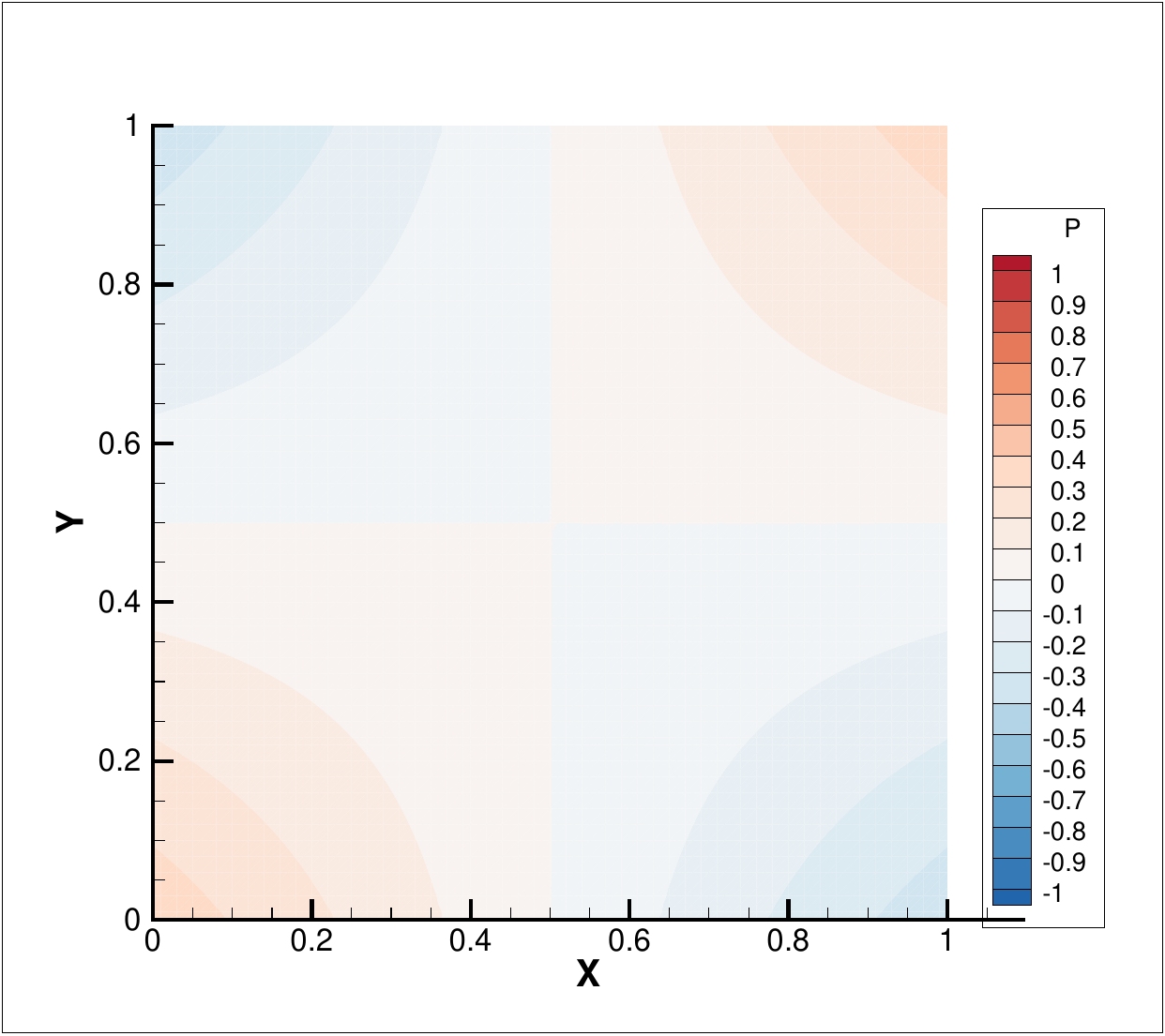}  
		\end{minipage}  
	\caption{Numerical solutions of pressure at times t = 0, 0.2, 0.4, 0.6, 0.8, 1.0 with $\nu =  exp(c)$ for the coupled finite element method.}  
	\label{coupledpressure}  
\end{figure}

\begin{figure}[htbp] 
	
	\centering 

		\begin{minipage}{0.32\textwidth} 
			\centering  
			\includegraphics[width=\textwidth]{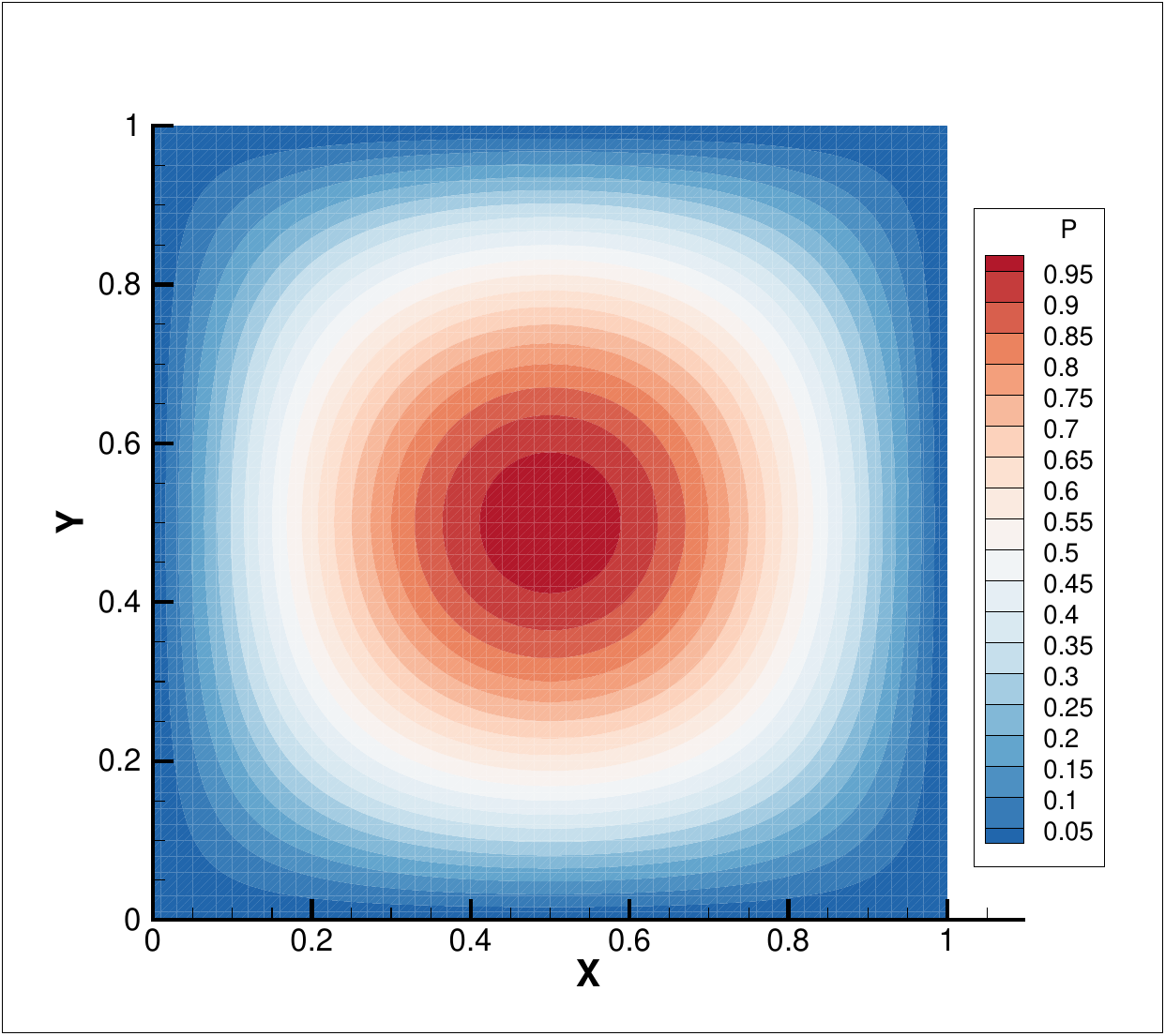}  
		\end{minipage}  
		\begin{minipage}{0.32\textwidth} 
			\centering  
			\includegraphics[width=\textwidth]{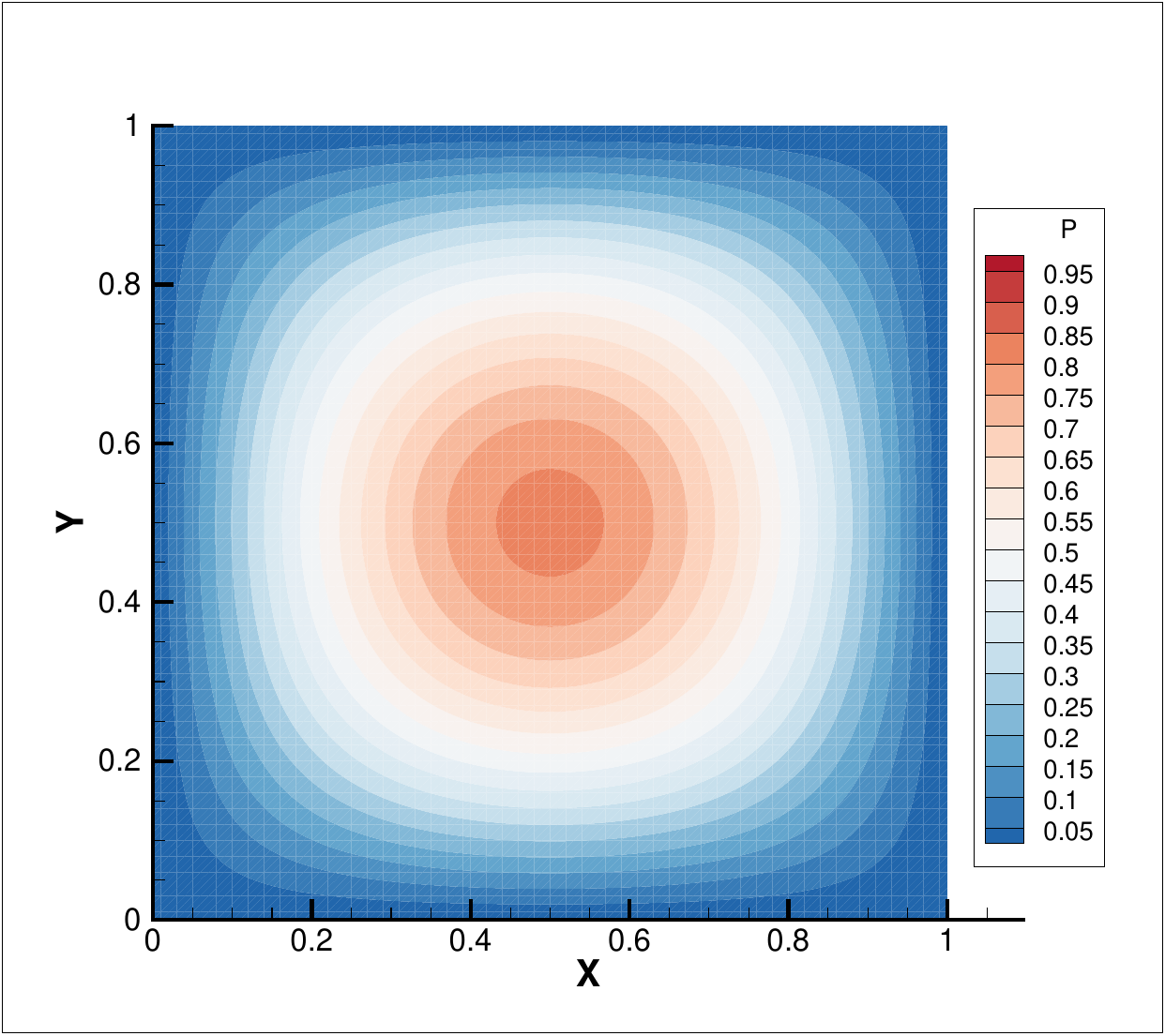} 
		\end{minipage}  
		\begin{minipage}{0.32\textwidth}  
			\centering  
			\includegraphics[width=\textwidth]{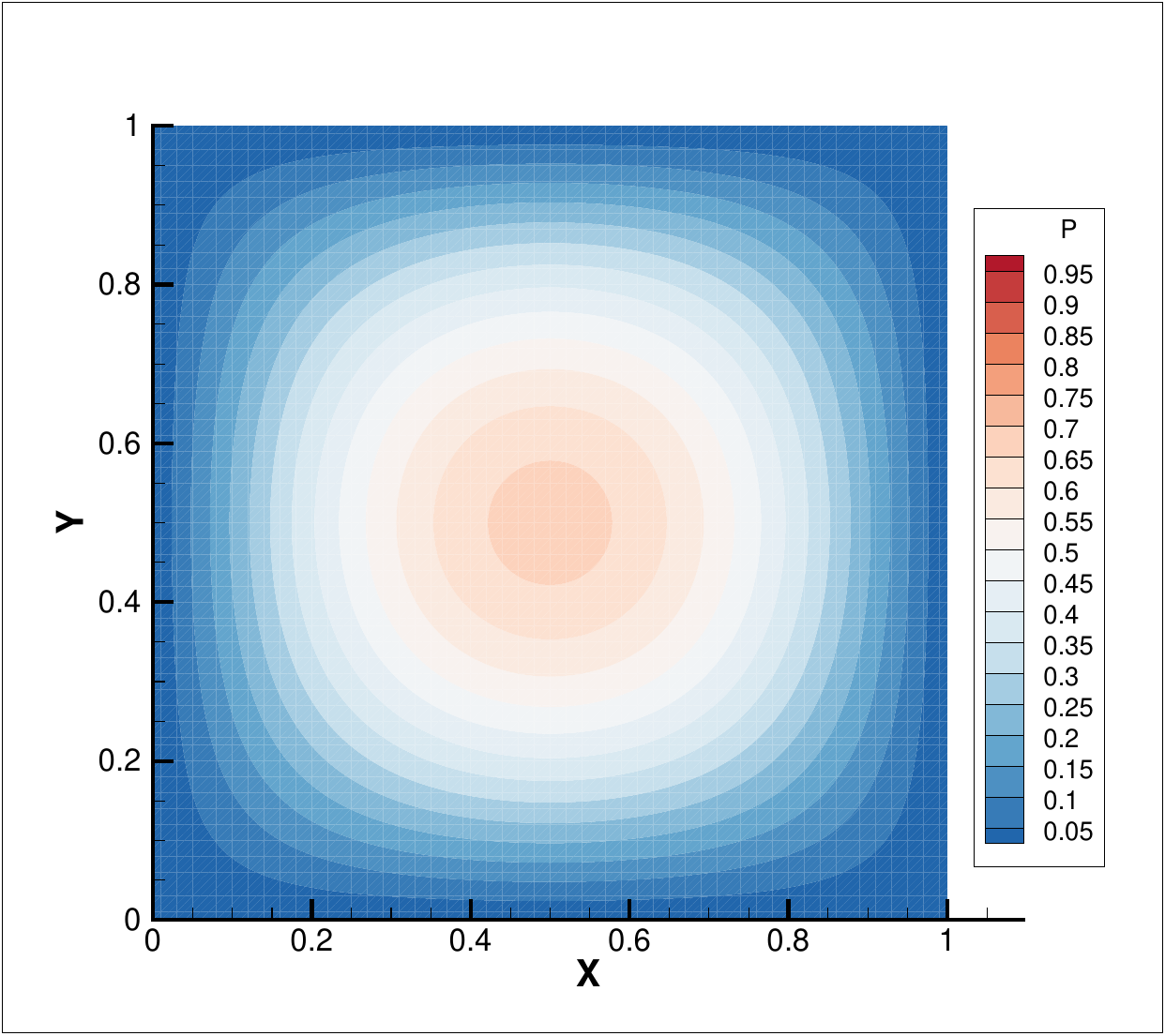}  
		\end{minipage}  

\begin{minipage}{0.32\textwidth} 
	\centering  
	\includegraphics[width=\textwidth]{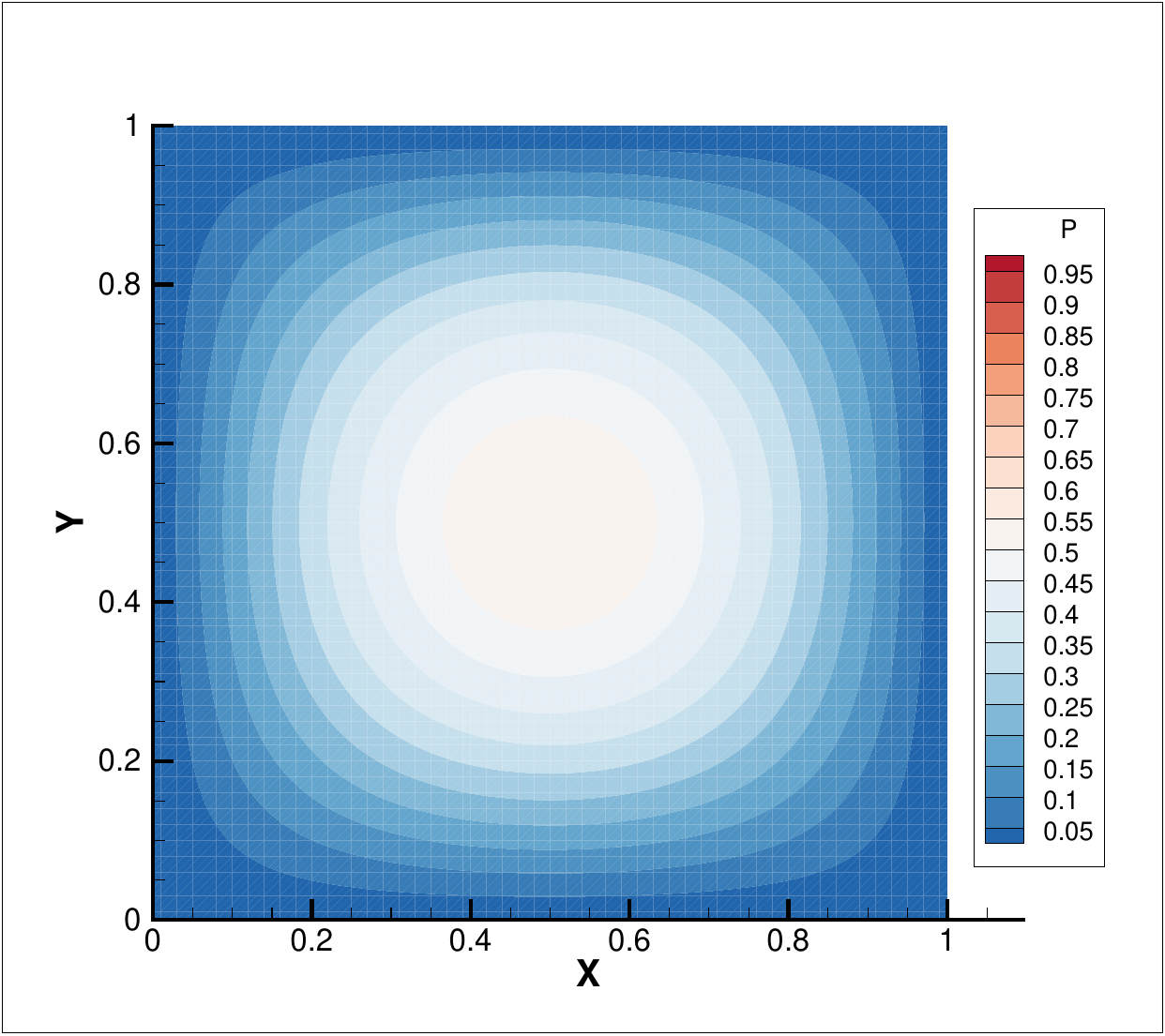}  
\end{minipage}  
\begin{minipage}{0.32\textwidth} 
	\centering  
	\includegraphics[width=\textwidth]{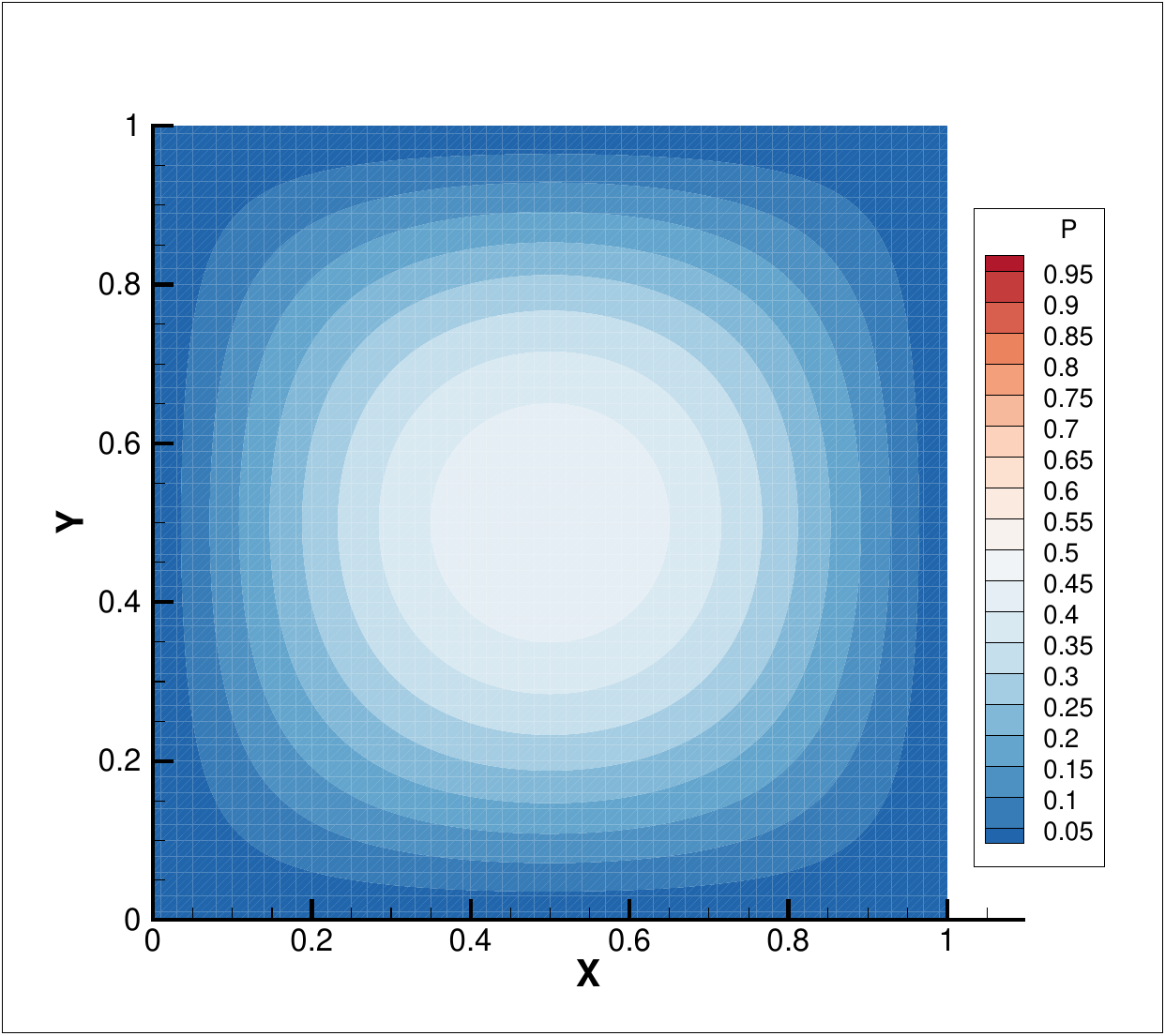} 
\end{minipage}  
\begin{minipage}{0.32\textwidth}  
	\centering  
	\includegraphics[width=\textwidth]{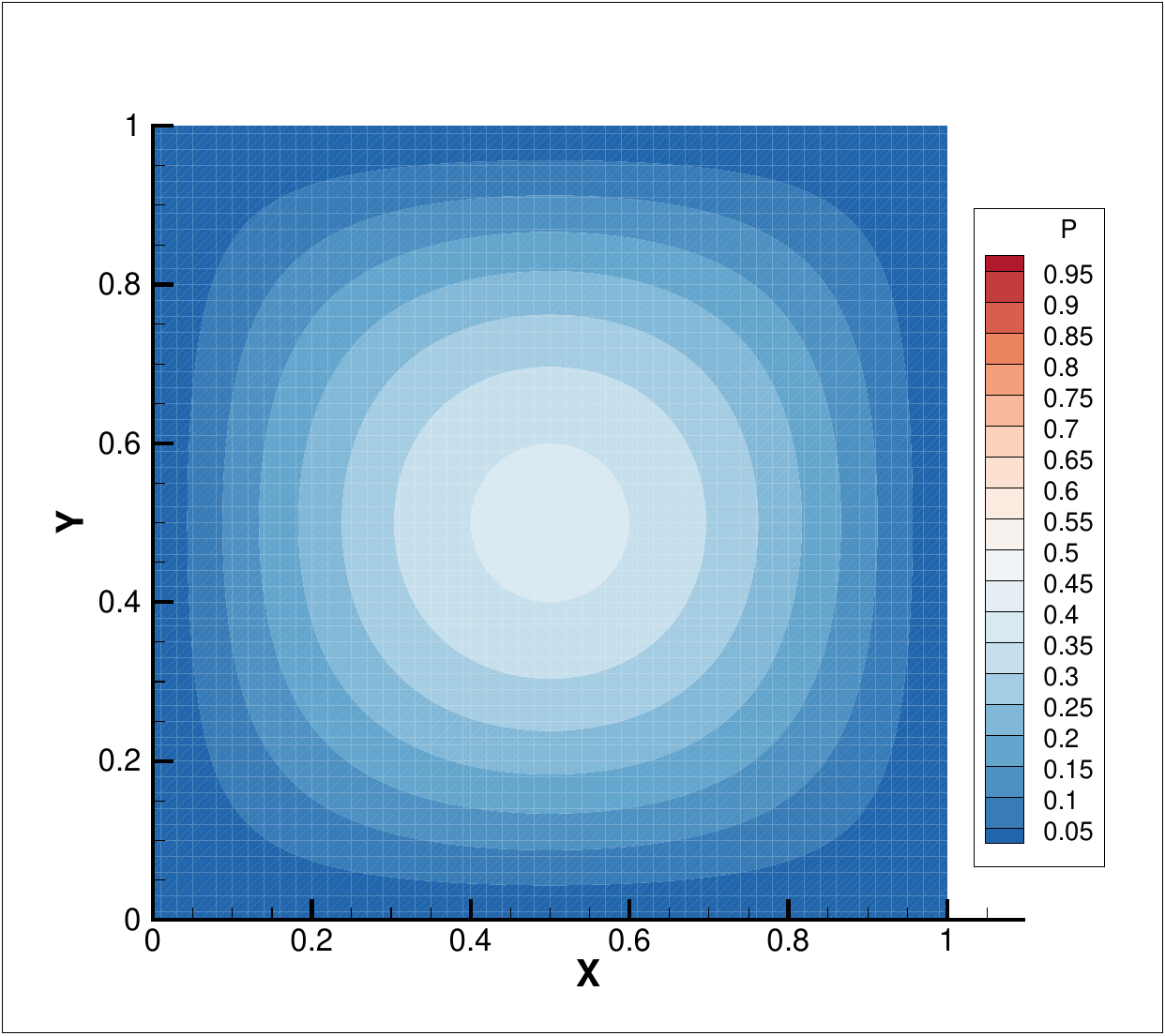}  
\end{minipage}  
	\caption{Numerical solutions of centration at times t = 0, 0.2, 0.4, 0.6, 0.8, 1.0 with $\nu = exp(c)$ for the coupled finite element method.}  
	\label{coupledcentration}  
\end{figure}

\section{Conclusion}
In this paper, we proposed the coupled and decoupled BDF2 finite element methods for the time-dependent bioconvection flows problem, which is coupled by the Navier-Stokes equations and the advection-diffusion equation. We have found optimal error estimate for the velocity and concentration in $L^2$-norm and $H^1$-norm in a bounded domain by mixed finite element methods. We present some numerical results by varying different concentration dependent viscosity, we can see that the proposed coupled and decoupled BDF2 finite element methods are effective and valid. It's important to research the unsteady bioconvection for the study of Chemotaxis-Navier–Stokes system, Patlak-Keller-Segel-Navier-Stokes system and Chemo-Repulsion-Navier-Stokes system.

\section*{Acknowledgments}
The authors thank the anonymous referees very much for their helpful comments and
suggestions which helped to improve the presentation of this paper.
\section*{Date availability statement}
The datasets generated and  analyzed during the current study are available from the corresponding author upon reasonable request.

\end{document}